\documentclass[a4paper, 11pt]{article}
\usepackage{bbm}%Êýѧ»·¾³ÓÃbbm×ÖÌå
\usepackage{amsmath,amssymb,amsthm,amsfonts}

\usepackage{graphicx}%²åͼÃüÁîºê°ü£¬Í¬Ê±¿ÉÒÔÔÚ±í¸ñÃüÁîÖа´±ÈÀýÀ©´ó±í¸ñ
\usepackage{epstopdf}
\usepackage{latexsym}
\usepackage{pdfsync}

\usepackage{multirow}%ºÏ²¢Ðеĺê°ü
\usepackage[font=bf,aboveskip=15pt]{caption}% ±í¸ñ±êÌâÏà¹Ø²ÎÊýµÄÉ趨
\usepackage[toc,page]{appendix}%¸½Â¼µÄÏà¹Øºê°ü
\usepackage[colorlinks=true]{hyperref}%ÓÃÁËÕâ¸öºê°üÖ®ºó£¬»á·¢ÏÖËùÓеIJο¼ÎÄÏ׺ͽ»²æÒýÓö¼¾ßÓÐÁ˳¬Á´½Ó¹¦ÄÜ.

\hypersetup{dvipdfmx,
       colorlinks,
       linkcolor=blue,
       filecolor=red,
       urlcolor=blue,
       citecolor=red,
       pdftitle={Vortex structures},
       pdfauthor={ Jun YANG},
       pdfsubject={   PDEs         },
       pdfkeywords={   Geometric flows},
       pdfproducer={ps2pdf}}
\makeatletter
\def\@currentlabel{2.1}\label{e:dispaa}
\def\@currentlabel{2.21}\label{e:dispau}
\def\@currentlabel{2.22}\label{e:dispav}
\def\@currentlabel{2.23}\label{e:dispaw}
\def\@currentlabel{2.24}\label{e:dispax}
\def\theequation{\thesection.\@arabic\c@equation}

\makeatother
\def\theequation{\thesection.\@arabic\c@equation}
\makeatother

\renewcommand{\theequation}{\thesection.\arabic{equation}}
\newtheorem{theorem}{Theorem}[section]

\newtheorem{proposition}[theorem]{Proposition}

\newtheorem{lemma}[theorem]{Lemma}

\allowdisplaybreaks

\def\e{\epsilon}

\def\m{\int_{\Bbb R}}
\def\e{\epsilon}

\def\G{{\mathfrak G}}
\def\g{\mathfrak g}

\def\R{{\mathfrak R}}

\def\begeq{\begin{equation}}
\def\endeq{\end{equation}}

\def\lf{\left}
\def\ri{\right}
\def\e{\epsilon}

\def\R{\Bbb R}

\def\e{\epsilon}

\def \i{\int_{\Bbb R}}

\def\G{\Gamma}\def\g{\gamma}
\def \t{\tilde}

\begin{document}
\title{On Ambrosetti-Malchiodi-Ni conjecture
 on  two-dimensional smooth bounded domains}%[Concentration on curves for an inhomogeneous problem]

\author{
Suting Wei
\\
School of Mathematics and Statistics, Central China Normal University,
\\
Wuhan 430079, P. R. China.
\quad
Email: stwei@mails.ccnu.edu.cn
\\
\\
Bin  Xu
 \\
School of Mathematics and  Statistics, Jiangsu  Normal University,
 \\
 Xuzhou, Jiangsu, 221116, P. R. China.
\quad
Email: dream-010@163.com
\\
\\
 Jun  Yang
\footnote{Corresponding author: Jun Yang, jyang@mail.ccnu.edu.cn}
\\
School of Mathematics and Statistics \& Hubei Key Laboratory of Mathematical
\\
Sciences, Central China Normal University,
\\
Wuhan 430079, P. R. China.
\quad
Email: jyang@mail.ccnu.edu.cn
}

\date{}
\maketitle

\begin{abstract}

\noindent  We consider the problem
 $$
 \epsilon^2 \Delta u-V(y)u+u^p\,=\,0,~~u>0~~\quad{\rm in}\quad\Omega,~~\quad\frac {\partial u}{\partial \nu}\,=\,0\quad{\rm on}~~~\partial \Omega,
 $$
 where $\Omega$~is a bounded domain in $\Bbb R^2$ with smooth boundary, the exponent $p>1$, $\epsilon>0$ is a small parameter, $V$ is a uniformly positive, smooth
 potential on $\bar{\Omega}$,  and $\nu$ denotes the outward normal of $\partial \Omega$.
 Let $\Gamma$ be a curve intersecting orthogonally with
 $\partial \Omega$ at exactly two points and dividing $\Omega$~into two parts. Moreover, $\Gamma$~satisfies {\it stationary and non-degeneracy conditions}~with respect to the functional $\int_{\Gamma}V^{\sigma}$, where $\sigma=\frac {p+1}{p-1}-\frac 12$. We prove the existence
 of a solution $u_\epsilon$~concentrating along the whole of $\Gamma$,  exponentially small in $\epsilon$~at any positive distance from it, provided that $\epsilon$ is small
 and away from certain critical numbers.
 In particular, this establishes the validity of the two dimensional case of a conjecture by A. Ambrosetti,\;A. Malchiodi and W.-M. Ni(p.327, \cite{AMNconjecture2}).
 \end{abstract}



 \section{Introduction}
 \label{section1}

 We consider the following problem for the existence of concentration phenomena
 \begin{equation}
 \label{original equation-01}
 \epsilon^2\Delta u-V(y)u+u^p\,=\,0,\;\;\;u>0\quad{\rm in}\quad \Omega,
\qquad
\frac {\partial u}{\partial \nu}\,=\,0\quad{\rm on}\quad\partial \Omega,
 \end{equation}
where $\Omega$ is a bounded domain in $\Bbb R^{\mathbf d}$ with smooth boundary, $\epsilon>0$ is a small parameter,~$\nu$~denotes the outward normal of $\partial \Omega$ and
 the exponent $p>1$.

\medskip
\subsection{Case 1: $V\equiv 1$}
If $V\equiv 1$, problem (\ref{original equation-01}) takes the form
\begin{equation}
 \label{original equationconstant}
 \epsilon^2\Delta u-u+u^p\,=\,0,\;\;\;u>0~\quad{\rm in}\quad \Omega, ~~~~~\frac {\partial u}{\partial \nu}\,=\,0~~~\;\;{\rm on}~\;\;~\partial \Omega,
 \end{equation}
which is known as the stationary equation of Keller-Segel system in chemotaxis \cite{LNT}. It can also be viewed as a
 limiting stationary equation of Gierer-Meinhardt system in biological pattern formation \cite{GM}. Problem (\ref{original equationconstant}) has been
 studied extensively in recent years. See \cite{Ni2} for backgrounds and references.

 \medskip
 In the pioneering papers \cite{LNT}, \cite{NT}-\cite{NT1}, under the condition that $p$ is subcritical, i.e.,
 $$
 1<p<\frac {{\mathbf d}+2}{{\mathbf d}-2}
\quad\mbox{when}\quad {\mathbf d}\geq 3
\quad\mbox{and}\quad
1<p<+\infty
\quad\mbox{when}\quad {\mathbf d}=2,
$$
Lin, Ni and Takagi established,  for $\epsilon$ sufficiently small, the existence of a least-energy
 solution $U_\epsilon$ of (\ref{original equationconstant}) with only one local maximum point $P_\epsilon\in \partial \Omega$.
 Moreover,
 $$
 H(P_\epsilon)\rightarrow \displaystyle {\max_{P\in {\partial \Omega}}}H(P)
 \quad\mbox{as}\quad
 \epsilon\rightarrow 0,
 $$
 where $H(\cdot)$ is the mean curvature
 of $\partial \Omega$.  Such a solution is called a boundary spike-layer.

\medskip
Since then, many papers investigated further the solutions of (\ref{original equationconstant}) concentrating at one or multiple points of $\bar{\Omega}$.
These solutions are called spike-layers.
A general principle is that the location of interior spikes is determined by the distance function from the boundary.
 We refer the reader to the articles
 \cite{BF},
 \cite{dy2},
 \cite{dPFW1},
 \cite{GPW}-\cite{GW},
 \cite{Wei2},
 and the references therein.
 On the other hand, boundary spikes are related to the mean curvature
 of $\partial \Omega$.
 This aspect was discussed in the papers
 \cite{BDS},
 \cite{DY},
 \cite{dPFW},
 \cite{GWW},
 \cite{Li},
 \cite{Wei1},
 \cite{WeiWinter},
 and the references therein.
 The coexistence of interior and boundary spikes was due to Gui and Wei \cite{GW1}.
%For any given pair of nonnegative integers
 %$k$ and $l$,  problem (\ref{original equationconstant}) has a solution concentrating on $k$ interior
 % points and $l$ boundary points.
 A good review of the subject up to 2004 can be found in \cite{Ni2}.

\medskip
 The question of constructing higher-dimensional concentration sets has been investigated only in recent years.
 It has been conjectured in \cite{Ni1} (see also \cite{Ni2}) by W.-M. Ni that:
 {\it for any
 $1\leq \mathbf{k}\leq {\mathbf d}-1$,  problem (\ref{original equationconstant}) has a solution $U_\epsilon$ which concentrates on a $\mathbf{k}$-dimensional subset of $\bar{\Omega}$.}

 \medskip
 We mention some results for the existence of higher dimensional boundary concentration phenomena.
 In \cite{MalMon} and \cite{MalMon1}, Malchiodi and Montenegro considered (\ref{original equationconstant}) and made an initial and successful attempt to construct a solution $U_\epsilon$ with concentration at the boundary $\partial \Omega$ (or any component of $\partial \Omega$).
 In \cite{Mal} and  \cite{Mal1}, Malchiodi showed the concentration phenomena for (\ref{original equationconstant}) along a closed non-degenerate geodesic of
 $\partial \Omega$ in the three dimensional smooth bounded domain $\Omega$.
 Mahmoudi and Malchiodi  \cite{MM} proved a full general concentration of solutions along $\mathbf{k}$-dimensional ($1\leq \mathbf{k}\leq {\mathbf d}-1$) non-degenerate minimal submanifolds of the boundary for ${\mathbf d}\geq 3$ and $1<p<({\mathbf d}-\mathbf{k}+2)/({\mathbf d}-\mathbf{k}-2)$.

\medskip
On the other hand,
in order to verify the two dimensional case of the conjecture, Wei and Yang \cite{wei-yang, weiyang2} considered problem (\ref{original equationconstant}) with solutions concentrating on curves near a non-degenerate line $\Gamma'$ connecting the boundary of $\Omega$ at right angle.
%The meaning of non-degeneracy of the line $\Gamma'$ can be defined similarly as the non-degeneracy of the hypersurface $\Gamma$ in the sequel.
There are also some other results \cite{AoMussoWeiJDE, AoMussoWeiSIAM, DT, DTPA-1, DTPA-2} to exhibit concentration phenomena on interior line segments connecting the boundary of $\Omega$.
For higher dimensional extension, the reader can refer to \cite{AoYang,DancerYan,Guoyang,ligongbao}. In \cite{DancerYan}, Dancer and Yan constructed solutions of (\ref{original equationconstant}) concentrating on higher dimensional subsets inside the domain and on the boundary of the domain separately.
In \cite{ligongbao} Li, Peng and Yan showed the existence of solutions concentrating simultaneously on several higher dimensional interior and boundary manifolds.
The reader can also refer to the survey paper by Wei\;\cite{Wei}.

\medskip
\subsection{Case 2: $V{\not\equiv}$ constant}

In this case, we first concern the problem on the whole space, i.e.,
 \begin{equation}
 \label{original equation-01111}
 \epsilon^2\Delta u-V(y)u+u^p\,=\,0,\;\;\;u>0~\quad{\rm in}\quad {\mathbb R}^{\mathbf d},
 \end{equation}
where $\epsilon>0$ is a small parameter, the exponent $p>1$,
and $V$ is a smooth function with
 \begin{equation}
 \label{inf-V111}
\displaystyle {\inf_{y\in {\mathbb R}^{\mathbf d}}} V(y)>0.
\end{equation}
In the pioneering work \cite{FW}, Floer and Weinstein constructed positive solutions to this problem
when $p=3$ and ${\mathbf d}=1$ with concentration taking place near a given point ${y}_0$ with $V'({y}_0)=0,\, V''({y}_0) \neq 0$,
being exponentially small in $\epsilon$ outside any neighborhood of ${y}_0$.
This result has been subsequently extended to higher dimensions to the construction of solutions exhibiting
 concentration around one or more points of space under various assumptions on the
potential and the nonlinearity by many authors.
We refer the reader for instance to
\cite{abc, ams, cl, dl1, dl2, dl3, dl4, ft} and the references therein.

\medskip
An important question is whether solutions exhibiting concentration on higher dimensional set exists.
In \cite{AMNconjecture1}, Ambrosetti, Malchiodi and Ni  considered the case of $V=V(|y|)$ and constructed
radial solutions $u_\epsilon(|y|)$ exhibiting concentration on a sphere $|y|=r_0$ in the form
\[
u_\epsilon(r)\sim V^{\frac1{p-1}}(r_0)w\big(V^{\frac12}(r_0)\epsilon^{-1}(r-r_0)\big),
\quad
r=|y|
\]
under the assumption that $r_0 >0$ is a non-degenerate critical point of
\begin{equation}
M(r)=r^{n-1}V^\sigma(r),
\end{equation}
where  $w$ is the unique (even) solution of
\begin{equation}\label{wsolution}
w''-w+w^p=0, \ \ w > 0, \ \ w'(0)=0, \ \ w(\pm \infty)=0,
\end{equation}
and
\begin{equation}\label{sigma}
\sigma=\frac{p+1}{p-1} - \frac12.
\end{equation}
Based on heuristic arguments, in 2003  A. Ambrosetti, A. Malchiodi and W.-M. Ni  raised the following conjecture (p.465, \cite{AMNconjecture1}):
{\it Let ${\mathcal K}$ be a  non-degenerate $\mathbf{k}$-dimensional stationary manifold  of the following functional
\[
\int_{\mathcal K} V^{\frac{p+1}{p-1} -\frac{1}{2} ({\mathbf d}-\mathbf{k})}.
\]
Then there exists a solution to (\ref{original equation-01111}) concentrating near ${\mathcal K}$, at least along a subsequence $\epsilon_j \to 0$.}

\medskip
For ${\mathbf d}=2$, $\mathbf{k}=1$, del Pino, Kowalczyk and Wei \cite{dPKW}  proved the validity of this conjecture under some gap condition.
%Namely, they proved that if  ${\mathcal K}$ in $\R^2$ is  a non-degenerate, stationary curve
%for the weighted length functional $\int_{\mathcal K} V^\sigma$, then given $c > 0$ there exists %$\epsilon_0$ such
%that for all $\epsilon < \epsilon_0$ satisfying the gap condition
%\[
%|\epsilon^2j^2 - \lambda_0| \ge c\epsilon, \quad \forall \ j\in \mathbb{N},
%\]
% where $\lambda_0$ is a fixed positive constant in (\ref{lambda0}), problem (\ref{original %equation-01111})
% has a positive solution $u_\epsilon$ which will concentrate on  ${\mathcal K}$.
% Moreover, for some positive number $c_0$ independent of $\epsilon$, $u_\epsilon$ satisfies globally
% \[
% u_\epsilon(\tilde{y}) \le \exp\big(-c_0\epsilon^{-1}\mbox{dist}(\tilde{y}, {\mathcal K})\big).
% \]
Recently, Mahmoudi, Malchiodi and  Montenegro \cite{mmm} established the validity of the above conjecture in the case of ${\mathbf d}=3, \mathbf{k}=1$. They also considered the complex solutions of (\ref{original equation-01111}) carrying momentums.
For the concentration phenomena directed along general manifolds of codimension one, see \cite{WangWeiYang}.
For a more general problem with suitable hypotheses for the potentials, Bartsch and Peng constructed spike layered solution which concentrates on certain ${\mathbf d}-1$ dimensional spheres in \cite{bartschpeng-1,bartschpeng-2}.

\medskip
Let us go back to the problem on smooth bounded domain with homogeneous Neumann boundary condition,
i.e., problem (\ref{original equation-01}) with the assumption that $V$ is a smooth function and
 \begin{equation}
 \label{inf-V}
\displaystyle {\inf_{y\in \bar{\Omega}}} V(y)>0.
\end{equation}
Ambrosetti, Malchioci and Ni \cite{AMNconjecture2} imposed the following hypotheses on the function $V:$
 \\
 {\bf (V1).} $V\in C^1(\Bbb R^{+},\Bbb R)$,
 \\
 {\bf (V2).} $V$ is bounded and satisfies (\ref{inf-V}),\\
and introduced the auxiliary potential
\begin{equation}
\label{potential}
M(r)\,=\,r^{{\mathbf d}-1}V^{\sigma}(r),~~~{\rm where}~\sigma\,=\,\frac {p+1}{p-1}-\frac 12.
\end{equation}
By taking advantage of the boundary effect, they showed the existence of interior concentration for radial symmetric solutions to (\ref{original equation-01}).
One of their results is as follows:
%\begin{theorem}
%\label{citetheorem}\cite{AMNconjecture2}
{\it Let $\Omega\subseteq \Bbb R^{\mathbf d}$ be the ball $B_1$ and $p>1$.
Suppose that the function $M$ satisfies the condition
\begin{equation}
\label{Mcondition}
M'(1)>0.
\end{equation}
Then there exists a family of radial solutions $u_\epsilon$ of (\ref{original equation-01}) concentrating on $|y|=r_{\epsilon}$,  where $r_\epsilon$ is a local maximum for $u_\epsilon$
such that $1-r_\epsilon\sim \epsilon|\log \epsilon|$.}
To the best knowledge of the authors, for (\ref{original equation-01}), there are no results on the interior non-radial and symmetric concentration phenomena(without intersection between $\partial\Omega$ and the concentration set).

\medskip
For problem (\ref{original equation-01}), another natural case is the existence of interior concentration phenomena intersecting the boundary of $\Omega$.
This is the conjecture also by A. Ambrosetti, A. Malchiodi and W.-M. Ni in 2004 (p.327, \cite{AMNconjecture2}), which can be stated as:
{\it Let ${\mathcal K}$ be a $\mathbf{k}$-dimensional manifold intersecting $\partial\Omega$ perpendicularly, which is also stationary and non-degenerate with respect to the following functional
\[
\int_{\mathcal K} V^{\frac{p+1}{p-1} -\frac{1}{2} ({\mathbf d}-\mathbf{k})}.
\]
Then there exists a solution to (\ref{original equation-01}) concentrating near ${\mathcal K}$, at least along a subsequence $\epsilon_j \to 0$.}

\subsection{Main results}

In the present paper, to avoid complicated computation, we will consider (\ref{original equation-01}) and provide an affirmative answer to the last-mentioned conjecture(called {\bf {Ambrosetti-Malchiodi-Ni conjecture}}) only in the case: ${\mathbf d}=2$ and $\mathbf{k}=1$.
It should be mentioned that the potential term $V$ and  the boundary of $\Omega$ will play an important role in the sense that they basically determine the location of the concentration set and also bring much more difficulties in our procedure of the constructing solutions due to the intersection between $\partial\Omega$ and the concentration set.
The new ingredient is to set up a new local  coordinate system  such that we can decompose the interaction among the interior concentration layers, the boundary of $\Omega$ and the potential term $V$.

\medskip
More precisely, in the present paper, for the existence of a solution $u_{\epsilon}$ concentrating along a curve, say $\Gamma$,
we make the following assumptions on $\Gamma$:

 {\bf (A1).} {\it $\Gamma$ is a curve intersecting $\partial \Omega$ at exactly two points, saying $P_1, P_2$,  and, at these
 points $\Gamma\bot \partial \Omega$.
In the small neighborhoods of $P_1,~P_2$, the boundary $\partial \Omega$ are two curves, say $\mathcal{C}_1$ and ${\mathcal C}_2$, which can be represented
 by the graphs of two functions respectively
 $$
 {y}_2\,=\,\varphi_1({y}_1)\quad\mbox{with }\ (0,\varphi_1(0))\,=\,P_1,
 $$
 $$
 {y}_2\,=\,\varphi_2({y}_1)\quad\mbox{with }\ (0,\varphi_2(0))\,=\,P_2.
 $$
 Moreover, after rescaling, we can always assume that $|\Gamma|=1$ and then denote $k$ the curvature of $\Gamma$, also $k_1$, $k_2$ the signed curvatures of ${\mathcal C}_1$ at $P_1$ and ${\mathcal C}_2$ at $P_2$. }

 {\bf (A2).} {\it The curve $\Gamma$ is a  non-degenerate geodesic relative to the weighted arc length $\int_{\Gamma}V^{\sigma}$, where
 $\sigma$ is defined in (\ref{potential}).
 This will be clarified in the next section, see (\ref{stationary}) and (\ref{nondegeneracy}).
 }

 \medskip
%Let $w$ denote the unique positive solution of problem
% \begin{equation}
% \label{wsolution}
% w{''}-w+w^{p}\,=\,0,\quad w>0,\quad w{'}(0)\,=\,0,\quad w(\pm \infty)\,=\,0.
% \end{equation}
By recalling $w$ given in (\ref{wsolution}),
we consider the associated linearized eigenvalue problem,
\begin{equation}
\label{eigenvalue}
h{''}-h+pw^{p-1}h\,=\,\lambda h\quad {\rm in}~\Bbb R,~\quad h(\pm \infty)\,=\,0.
\end{equation}
It is well known that this equation possesses a unique positive eigenvalue $\lambda_0$ in $H^1(\Bbb R)$,
with associated eigenfunction $Z$ (even and positive) which can be normalized so that $\m Z^2=1$.
In fact, a simple computation shows that
\begin{equation}
\label{lambda0}
\lambda_0\,=\,\frac 14(p-1)(p+3),\quad Z\,=\,\frac {1}{\sqrt{\displaystyle\int_{\mathbb R} w^{p+1}}\,}w^{\frac {p+1}{2}}.
\end{equation}
By setting
 \begin{equation}
\label{definenumber}
\lambda_{*}\,=\,\frac {\lambda_0 \ell^2}{\pi^2}\quad\mbox{with}\quad \ell=\int_\Gamma V^{\frac 12},
\end{equation}
we can formulate our main result.

\begin{theorem}
\label{theorem 1.1}
Let ${\mathbf d}=2$ and $\Gamma$ be a non-degenerate, stationary curve for the weighted length functional $\int_{\Gamma}V^{\sigma}$,  as described above. For given $c>0$ there exists $\epsilon_0>0$
 such that for all $\epsilon<\epsilon_0$ satisfying the gap condition
 \begin{equation}
 \label{gapcondition}
 |\epsilon^2j^2-\lambda_{*}|\geq c\epsilon,~~\forall ~j\in\Bbb N,
 \end{equation}
problem (\ref{original equation-01}) has a positive solution $u_{\epsilon}$, which, for $y$ given by (\ref{Fermicoordinates}) in the neighborhood of $\Gamma$, takes the form
 \begin{equation}
 \label{taketheform}
 u_{\epsilon}(y)\,=\,V(0,\theta)^{\frac {1}{p-1}}w\Big(V(0,\theta)^{\frac 12}\frac t \epsilon\Big)(1+o(1)).
 \end{equation}
For some number $c_0>0,~u_{\epsilon}$~satisfies globally
 $$u_{\epsilon}(y)\,\leq\, \exp \big(-c_0\epsilon^{-1}{\mathrm{dist}} (y,\Gamma)\big).$$
 \qed
\end{theorem}

Here are some words for further discussion.
We do not prove the result for all small values $\epsilon>0$ but only for those which lie away from certain critical numbers due to the resonance phenomena.
 In fact, the gap condition (\ref{gapcondition}) will be useful in dealing with the
 resonance phenomenon, which was described in Section 1 of \cite{dPKW}.
%Note that (\ref{gapcondition}) will be the case whenever $\epsilon$
% is small and away from the critical numbers $\frac {\sqrt{\lambda_{*}}}{k}, k\in \Bbb N $,  in the sense %that for fixed and arbitrarily small $c<\sqrt{\lambda_{*}}$,
% $$
% \epsilon \notin \Big[\frac {\sqrt{\lambda_{*}}}{k}-\frac {c}{k^2},\,
% \frac {\sqrt{\lambda_{*}}}{k}+\frac {c}{k^2}\Big],\quad {\rm for\, all}\quad k\in \Bbb N.
% $$

\medskip

For the convenience of expression in the procedure to show the validity of Theorem \ref{theorem 1.1}, by the rescaling
\begin{align}\label{rescaling}
y\,=\,\epsilon {\tilde y}
\end{align}
 in $\Bbb R^2$,
 (\ref{original equation-01}) will be rewritten as
\begin{equation}
\label{problemafterscaling}
\Delta u_{\tilde y}-V(\epsilon \tilde y)u+u^p\,=\,0\quad{\rm in }\ \Omega_\epsilon,
\qquad
\frac {\partial u}{\partial \nu_\e}\,=\,0\quad{\rm on }\ \partial \Omega_\epsilon,
\end{equation}
where $\Omega_\epsilon=\Omega/\epsilon$ and then also $\Gamma_\epsilon=\Gamma/\epsilon$.
In order to decompose the interaction among concentration layers, the potential $V$
and the boundary $\partial\Omega_\epsilon$, and then construct good approximations to real solutions,
we will first set up a local coordinate system to present $y$ and then ${\tilde y}$ by the relation (\ref{rescaling}), called {\bf modified Fermi coordinates}, in Section \ref{section2}.
This local coordinate system helps us set up the stationary and non-degeneracy conditions for $\Gamma$,
see (\ref{stationary}) and (\ref{nondegeneracy}).

\medskip
An outline of proof for Theorem \ref{theorem 1.1} will be given in Section \ref{section3}.
Gluing procedure from \cite {dPKW}~will be applied to transform \eqref{problemafterscaling} into a projected form of a local problem,
see \eqref{projectedproblem1}-\eqref{projectedproblem4} in Section \ref{section3},
which  will be solved in Sections \ref{theinvertibility} and \ref{solvingthenonlinearprojectionproblem}.
Note that the resolution theory of  \eqref{projectedproblem1}-\eqref{projectedproblem4}  relies heavily
on the properties of the local approximate solution (cf. (\ref{basic approximate})),
which will be constructed in Section {\ref{section4}} step by step in such
a way that it solves the nonlinear problem  locally up to order $O(\epsilon^2)$.

\medskip
To get a real solution, the well-known infinite dimensional reduction method \cite{dPKW} will be needed in Sections \ref{section6}-\ref{section7}. In fact, the reduced problem inherits the resonance phenomena and will be handled by complicated Fourier analysis with the help of the gap condition in (\ref{gapcondition}).
Some tedious computation will be given in Appendices \ref{appendixA} and \ref{appendixB}.

\section{Geometric preliminaries}
\label{section2}
\setcounter{equation}{0}

\noindent In this section, inspired by \cite{ksakamoto} we will set up a local coordinate system in the neighborhood of $\Gamma$, called modified Fermi coordinates,
and then derive its asymptotic expansion.
The notion of stationary and non-degenerate curve $\Gamma$  will be also derived in the last part of this section.

\subsection{Modified Fermi Coordinates}\label{section2.1}

Recall the assumptions  {\bf (A1)} and {\bf (A2)} in Section \ref{section1} and the notation therein.
For basic notions of curves, such as the signed curvature, the reader can refer to the book by do Carmo \cite{docarmo}. The local coordinate system can be given as the following:\\
\noindent {\bf Step  1.}
Let the natural parameterization of the curve $\G$ be as follows:
$$
\g_0:[0,1]\rightarrow \G \subset \bar\Omega\subset \R^2.
$$
We can extend $\gamma_0$ slightly in a smooth way, i.e., for some small positive numbers $\sigma_0$, and define the mapping
$$
\g:(-\sigma_0,1+\sigma_0)\rightarrow \R^2,
$$
such that
$$\g(\t {\theta})\,=\,\g_0(\t{\theta}),~~\forall\, \t{\theta}\in [0,1].$$
There holds the Frenet formula
$$
\gamma''\,=\,kn,\quad n'\,=\,-k\gamma',
$$
where $k$, $n$ are the curvature and the normal of $\gamma$.
Choosing $\delta_0>0$~very small, and setting
$$
{\mathfrak S}_1\,\equiv\,(-\delta_0,\delta_0)\times (-\sigma_0,1+\sigma_0),
$$
we construct the following mapping
\begin{align*}
{\mathbb H}: {\mathfrak S}_1\rightarrow {\mathbb H}({\mathfrak S}_1)\,\equiv\,\Omega_{\delta_0,\sigma_0}
\quad
\mbox{with}
\quad
\mathbb{H}(\t {t},\t {\theta})\,=\,\gamma(\t{\theta})+\t{t}\,n(\t{\theta}).
\end{align*}
Note that ${\mathbb H}$ is a diffeomorphism (locally) and ${\mathbb H}(0,\t{\theta})=\g(\t{\theta})$.

\medskip
\noindent {\bf Step 2}.
Denote the preimage
$$
{\tilde{\mathcal C}}_1\,\equiv\, {\mathbb H}^{-1}\lf(\mathcal C_1\ri),
\qquad
{\tilde{\mathcal C}}_2\,\equiv\, {\mathbb H}^{-1}\lf(\mathcal C_2\ri),
$$
which can be parameterized respectively by $\lf(\t{t}, {\tilde \varphi}_1(\t{t})\ri)$ and  $ \lf(\t{t}, {\tilde \varphi}_2(\t{t})\ri)$ for some smooth functions ${\tilde \varphi}_1(\t{t}),\, {\tilde \varphi}_2(\t{t})$,
and then define a  mapping
$$
\t{{\mathbb H}}:\, {\mathfrak S}_1\rightarrow {\mathfrak S}_2\,\equiv\,\t{{\mathbb H}}({\mathfrak S}_1)\subset \Bbb R^2,
$$
in the form
$$   t=\t{t},\quad \theta=\frac {\t{\theta}-{\tilde \varphi}_1(\t{t})}{{\tilde \varphi}_2(\t{t})-{\tilde \varphi}_1(\t{t})}. $$
This transformation will straighten up the curves ${\tilde{\mathcal C}}_1$ and ${\tilde{\mathcal C}}_2$.

\medskip
It is obvious that
\begin{align}\label{fact1}
{\tilde {\mathbb H}}^{-1}(0,\theta)\,=\,(0,{\theta}),\ \theta\in [0,1].
\end{align}
% \noindent {\it Proof. }
%Note that $${\tilde \varphi}_1(0)\,=\,0,~{\tilde \varphi}_2(0)-{\tilde \varphi}_1(0)=1$, $~it is easy to obtain the conclusion.
Moreover, we have
\begin{lemma}\label{lemma2.1}
There hold
\begin{align}\label{fact2}
\t\varphi_1'(0)=0,\quad \t\varphi_2'(0)=0.
\end{align}
% $\frac {\partial {\mathbb H}^{-1}(0,\theta)}{\partial t}=0$.
%We can let
%$${\tilde k}_1\,=\,|{\tilde \varphi}''_1(0)|,
%\quad
%{\tilde k}_2\,=\,|{\tilde \varphi}_2''(0)|,
%$$
%be the curvatures of the curves ${\tilde{\mathcal C}}_1$ and ${\tilde{\mathcal C}}_2$ at the points of $\t %t=0$ respectively.
Then, there exist the relations
\begin{align}\label{fact3}
k_1={\tilde \varphi}''_1(0),
\qquad
k_2={\tilde \varphi}''_2(0).
\end{align}

\end{lemma}
The proof of this lemma will be given in Appendix \ref{appendixA}.
\qed

\medskip
\medskip
\noindent {\bf Step 3}.
We define the {\bf modified Fermi coordinates}
\begin{equation}
\label{Fermicoordinates}
y\,=\,F(t,\theta)\,=\,{\mathbb H}\circ{\tilde {\mathbb H}}^{-1}(t,\theta)\,:\, (-\delta_0,\delta_0)\times(-\sigma_0,1+\sigma_0)\rightarrow{\mathbb R}^2
\end{equation}
for the given small positive constants $\sigma_0$ and $\delta_0$.
More precisely,
\begin{align*}
\begin{aligned}
F(t,\theta)&\,=\,{\mathbb H}\Big(t, \,\, \big({\tilde \varphi}_2(t)-{\tilde \varphi}_1(t)\big)\theta+{\tilde \varphi}_1(t)\Big)
\\
&\,=\,\g\Big(\big({\tilde \varphi}_2(t)-{\tilde \varphi}_1(t)\big)\theta+{\tilde \varphi}_1(t)\Big)
\,+\,
t\,n\Big(\big({\tilde \varphi}_2(t)-{\tilde \varphi}_1(t)\big)\theta+{\tilde \varphi}_1(t)\Big).
\end{aligned}
\end{align*}

For convenience's sake, in the following, we also denote
$$
\Theta(t, \theta)
\,\equiv\, \big({\tilde \varphi}_2(t)-{\tilde \varphi}_1(t)\big)\theta
\,+\,{\tilde \varphi}_1(t),
$$
and so that
\begin{align}\label{Fformula}
F(t,\theta)\,=\,\gamma\big(\Theta(t,\theta)\big)\,+\,t\,n\big(\Theta(t,\theta)\big).
\end{align}
Note also that
\begin{align}
\begin{aligned}\label{alpha}
\Theta(0,\theta)\,=\,\theta,
\qquad
\Theta_t(0,\theta)\,=\,0,
\qquad
\Theta_\theta(0,\theta)\,=\,1,\quad\
\\
\Theta_{tt}(0,\theta)\,=\,(k_2-k_1)\theta+k_1,
\qquad
\Theta_{tt\theta}(0,\theta)\,=\,k_2-k_1,
\end{aligned}
\end{align}
due to \eqref{fact2} and (\ref{fact3}).
These quantities will play an important role in the further settings.

\medskip
The asymptotic expressions of this coordinate system
will be given by the following basic facts.

\begin{lemma}\label{derivativeofF}
The mapping $F$ in \eqref{Fformula} has the following properties:
\\

\noindent {\rm\textbf{(1).}} $\,\,F(0,\theta)\,=\,\g(\theta)$,
\\

\noindent {\rm\textbf{(2).}} $\,\,\frac{{\partial} F}{{\partial} t}(0,\theta)\,=\,n(\theta)$,
\\

\noindent {\rm\textbf{(3).}} $\,\,q_1(\theta)\,\equiv\,\frac{{\partial}^2 F}{{\partial} t^2}(0,\theta)\,=\,\gamma{\,'}(\theta)\cdot \Theta_{tt}(0,\theta) \quad \bot\, n(\theta)$,
\\

\noindent {\rm\textbf{(4).}} $\,\,q_2(\theta)\,\equiv\, \frac{{\partial}^3F}{{\partial} t^3}(0,\theta)\,=\,\gamma{\,'}(\theta)\cdot \Theta_{ttt}(0,\theta)+3n{'}(\theta)\cdot \Theta_{tt}(0,\theta)
\quad\bot\, n(\theta)$.
\end{lemma}

The proof of this lemma will be given in Appendix \ref{appendixA}.
\qed

\subsection{Stationary and non-degenerate curves}

Finally, for a curve $\Gamma$ connecting the boundary $\partial\Omega$,
 we will make precise the concept of being non-degenerate  geodesic for the metric
 ${\mathrm d}s^2\,=\,V^{\sigma}({\mathrm d}{y}_1^2+{\mathrm d}{y}_2^2)$
 where $\sigma$ is given in (\ref{potential}).
Consider the deformation of $\Gamma$ along the normal with end points staying on $\partial \Omega$,  i.e.,
\begin{equation}\label{non-degeneracy}
\Gamma_{\mathfrak{g}}\,:\, \gamma\big(\Theta({\mathfrak{g}}(\theta),\theta)\big) \,+\,{\mathfrak{g}}(\theta)\,n\big(\Theta({\mathfrak{g}}(\theta),\theta)\big).
\end{equation}
Taking the functional defined on ${\mathfrak{g}}$ as
\begin{equation*}\begin{split}
{\mathcal J}({\mathfrak{g}})&\,=\,\int_0^1V^{\sigma}\big(\Gamma_{{\mathfrak{g}}}(\theta)\big)\big |\Gamma'_{{\mathfrak{g}}}(\theta)\big|\,{\rm d}\theta
\\
&\,=\,\int_0^1V^{\sigma}\Big [(1-k{\mathfrak{g}})^2\big(\Theta_t{\mathfrak{g}}{'}+\Theta_{\theta}\big)^2+({\mathfrak{g}}{'})^2\Big]^{\frac 12}\,{\rm d}\theta,
\end{split}\end{equation*}
where $\Theta={\Theta}\big({\mathfrak{g}}(\theta),\theta\big)$.

\medskip
We write $V({y})=V(t,\theta)$ in the modified Fermi coordinates by setting
 \begin{equation}\label{functional}
 t\,=\,{\mathfrak{g}}(\theta),
 \qquad
 V({y})|_{\Gamma_{{\mathfrak{g}}}}\,=\,V\big({\mathfrak{g}}(\theta),\theta\big).
 \end{equation}
The curve  $\Gamma$ is said to be stationary for the weighted length $\int_{\Gamma_{{\mathfrak{g}}}}V^{\sigma}$ if the first variation of
${\mathcal{J}}$ at ${\mathfrak{g}}=0$ is equal to zero.
That is, for any smooth function $h(\theta)$ defined at $[0,1]$

 \begin{equation*}
 \begin{split}
 0\,=\,{\mathcal{J}}'(0)[h]\,=\,&\int_0^1V^{\sigma}(0,\theta)\Theta_{\theta}^{-1}
 \Big[
 (-k h)\Theta_{\theta}^2\,+\,\Theta_{\theta}(\Theta_t h'+\Theta_{\theta t}h)
 \Big]\,{\rm d}\theta
 \\
 &+\int_{0}^1 \big (V^{\sigma}\big)_t\cdot h\cdot \Theta_{\theta}\,{\rm d}\theta,
 \end{split}\end{equation*}
which is equivalent to the relation
\begin{equation} \label{stationary}
\sigma V_t(0,\theta)\,=\,k(\theta)V(0,\theta).
\end{equation}
 We assume the validity of this relation at $\Gamma$ and then consider the second variation of ${\mathcal{J}}$
 \begin{equation*}\begin{split}
 {\mathcal{J}}{''}(0)[h,f]\,=\,&\,\frac {{\rm d}}{{\rm d}s}\Big|_{s=0}{\mathcal{J}}'(0+sf)[h]
 \\
 \,=\,&\, \int_{0}^1 V^{\sigma}(\Theta_{tt} f+f')h'\,{\rm d}\theta
 +\int_{0}^1 V^{\sigma}\Theta_{tt}\cdot f' h\,{\rm d}\theta
 \\
 &+\int_0^1 \Big [(V^{\sigma})_{tt}-2k (V^{\sigma})_{t}+\Theta_{tt \theta}V^{\sigma}\Big]f h\,{\rm d}\theta
 \\
 \,=\,&\, h\cdot V^{\sigma}(\Theta_{tt}f+f')\big|_{0}^1
 +\int_0^1 \big(-V^{\sigma}f{''}-\sigma V^{\sigma-1}V_{\theta}f'\big)h\,{\rm d}\theta
 \\
 &+\int_0^1 \Big[-\sigma V^{\sigma-1}V_{\theta}\Theta_{tt}+(V^{\sigma})_{tt}-2k(V^{\sigma})_t\Big]fh\,{\rm d}\theta.
 \end{split}\end{equation*}
The curve $\Gamma$ is said to be nondegenerate if
 \begin{equation}\label{2 relation}
  {\mathcal{J}}{''}(0)[h,f]\,=\,0,~~~\forall\, h\in H^1(0,1),
  \end{equation}
then $f\equiv 0$.
 Using (\ref{stationary}), it is easy to check that the relation (\ref{2 relation}) is equivalent to that the boundary problem
 \begin{equation}\label{nondegeneracy}
 \begin{split}
f{''}+\sigma V^{-1}&V_{\theta}f'+\Big [\sigma V^{-1}V_{\theta}\Theta_{tt}-\sigma V^{-1}V_{tt}+(1+\frac 1 \sigma)k^2\Big]f\,=\,0
\quad \mbox{in } (0,1),
\\
&f'(1)+k_2f(1)\,=\,0,\qquad
f'(0)+k_1f(0)\,=\,0,
\end{split}\end{equation}
 has only the trivial solution.

\section{Outline of the proof}
\label{section3}
\setcounter{equation}{0}

In this section, the strategy to prove Theorem \ref{theorem 1.1} will be provided step by step.

\subsection{The gluing procedure}
 For any given approximate solution ${\mathbf W}$ (to be chosen later, cf. (\ref{globalapproximation})) and a perturbation term $\Phi$ defined on $\Omega_\e$,
 the function $u({\tilde y})={\mathbf W}({\tilde y})+\Phi({\tilde y})$ satisfies (\ref{problemafterscaling}) if and only if
\begin{equation}
\label{tildeL}
{\mathbf L}(\Phi)\,=\,-{\mathbf E}-{\mathbf N}(\Phi)
~~{\rm in}~\Omega_{\epsilon},
\end{equation}
with boundary condition
\begin{equation}
\label{boundary condition}
\frac {\partial  \,\Phi}{\partial\, {\nu_\e}}+\frac {\partial \,{\mathbf W}}{\partial\, {\nu_\e}}\,=\,0~~~{\rm on }~\partial \Omega_{\epsilon},
\end{equation}
where
\begin{align}
\begin{aligned}\label{globalerror}
{\mathbf L}(\Phi)\,=\,\Delta_{\tilde y}\Phi-V(\epsilon \tilde{y})\Phi+p\,{\mathbf W}^{p-1}\Phi,
\qquad
{\mathbf E}\,=\,\Delta_{\tilde{y}}{\mathbf W}-V(\epsilon \tilde{y}){\mathbf W}+{\mathbf W}^p,
\\
{\mathbf N}(\Phi)\,=\,({\mathbf W}+\Phi)^p-{\mathbf W}^p-p\,{\mathbf W}^{p-1}\Phi.
\qquad
\qquad
\qquad
\end{aligned}
\end{align}

\medskip
 We separate $\Phi$ in the following form
$$
\Phi({\tilde y})\,=\,\eta_{3\delta}^{\epsilon}(s){\check\phi}({\tilde y})\,+\,{\check\psi}({\tilde y}),
$$
where $s$ is the normal coordinate to $\Gamma_\epsilon$.
In the above formula, the cut-off function is defined as
\begin{align}\label{cutoff}
\eta_{3\delta}^{\epsilon}(s)=\eta_{3\delta}(\epsilon |s|),
\end{align}
where $\eta_{\delta}(t)$  is also a smooth cut-off function defined as
$$
\eta_{\delta}(t)=1,\quad\forall\, 0\leq t\leq \delta
\quad\mbox{and}\quad
\eta(t)=0,\quad \forall\, t>2\delta,
$$
for any fixed number $\delta<\delta_0/100$, where $\delta_0>0$ is a small constant in (\ref{Fermicoordinates}).
With this definition, $\Phi$ is a solution of (\ref{tildeL})-(\ref{boundary condition}) if the pair $({\check\phi},{\check\psi})$ satisfies the following coupled system:
\begin{align}
\eta_{3\delta}^{\epsilon}\,\big(\Delta_{\tilde{y}}{\check\phi}-V{\check\phi}+p\,{\mathbf W}^{p-1}{\check\phi}\big)
\,&=\,
\eta_{\delta}^{\epsilon}\,\big[-{\mathbf N}(\eta_{3\delta}^{\epsilon}\,{\check\phi}+{\check\psi})
-{\mathbf E}-p\,{\mathbf W}^{p-1}{\check\psi}\big]\quad\mbox{in } \Omega_\epsilon,
\label{equivalent system-1}
\\
\Delta_{\tilde{y}}\,{\check\psi}-V\,{\check\psi}+(1-\eta_{\delta}^{\epsilon})\,p\,{\mathbf W}^{p-1}\,{\check\psi}
\,&=\,-\epsilon^2\,(\Delta_{\tilde{y}}\,\eta_{3\delta}^{\epsilon})\,{\check\phi}-2\epsilon(\nabla_{\tilde{y}}\, \eta_{3\delta}^{\epsilon})\,(\nabla_{\tilde{y}} \,{\check\phi})\nonumber
\\
&\quad\,\,-(1-\eta_{\delta}^{\epsilon})\,{\mathbf N}(\eta_{3\delta}^{\epsilon}\,{\check\phi}+{\check\psi})-(1-\eta_{\delta}^{\epsilon})\,{\mathbf E}
\quad\mbox{in } \Omega_\epsilon,
\label{equivalent system-2}
\end{align}
with the boundary conditions
\begin{equation}
\label{on the boundary-1}
\eta_{3\delta}^{\epsilon}\,\frac {\partial\, {\check\phi}}{\partial\, {\nu_\e}}+\eta_{\delta}^{\epsilon}\,\frac {\partial\, {\mathbf W}}{\partial \,{\nu_\e}}\,=\,0
\qquad\mbox{on } \partial\,\Omega_\epsilon,
\end{equation}

\begin{equation}
\label{on the boundary-2}
\frac {\partial\, {\check\psi}}{\partial \,{\nu_\e}}+(1-\eta_{\delta}^{\epsilon})\,\frac {\partial \, {\mathbf W}}{\partial\, {\nu_\e}}+\epsilon\, \frac {\partial\, \eta_{3\delta}^{\epsilon}}{\partial \, {\nu_\e}}\,{\check\phi}\,=\,0
\quad\mbox{on } \partial\,\Omega_\epsilon.
\end{equation}

\medskip
Given a small ${\check\phi}$,  we first solve
problem (\ref{equivalent system-2}) and (\ref{on the boundary-2}) for ${\check\psi}$.  Assume now that ${\check\phi}$ satisfies the following decay property
\begin{equation}
\label{decay property}
\big |\nabla {\check\phi}(\tilde{y})\big|+\big|{\check\phi}(\tilde{y})\big|\,\leq\, Ce^{-\gamma/\epsilon} ~~{\rm ~if }\ \mathrm{dist}({\tilde y}, \Gamma_\e)>\delta/\epsilon,
\end{equation}
for a certain constant $\gamma>0$, also that
\begin{align}\label{decayw}
{\mathbf W}({\tilde y})
\ \mbox{is exponentially small if }\ \mathrm{dist}({\tilde y}, \Gamma_\e)>\delta/\epsilon.
\end{align}
These two assumptions will be fulfilled by the choosing of $\phi$ in
Proposition \ref{prop} due to the relation (\ref{vdefine}), and also ${\mathbf W}$ in (\ref{globalapproximation}).
  The solvability can be done in the following way: let us observe that $s$ is the normal coordinate to $\Gamma_\epsilon$.  \;Then the problem
\begin{equation*}
\begin{split}
&\Delta {\check\psi}-\big[V-(1-\eta_{\delta}^{\epsilon})\,p\,{\mathbf W}^{p-1}\big]{\check\psi}\,=\,h
\qquad{\rm in }\,\,\Omega_\epsilon,
\\
&\frac {\partial\,  {\check\psi}}{\partial\,  {\nu_\e}}\,=\,-(1-\eta_{\delta}^{\epsilon})\,\frac {\partial\,  {\mathbf W}}{\partial \, {\nu_\e}}-\epsilon \,\frac {\partial \, \eta_{3\delta}^{\epsilon}}{\partial\,  {\nu_\e}}\,{\check\phi}
\qquad\,\,{\rm on }\,\,\partial\Omega_\epsilon,
\end{split}
\end{equation*}
has a unique bounded solution ${\check\psi}$ whenever $\|h\|_{\infty}<+\infty$.
Moreover,
$$
\|{\check\psi}\|_{\infty}\,\leq\, C\|h\|_{\infty}.
$$
Since ${\mathbf N}$ is power-like with power greater than one, a direct application of contraction mapping principle yields that (\ref{equivalent system-2}) and (\ref{on the boundary-2}) has a unique (small) solution ${\check\psi}={\check\psi}({\check\phi})$ with
\begin{equation}
\label{contraction-1}
\|{\check\psi}({\check\phi})\|_{L^{\infty}}\,\leq\, C\epsilon\big[\|{\check\phi}\|_{L^{\infty}(|s|>\delta/\epsilon)}+
\|\nabla {\check\phi}\|_{L^{\infty}(|s|>\delta/\epsilon)}+e^{-\delta/\epsilon}\big],
\end{equation}
where $|s|>\delta/\epsilon$ denotes the complement in $\Omega_{\epsilon}$ of $\delta/\epsilon$-neighborhood of $\Gamma_\epsilon$. Moreover, the nonlinear operator ${\check\psi}$ satisfies a Lipschitz condition of the form
\begin{equation}
\label{psi-lip}
\|{\check\psi}({\check\phi}_1)-{\check\psi}({\check\phi}_2)\|_{L^{\infty}}\,\leq\, C\epsilon \Big[\,\|{\check\phi}_1-{\check\phi}_2\|_{L^{\infty}(|s|>\delta/\epsilon)}
\,+\,
\|\nabla {\check\phi}_1-\nabla {\check\phi}_2\|_{L^{\infty}(|s|>\delta/\epsilon)}\,\Big].
\end{equation}

\medskip
The key observation is that, after solving (\ref{equivalent system-2}) and (\ref{on the boundary-2}),
we can concern (\ref{equivalent system-1}) and (\ref{on the boundary-1}) as  a local nonlinear problem  involving ${\check\psi}\,=\,{\check\psi}({\check\phi})$, which can be solved in local coordinates in the sense that we can decompose the interaction among the boundary, the concentration set and the potential $V$, and then construct a good approximate solution and also derive the resolution theory of the nonlinear problem by delicate analysis.
This procedure is called a gluing technique in \cite{dPKW}.

\medskip
\subsection{Local formulation of the problem}\label{Sectionlocaloperators}

As described in the above, the next step is to consider (\ref{equivalent system-1})-(\ref{on the boundary-1}) in the neighbourhood of $\Gamma_\e$ so that by the relation ${\tilde y}=y/\e$ in (\ref{rescaling}) close to $\Gamma_\e$, the variables $y$ can be represented by the modified Fermi coordinates $(t, \theta)$ in (\ref{Fermicoordinates}), which
have been set up in Section \ref{section2}.

\subsubsection{Local form of the problem} By using the local coordinates $(t, \theta)$ in (\ref{Fermicoordinates}),
the local forms of the differential operators $\triangle_y$ and $\partial/\partial{\nu}$ in (\ref{original equation-01}) are given in (\ref{laplaceorigin}), (\ref{boundaryoriginal0}) and (\ref{boundaryoriginal1}).
By recalling the rescaling $y=\epsilon {\tilde y}$ in (\ref{rescaling}),
it is useful to  locally introduce change of variables
\begin{equation}
\label{szcoordinate}
u(y)\,=\,\hat{u}({\tilde y})\qquad{\rm and}\qquad \epsilon\,s\,=\,t,~~\epsilon\,z\,=\,\theta.
\end{equation}
%It is easy to check that
%$$\Delta_{t,\theta}u(t,\theta)\,=\,\frac {1}{\epsilon^2}\Delta_{s,z}\hat{u}(s,z).$$
Then (\ref{laplaceorigin}) will give
$$
\epsilon^2\Delta_y \,=\,\epsilon^2\partial^2_{tt}+\epsilon^2\partial^2_{\theta\theta}+\epsilon^2\bar{B}_1(\cdot)
+\epsilon^2\bar{B}_0(\cdot),
$$
where the form of $\bar{B}_1$ in (\ref{B1bar}) will provide
\begin{equation*}
\begin{split}
\epsilon^2\bar{B}_1(\cdot)&\,=\,-\epsilon^2(k+k^2t)\partial_t
-2\epsilon^2t\,\varpi\,\partial^2_{t\theta}
-\epsilon^2\,\varpi\,\partial_{\theta}
\\
&\,=\,-(\epsilon k+\epsilon^2k^2s)\partial_s
-2\epsilon s\,\varpi\,\partial^2_{sz}
-\epsilon\,\varpi\,\partial_z
\\
&\,\equiv\, \hat{B}_1(\cdot),
\end{split}
\end{equation*}
and the expression of $\bar{B}_0$ in (\ref{B0bar}) indicates that
\begin{equation}
\label{tilteB00}
\begin{split}
\epsilon^2\bar{B}_0(\cdot)\,=\,&\,
2\epsilon^2k\,t\,\partial^2_{\theta\theta}
+a_1\,\epsilon^2\,t^2\,\partial^2_{\theta\theta}
+a_2\,\epsilon^2\,t^3\,\partial^2_{tt}
+a_3\,\epsilon^2\,t^2\,\partial^2_{t\theta}
+a_4\,\epsilon^2\,t^2\,\partial_t
+a_5\,\epsilon^2\,t\,\partial_\theta
\\
\,=\,&\, 2\,k\,\epsilon\, s\,\partial^2_{zz}
+a_1\,\epsilon^2\,s^2\,\partial^2_{zz}
+a_2\,\epsilon^3\,s^3\,\partial^2_{ss}
+a_3\,\epsilon^2\,s^2\,\partial^2_{sz}
+a_4\,\epsilon^3\,s^2\,\partial_s
+a_5\,\epsilon^2\,s\partial_z
\\
\,\equiv\,&\, \hat{B}_0(\cdot).
\end{split}
\end{equation}
In the above formulas, the smooth functions $\varpi, a_1, \cdots, a_5$
are given in (\ref{a0}) and (\ref{a1a5}).
This implies that
\begin{align*}
\Delta_{\tilde y}=\e^2\Delta_y=\partial^2_{ss}+\partial^2_{zz}+\hat{B}_1(\cdot)+\hat{B}_0(\cdot).
\end{align*}
Similarly, the normal derivative $\partial/\partial\nu_\e$ can be derived from (\ref{boundaryoriginal0}) and (\ref{boundaryoriginal1}). Indeed, if $z=0$, it becomes the following boundary operator
\begin{equation}
\label{boundarycondition10}
\mathbb{D}_0\,\equiv\,k_1\,\epsilon \,s\,\partial_s+b_1\,\epsilon^2\,s^2\,\partial_s
-\partial_z-k(0)\,\epsilon \,s\,\partial_z+b_2\,\epsilon^2 \,s^2 \partial_z
+\hat{D}^0_0(\cdot),
\end{equation}
where
$$\hat{D}^0_0\,=\,\epsilon\,\bar{D}^0_0.$$
And, at $z=1/\epsilon$, it has the form
\begin{equation}
\label{boundarycondition20}
\begin{split}
\mathbb{D}_1\,\equiv\,k_2\,\epsilon \,s\,\partial_s+b_3\,\epsilon^2\,s^2\,\partial_s
-\partial_z-k(1)\,\epsilon\, s\partial_z+b_4\,\epsilon^2 \,s^2 \,\partial_z
+\hat{D}^1_0(\cdot),
\end{split}
\end{equation}
where
$$\hat{D}^1_0\,=\,\epsilon\,\bar{D}^1_0.$$
In the above equations, the constants $b_1, b_2, b_3, b_4$ are given in (\ref{b1b2}) and (\ref{b3b4}).

\medskip
Hence, the equation in (\ref{equivalent system-1})
 can be locally recast in $(s,z)$ coordinate system as follows
\begin{align}
\begin{aligned}
\label{s-z-laplace}
\eta_{3\delta}^{\epsilon}\check{L}({\check\phi})=\eta_{\delta}^{\epsilon}\,\big[-{\mathbf N}({\eta_{3\delta}^{\epsilon}\check\phi}+{\check\psi})
-{\mathbf E}-p\,{\mathbf W}^{p-1}{\check\psi}\big],
\end{aligned}
\end{align}
where the linear operator is
\begin{align}
\label{tilde L check phi}
{\check L}({\check\phi})={\check\phi}_{ss}+{\check\phi}_{zz}+\hat{B}_{1}({\check\phi})
+\hat{B}_0({\check\phi})-V(\epsilon s,\epsilon z)\,{\check\phi}+p{\mathbf W}^{p-1}{\check\phi}
\end{align}
and  the error is now locally expressed in the form
\begin{align}
\begin{aligned}
{\mathbf E}={\mathbf W}_{ss}+{\mathbf W}_{zz}+\hat{B}_{1}({\mathbf W})+\hat{B}_0({\mathbf W})-V(\epsilon s,\epsilon z)\,{\mathbf W}+{\mathbf W}^{p}.
\end{aligned}
\end{align}
The boundary condition in (\ref{on the boundary-1}) can also be expressed precisely in the local coordinates.
If $z=0$,
\begin{equation}
\label{boundarycondition11}
\eta_{3\delta}^{\epsilon}\,\mathbb{D}_0(\check{\phi})=-\eta^\e_\delta\,{\mathbf G}_0\qquad{\rm with}\qquad {\mathbf G}_0=\mathbb{D}_0({\mathbf W}),
\end{equation}
and also, at $z=1/\epsilon$ there holds
\begin{equation}
\label{boundarycondition21}
\eta_{3\delta}^{\epsilon}\,\mathbb{D}_1(\check{\phi})=-\eta^\e_\delta\,{\mathbf G}_1\qquad{\rm with}\qquad {\mathbf G}_1=\mathbb{D}_1({\mathbf W}).
\end{equation}

\subsubsection{Further changing of variables}\label{further change}
In this section we consider a further change of variables in equation
(\ref{s-z-laplace}) in such a way that it replaces at main order the potential $V$ by $1$.

\medskip
Let
\begin{equation}
\label{alpha-beta}
\alpha(\theta)\,=\,V(0,\theta)^{\frac {1}{p-1}},~~~~~\quad \beta(\theta)\,=\,V(0,\theta)^{\frac 12},
\end{equation}
and fix two twice differentiable functions $f$ and $h$, which satisfy the following constraints
\begin{equation}
\label{fnorm}
\|f\|_{*}\,\equiv\, \|f\|_{L^{\infty}(0,1)}+\|f'\|_{L^{\infty}(0,1)}+\|f{''}\|_{L^2(0,1)}\,\leq\, \epsilon^{ 1/2},
\end{equation}
and
\begin{equation}
\label{constraints of h}
h'(1)+k_2h(1)=0,\qquad\qquad h'(0)+k_1h(0)=0.
\end{equation}
In fact, $f$ will be found by the reduction method described in Sections \ref{section6}-\ref{section7},
while $h$ is a solution to problem (\ref{equation of h}).
By defining the relation
\begin{equation}
\label{vdefine}
{\check\phi}(s,z)\,=\,\alpha(\epsilon z)\phi(x,z)
\quad\mbox{with}\quad
x\,=\,\beta(\epsilon z)\big(s-f(\epsilon z)-h(\epsilon z)\big),
\end{equation}
we can express equation (\ref{s-z-laplace}) in terms of these new coordinates
due to the following trivial computation
\begin{equation}
\label{s-z-transform}
\begin{split}
{\check\phi}_s\,&=\,\alpha \beta \phi_x,
\qquad\qquad
{\check\phi}_{ss}\,=\,\alpha \beta^2\phi_{xx},
\\
{\check\phi}_{z}&\,=\,\epsilon \alpha'\phi+\alpha \phi_x\big(\beta(s-f-h)\big)_z+\alpha \phi_z,
\\
{\check\phi}_{zz}\,&=\,\epsilon^2\,\alpha{''}\phi+2\,\epsilon \,\alpha'\,\Big[\,\phi_x\,\big(\beta(s-f-h)\big)_z+\phi_z\Big]
\\
&\quad+\alpha\Big[\phi_{xx}
\big|\big(\beta(s-f-h)\big)_z\big|^2+2\,\phi_{xz}\big(\beta(s-f-h)\big)_z
+\phi_x\big(\beta(s-f-h)\big)_{zz}+\phi_{zz}\,\Big],
\\
{\check\phi}_{sz}&\,=\,\epsilon \,\alpha'\,\beta \,\phi_x+\epsilon \,\alpha \,\beta'\, \phi_x+\alpha \,\beta\,\phi_{xx}\big(\beta(s-f-h)\big)_z+\alpha \,\beta\, \phi_{xz},
\end{split}
\end{equation}
where
\begin{equation*}
\begin{split}
\big(\beta(s-f-h)\big)_z\,=\,&\,\epsilon\big[\,\beta'(s-f-h)-\beta f'-\beta h'\,\big]
\,=\,\epsilon \beta^{-1}\beta'x-\epsilon\beta f'-\epsilon \beta h',
\\
\big(\beta(s-f-h)\big)_{zz}\,=\,&\,\epsilon^2\Big[\,\beta{''}(s-f-h)-2\beta'f'-2\beta' h'-\beta f{''}-\beta h{''}\,\Big]
\\[1mm]
\,=\,&\,\epsilon^2\beta^{-1}\beta''x-2\epsilon^2\beta'f'-2\epsilon^2 \beta'h'-\epsilon^2\beta f''-\epsilon^2\beta h''.
\end{split}
\end{equation*}
In order to write down the equation, it is also convenient to expand
\begin{equation}
\label{Vexpan}
V(\epsilon s,\epsilon z)\,=\,V(0,\epsilon z)+V_t(0,\epsilon z)\cdot \epsilon s+\frac 12 V_{tt}(0,\epsilon z)\cdot \epsilon^2s^2+a_6(\epsilon s,\epsilon z)\,\epsilon^3\,s^3
\end{equation}
for a smooth function $a_6(t,\theta)$.

\medskip
By the above change of coordinates, the linear operator ${\check L}$ has the following form
\begin{equation*}
%\label{s-z-x-z}
\begin{split}
{\check L}({\check\phi})=&\alpha \beta^2 \phi_{xx}+\epsilon^2 \alpha{''}\phi+2\epsilon \alpha{'}\Bigg[\epsilon \Big(\frac {\beta'}{\beta}x-\beta f'-\beta h'\Big)\phi_x+\phi_z\Bigg]
\\
&+\alpha\Bigg[\epsilon^2\Big |\frac {\beta'}{\beta}x-\beta f'-\beta h'\Big |^2\phi_{xx}+2\epsilon\Big(\frac {\beta'}{\beta}x-\beta f'-\beta h'\Big)\phi_{xz}
\\
&+\epsilon^2\Big(\frac {\beta{''}}{\beta}x-2\beta'f'-2\beta' h'-\beta f{''}-\beta h{''}\Big)\phi_x+\phi_{zz}\Bigg]
-\Bigg[\epsilon k+\epsilon^2k^2\Big(\frac {x}{\beta}+f+h\Big)\Bigg]\alpha \beta \phi_x
\\
&-2\,\epsilon \,\varpi\,\Big(\frac {x}{\beta}+f+h\Big)\,\Big[\epsilon \alpha'\beta \phi_x+\epsilon \alpha \beta' \phi_x+\epsilon\alpha \beta' x\phi_{xx}-\epsilon \alpha \beta^2(f'+h')\phi_{xx}
+\alpha \beta \phi_{xz}\Big]
\\
&-\epsilon\,\varpi\,\Bigg[\epsilon \alpha'\phi+\epsilon\alpha \Big(\frac{\beta'}{{\beta}}x-\beta f'-\beta h'\Big)\phi_{x}+\alpha \phi_z\Bigg]+\hat{B}_0({\check\phi})
\\
&-\Bigg[V(0,\epsilon z)+V_t(0,\epsilon z)\cdot \epsilon s+\frac 12 V_{tt}(0,\epsilon z)\epsilon^2s^2+\epsilon^3a_6(\epsilon s,\epsilon z)s^3\Bigg]\alpha \phi+p{\mathbf W}^{p-1}\alpha\phi.
\end{split}
\end{equation*}
Dividing the above equation by $\alpha \beta^{2}$, we can induce that
\begin{equation}
\bar{L}(\phi)\equiv(\alpha\beta^2)^{-1}{\tilde L}({\check\phi})=\beta^{-2}\phi_{zz}+\phi_{xx}-\phi+\beta^{-2}p{\mathbf W}^{p-1}\phi+B_3(\phi),
\end{equation}
where $B_3(\phi)$ is a linear differential operator defined by
\begin{align}
B_3(\phi)\,=\,&-\beta^{-1}\Big[\epsilon k+\epsilon^2k^2\Big(\frac x \beta+f+h\Big)\Big]\phi_x
\nonumber
\\
&+\beta^{-2}\Bigg[\epsilon^2\Big|\frac {\beta'}{\beta}x-\beta f'-\beta h'\Big|^2\phi_{xx}+2\epsilon \Big(\frac {\beta'}{\beta}x-\beta f'-\beta h'\Big)\phi_{xz}
\nonumber
\\
&\qquad\qquad+\epsilon^2\Big(\frac {\beta{''}}{\beta}x-2\beta' f'-2\beta' h'-\beta f{''}-\beta h{''}\Big)\phi_x\Bigg]
\nonumber
\\
&+\frac {\epsilon^2}{\alpha \beta^2}\alpha{''}\phi
\,+\,\frac {2\epsilon \alpha'}{\alpha \beta^2}\Bigg[\epsilon \Big(\frac {\beta'}{\beta}x-\beta f'-\beta h'\Big)\phi_x+\phi_z\Bigg]
\label{B3v}
\\
&-\frac {2\epsilon\,\varpi}{\alpha \beta^2}\Big(\frac x \beta+f+h\Big)\,\Big[\epsilon \alpha'\beta \phi_x+\epsilon \alpha \beta' \phi_x+\epsilon\alpha \beta' x\phi_{xx}-\epsilon \alpha \beta^2f'\phi_{xx}
\nonumber
\\
&\qquad\qquad\qquad\qquad\qquad-\epsilon \alpha \beta^2h'\phi_{xx}+\alpha \beta \phi_{xz}\Big]
\nonumber
\\
&-\frac {\epsilon\,\varpi}{\alpha \beta^2}\,\Bigg[\epsilon\alpha'\phi+\epsilon\alpha \Big(\frac {\beta'}{\beta}x-\beta f'-\beta h'\Big)\phi_x+\alpha \phi_z\Bigg]
\nonumber
\\
&-\Bigg[\epsilon \,\sigma^{-1}k\,\Big(\frac x \beta+f+h\Big)+\frac {\epsilon^2}{2}\beta^{-2}V_{tt}\Big(\frac x \beta+f+h\Big)^2\Bigg]\phi+B_2(\phi),
\nonumber
\end{align}
where $\varpi$ is defined in (\ref{a0}).~
In the last line of (\ref{B3v}), we have used the relation (\ref{stationary}) given by the assumption that $\Gamma$ is stationary.
Here
\begin{equation}
\label{B2v}
B_2(\phi)\,=\,(\alpha \beta^2)^{-1}\,\hat{B}_0({\check\phi})+(\alpha \beta^2)^{-1}\,a_6(\epsilon s,\epsilon z)\,\epsilon^3 \,s^3\,\alpha\,\phi,
\end{equation}
and $\hat{B}_0({\check\phi})$ is the operator in (\ref{tilteB00}) where derivatives are expressed in terms of formulas through (\ref{s-z-transform}), $a_6$ is given by (\ref{Vexpan}), and $s$ is replaced by $\beta^{-1}x+f+h$.

\medskip
In the coordinates $(x,z)$, the boundary conditions in (\ref{boundarycondition11}) and (\ref{boundarycondition21}) can be recast as follows.  For $z=0$,
\begin{equation}
\label{boundary-1-x-z}
\eta_{3\delta}^{\epsilon}\,\big(D_3^0(\phi)- \phi_z+D_0^0(\phi)\big)\,=\,-\alpha^{-1}\,\eta^\e_\delta\,{\mathbf G}_0,
\end{equation}
where
\begin{equation}
\label{D30}
\begin{split}
D_3^0(\phi)\,=\,&\epsilon\big[b_5x+\beta (k_1f+f')\big]\phi_{x}-\epsilon\alpha'\alpha^{-1}\phi-\epsilon k\Big(\frac{x}{\beta}+f+h\Big)\phi_z
\\
&+\epsilon^2\Big[b_1\,\Big(\frac x \beta +f+h\Big)^2\beta-k\Big(\frac{x}{\beta}+f+h\Big)\Big(\frac{\beta'}{\beta}x-\beta f'-\beta h'\Big)\Big]\phi_x
\\
&-\epsilon^2\frac{k\alpha'}{\alpha}\Big(\frac{x}{\beta}+f+h\Big)\phi+\epsilon^2b_2\Big(\frac{x}{\beta}+f+h\Big)^2\phi_z,
\end{split}
\end{equation}
the constant $b_5$ is
\begin{align}\label{b5}
b_5\,=\,k_1 - \frac {\beta'(0)}{\beta(0)},
\end{align}
and
\begin{align}
\begin{aligned}
D_0^0(\phi)\,=&\,\alpha^{-1}{\hat D}_0^0({\check{\phi}})+\epsilon^3\,b_2\,\Big(\frac{x}{\beta}+f+h\Big)^2\frac{\alpha'}{\alpha}\phi
\\
&+\epsilon^3\,b_2\,\phi_x\,\Big(\frac{x}{\beta}+f+h\Big)^2\Big(\frac{\beta'}{\beta}x-\beta f'-\beta h'\Big).
\end{aligned}
\end{align}
Similarly, for $z=1/\epsilon$,  we have
\begin{equation}
\label{boundary-2-x-z}
\eta_{3\delta}^{\epsilon}\,\big(D_3^1(\phi)- \phi_z+D_0^1(\phi)\big)\,=\,-\alpha^{-1}\,\eta^\e_\delta\,{\mathbf G}_1,
\end{equation}
where
\begin{equation}
\label{D31}
\begin{split}
D_3^1(\phi)\,=\,&\epsilon\big[b_6x+\beta (k_2f+f')\big]\phi_{x}-\epsilon\alpha'\alpha^{-1}\phi-\epsilon k\Big(\frac{x}{\beta}+f+h\Big)\phi_z
\\
&+\epsilon^2\Big[b_3\Big(\frac x \beta +f+h\Big)^2\beta-k\Big(\frac{x}{\beta}+f+h\Big)\Big(\frac{\beta'}{\beta}x-\beta f'-\beta h'\Big)\Big]\phi_x
\\
&-\epsilon^2\frac{k\alpha'}{\alpha}\Big(\frac{x}{\beta}+f+h\Big)\phi+\epsilon^2b_4\Big(\frac x   \beta+f+h\Big)^2\phi_z,
\end{split}
\end{equation}
 the constant
\begin{align}\label{b6}
b_6\,=\,k_2 - \frac {\beta'(1)}{\beta(1)},
\end{align}
and
\begin{align}
\begin{aligned}
D_0^1(\phi)\,=&\,\alpha^{-1}{\hat D}_0^1({{\check{\phi}}})+\epsilon^3\,b_4\,\Big(\frac{x}{\beta}+f+h\Big)^2\frac{\alpha'}{\alpha}\phi
\\
&+\epsilon^3\,b_4\,\phi_x\,\Big(\frac{x}{\beta}+f+h\Big)^2\Big(\frac{\beta'}{\beta}x-\beta f'-\beta h'\Big).
\end{aligned}
\end{align}

\medskip
As a conclusion,
it is derived that (\ref{equivalent system-1}) and (\ref{on the boundary-1}) becomes,
in local coordinates $(x,z)$,
\begin{equation}\label{localproblem1}
\eta_{3\delta}^{\epsilon}\, \bar{L}(\phi)\,=(\alpha\beta^2)^{-1}\,\eta_{\delta}^{\epsilon}\,
\big[-{\mathbf N}({\eta_{3\delta}^{\epsilon}\check\phi}+{\check\psi})
-{\mathbf E}-p\,{\mathbf W}^{p-1}{\check\psi}\big],
\end{equation}
\begin{equation}
\eta_{3\delta}^{\epsilon}\,\big(D_3^0(\phi)- \phi_z+D_0^0(\phi)\big)\,=\,-\alpha^{-1}\,\eta^\e_\delta\,{\mathbf G}_0,
\end{equation}
\begin{equation}\label{localproblem3}
\eta_{3\delta}^{\epsilon}\,\big(D_3^1(\phi)- \phi_z+D_0^1(\phi)\big)\,=\,-\alpha^{-1}\,\eta^\e_\delta\,{\mathbf G}_1.
\end{equation}

\subsection{Derivation of the projected problem by extension method}

Before going further to the resolution theory of (\ref{localproblem1})-(\ref{localproblem3}), we give some notation.

\medskip
\noindent{\textbf {Notation:}}
{\it
For simplicity, denote
\begin{equation}
\label{mathcalF}
\mathcal F\,=\,\big\{\,(f,e)\,|\, \mbox{ the functions } f \mbox{ and } e \mbox{ satisfy } (\ref{fnorm})
\mbox{ and } (\ref{enorm}) \mbox{ respectively}\,\big\}.
\end{equation}
Observe that all functions involved are expressed in $(x,z)$-variables, and the natural domain for those variables can be extended to the infinite strip
\begin{align}
\begin{aligned}\label{domainS}
{\mathcal S}\,=\,&\Big\{\, (x, z)\,:\,-\infty<x<\infty,~~~0<z<1/\epsilon\,\Big\},
\\
\partial_0{\mathcal S}\,=\,&\Big\{\, (x, z)\,:\,-\infty<x<\infty,~~~z=0\,\Big\},
\\
\partial_1{\mathcal S}\,=\,&\Big\{\, (x, z)\,:\,-\infty<x<\infty,~~~z=1/\epsilon\,\Big\}.
\end{aligned}
\end{align}
Accordingly, we define
\begin{align}
\begin{aligned}\label{domainS1}
{\mathcal {\hat{S}}}\,=\,&\Big\{\, (x, \tilde{z})\,:\,-\infty<x<\infty,~~~0<\tilde{z}<{\ell}/\epsilon\,\Big\},
\\
\partial_0{\mathcal {\hat{S}}}\,=\,&\Big\{\, (x, \tilde{z})\,:\,-\infty<x<\infty,~~~\tilde{z}=0\,\Big\},
\\
\partial_1{\mathcal {\hat{S}}}\,=\,&\Big\{\, (x, \tilde{z})\,:\,-\infty<x<\infty,~~~\tilde{z}={\ell}/\epsilon\,\Big\},
\end{aligned}
\end{align}
where ${\ell}$ is a constant defined in (\ref{definenumber}).
}
\qed

\medskip
One of the left jobs is to find the local forms of the approximate solution ${\mathbf W}$
with the constraint (\ref{decayw}) and also of the error ${\mathbf E}$.
We recall the transformation in (\ref{alpha-beta}) and (\ref{vdefine}),
and then define for the local form of ${\mathbf W}$ by the relation
\begin{equation}
\label{vdefine1}
\eta^{\epsilon}_{10\delta}(s)\,{\mathbf W}\,=\,\eta^{\epsilon}_{3\delta}(s)\,\alpha(\epsilon z)\,v(x,z)
\quad\mbox{with}\quad
x\,=\,\beta(\epsilon z)\big(s-f(\epsilon z)-h(\epsilon z)\big).
\end{equation}
As we have done for the equation (\ref{s-z-laplace}),
${\mathbf E}$ can be locally recast in $(x,z)$  coordinate system   by the relation
\begin{equation}\label{errorrelationinterior}
(\alpha\,\beta^2)^{-1}\eta_\delta^\epsilon(s)\,{\mathbf E}\,=\,\eta_\delta^\epsilon(s)\,{\mathcal E},
\end{equation}
where
\begin{equation}
\label{S(v)}
{\mathcal E}=S(v)
\quad\mbox{with}\quad
S(v)\,\equiv\,\beta^{-2}v_{zz}+v_{xx}-v+v^p+B_3(v),
\end{equation}
with the operator $B_3$ defined in (\ref{B3v}).
In the coordinates $(x,z)$, the boundary errors can be recast as follows.
For $z=0$, there holds
\begin{equation}
\label{errorrelationboundary0}
\begin{split}
\alpha^{-1}\eta_\delta^\epsilon(s)\,{\mathbf G}_0\,=\,\eta_\delta^\epsilon(s)\,g_0
\quad\mbox{with}\quad
g_0=D^0_3(v)-v_z+D^0_0(v),
\end{split}
\end{equation}
and also for $z=1/\epsilon$,  we have
\begin{equation}
\label{errorrelationboundary1}
\begin{split}
\alpha^{-1}\eta_\delta^\epsilon(s)\,{\mathbf G}_1\,=\,\eta_\delta^\epsilon(s)\,g_1
\quad\mbox{with}\quad
g_1=D^1_3(v)-v_z+D^1_0(v).
\end{split}
\end{equation}
It is of importance that (\ref{errorrelationinterior}), (\ref{errorrelationboundary0}) and (\ref{errorrelationboundary1}) hold only in a small neighbourhood of $\Gamma_\epsilon$.
Hence we will consider $v$, ${\mathcal E}$ as functions of the variables $x$ and $z$ on ${\mathcal S}$, and also $g_0$, $g_1$ on $\partial_0{\mathcal S}$ and $\partial_1{\mathcal S}$ in the sequel.
We will find $v=v_5$ in (\ref{basic approximate}) step by step in Section \ref{section4}, so that ${\mathbf W}$ will be given in (\ref{globalapproximation}) with the property in (\ref{decayw}).
In fact, to deal with the resonance, in addition to the parameters $f$ and $h$,
we shall add one more parameter, say $e$ (cf. (\ref{u3define})),  in the approximate solution $v$.
In all what follows, we will assume the validity of the following constraint on the parameter $e:$
\begin{equation}
\label{enorm}
\|e\|_{**}\,\equiv\,\|e\|_{L^{\infty}(0,1)}+\epsilon\|e'\|_{L^2(0,1)}+\epsilon^2\|e{''}\|_{L^2(0,1)}\,\leq\, \epsilon^{ 1/2}.
\end{equation}
The exact forms of the error terms ${\mathcal E}$, $g_0$ and $g_1$ will be given in (\ref{new error-2}) and (\ref{g0}).

\medskip
Then define an operator on the whole strip $\mathcal S$ in the form
\begin{equation}
\label{Lmathcal}
\begin{split}
{\mathcal L}(\phi)\,\equiv\,&\,\beta^{-2}\,\phi_{zz}+\phi_{xx}-\phi+p\,w^{p-1}\,\phi\,
+\,\chi(\e|x|)\,B_3(\phi)
\\
&+\,p\,\big(\beta^{-2}\,\chi(\e|x|)\,{\mathbf W}^{p-1}-w^{p-1}\big)\,\phi\,
\quad{\rm in}~\mathcal S,
\end{split}
\end{equation}
and also the operators
\begin{align}
{\mathcal D}_1(\phi)\,=\,\chi(\e|x|)\, D_3^1(\phi)-\phi_z+\,\chi(\e|x|)\, D_0^1(\phi)
\quad{\rm on}~\partial_1\mathcal S,
\label{D1mathcal}
\\
{\mathcal D}_0(\phi)\,=\,\chi(\e|x|)\, D_3^0(\phi)-\phi_z+\,\chi(\e|x|)\, D_0^0(\phi)
\quad{\rm on}~\partial_0\mathcal S,
\label{D0mathcal}
\end{align}
where $\chi(r)$ is a smooth cut-off function which equals $1$ for $0\leq r<10\delta$ that vanishes identically for $r>20\delta$.
For the local form of the nonlinear part, we have
\begin{align}\label{nonlinearpart}
\eta_{\delta}^{\epsilon}\,(\alpha\beta^2)^{-1}\,
\big[{\mathbf N}({\eta_{3\delta}^{\epsilon}\check\phi}+{\check\psi})
+p\,{\mathbf W}^{p-1}{\check\psi}\big]
=\eta_{\delta}^{\epsilon}\,{\mathcal N}(\phi),
\end{align}
by the notation
\begin{align}\label{Nmathcal}
{\mathcal N}(\phi)=\big(v_5+\phi+\psi\big)^p\,-\,v_5^p\,-\,pv_5^{p-1}\phi,
\end{align}
where in the right hand side the term $\psi$ is transformed from ${\check\psi}$ by the relation (\ref{vdefine}).
Rather than solving problem (\ref{localproblem1})-(\ref{localproblem3}) directly, we deal with the following projected problem: for each pair of parameters $f$ and $e$ in $\mathcal F$,
finding functions $\phi\in H^2(\mathcal S), c, d\in L^2(0,1)$ and constants $l_0, l_1, m_0, m_1$ such that
\begin{equation}
\label{projectedproblem1}
{\mathcal L}(\phi)\,=\,\eta_{\delta}^{\epsilon}(s)\,
\big[-{\mathcal E}-{\mathcal N}(\phi)\,\big]
\,+\,c(\epsilon z)\,\chi(\e|x|)\, w_x
\,+\,d(\epsilon z)\,\chi(\e|x|)\, Z\quad{\rm in}~\mathcal S,
\end{equation}
\begin{equation}
\label{projectedproblem2}
{\mathcal D}_1(\phi)\,=\,-\eta_{\delta}^{\epsilon}(s)\, g_1+l_1\,\chi(\e|x|)\, w_x+m_1 \,\chi(\e|x|)\, Z\quad{\rm on}~\partial_1 \mathcal S,
\end{equation}
\begin{equation}
\label{projectedproblem3}
{\mathcal D}_0(\phi)\,=\,-\eta_{\delta}^{\epsilon}(s)\, g_0+l_0\,\chi(\e|x|)\, w_x+m_0 \,\chi(\e|x|)\, Z\quad{\rm on}~\partial_0 \mathcal S,
\end{equation}
\begin{equation}
\label{projectedproblem4}
\int_{\Bbb R}\phi(x,z)w_x(x)\,{\rm d}x \,=\,\int_{\Bbb R}\phi(x,z)Z(x)\,{\rm d}x \,=\,0,\quad 0<z<\frac {1}{\epsilon}.
\end{equation}
In Proposition \ref{prop}, we will prove that this problem has a unique solution $\phi$ so that  ${\check \phi}$ satisfies the constraint (\ref{decay property}) due the relation (\ref{vdefine}).

\medskip
After this has been done, our task is to adjust the parameters $f$ and $e$  by the reduction method such that the functions $c$ and $d$ are identically zero, and the constants $l_1, l_0, m_1$ and $m_0$ are zero too.
By the estimates in Section \ref{section6}, it is equivalent to solving a nonlocal, nonlinear coupled second order system of differential equations for the pair $(f,e)$ with suitable boundary conditions. In section \ref{section7}, we will prove this system is solvable in~$\mathcal F$.

\section{The local  approximate solutions}
\label{section4}
\setcounter{equation}{0}

The main objective of this section is to construct the approximation $v$ and
then evaluate its error terms ${\mathcal E}$, $g_0$ and $g_1$ in the coordinate system $(x, z)$.

\medskip

\subsection{The first approximate solution}
Let $w(x)$ denote the unique positive solution of (\ref{wsolution}).
Then, we take
$$
v_1(x, z)\,=\,w(x)
$$
as the first approximate solution.
For further improvements, it is of importance to evaluate the errors by plugging the approximate solution to  (\ref{S(v)}), (\ref{errorrelationboundary0}) and (\ref{errorrelationboundary1}).

\medskip
We first consider $S(v_1)$ and identify the terms of order $\epsilon$ and those of order $\epsilon^2:$
\begin{align}
&S(v_1)=S(w)\,=\,B_3(w)
\nonumber
\\
&\,=\,-\beta^{-1}\Big[\epsilon\, k+\epsilon^2\,k^2\,\Big(\frac x \beta+f+h\Big)\Big]w_x
\nonumber
\\
&\quad+\beta^{-2}\Bigg[\epsilon^2\Big|\frac {\beta'}{\beta}x-\beta f'-\beta h'\Big|^2w_{xx}+\epsilon^2\Big(\frac {\beta{''}}{\beta}x-2\beta'f'-2\beta'h'-\beta f{''}-\beta h{''}\Big)w_x\Bigg]
\nonumber
\\
&\quad+\frac {\epsilon^2}{\alpha \beta^2}\alpha{''}w
\,+\,\frac {2\epsilon \alpha'}{\alpha \beta^2}\Bigg[\epsilon \Big(\frac {\beta'}{\beta}x-\beta f'-\beta h'\Big)w_x\Bigg]
\label{S(w)}
\\
&\quad-\frac {2\epsilon\,\varpi}{\alpha \beta^2}\Big(\frac x \beta+f+h\Big)\Big[\epsilon \alpha'\beta w_x+\epsilon \alpha \beta' w_x+\epsilon\alpha \beta' xw_{xx}-\epsilon \alpha \beta^2f'w_{xx}
-\epsilon \alpha \beta^2h'w_{xx}\Big]
\nonumber
\\
&\quad-\frac {\epsilon\,\varpi}{\alpha \beta^2}\Big[\epsilon\alpha'w+\epsilon\alpha \Big(\frac {\beta'}{\beta}x-\beta f'-\beta h'\Big)w_x\Big]
\nonumber
\\
&\quad-\Big[\epsilon \,\sigma^{-1}k\,\Big(\frac x \beta+f+h\Big)+\frac {\epsilon^2}{2\beta^{2}}V_{tt}\Big(\frac x \beta+f+h\Big)^2\Big]w+B_2(w).
\nonumber
\end{align}
Here $B_2(w)$ turned out to be of size $\epsilon^3$.  Gathering terms of order $\epsilon$ and $\epsilon^2$,  we get
\begin{align}
\label{sw-gather}
\begin{aligned}
S(w)\,=\,&-\epsilon\beta^{-1}\Big[kw_x+ \,\sigma^{-1}k\, xw\Big]-\epsilon \,\sigma^{-1}k\,(f+h)w
\\
&-\epsilon^2\Big[\frac {k^2}{\beta}fw_x+\frac {f{''}}{\beta}w_x+\frac {2\beta'}{\beta^2}f'w_x+\frac {2\alpha'}{\alpha \beta}f'w_x
+\frac {2\beta'}{\beta^2}f'xw_{xx}+\frac {V_{tt}}{\beta^3}fxw\Big]
\\
&+\epsilon^2( f'^2\,w_{xx}+2f'h'w_{xx}-\frac{1}{2}\,\beta^{-2}\,V_{tt}\,f^2\,w
+2\,\varpi\,f\,f'\,w_{xx}+2\,\varpi\,f'\,h\,w_{xx})
\\
&-\frac {2\epsilon^2\,\varpi}{\alpha\beta^2}\,\Big[-\alpha \beta f{'}xw_{xx}+\alpha'\beta fw_x+\alpha\beta'fw_x
+\alpha \beta'fxw_{xx}-\frac 12\alpha\beta f'w_x\Big]
\\
&+\frac{\epsilon^2}{\beta^{2}}\Big[(-k^2+\frac {\beta{''}}{\beta}+\frac {2\alpha'\beta'}{\alpha \beta})xw_x+\frac {|\beta'|^2}{\beta^2}x^2w_{xx}+\beta^2\,h'^2\,w_{xx}+\frac {\alpha{''}}{\alpha}w
\\
&\qquad\qquad-\frac {1}{2\beta^2}V_{tt}x^2w-\frac 12V_{tt}(2fh+h^2)w\Big]
\\
&-\epsilon^2\Big[\frac {k^2}{\beta}hw_x+\frac {h{''}}{\beta}w_x+\frac {2\beta'}{\beta^2}h'w_x+\frac {2\alpha'}{\alpha \beta}h'w_x
+\frac {2\beta'}{\beta^2}h'xw_{xx}+\frac {V_{tt}}{\beta^3}hxw\Big]
\\
&-\frac {2\epsilon^2\,\varpi}{\alpha \beta^2}\,\Big[\alpha'xw_x+\frac {3\alpha \beta'}{2\beta}xw_x
+\frac {\alpha \beta'}{\beta}x^2w_{xx}- \alpha \beta^2(fh'+hh')w_{xx}+\frac 12 \alpha'w\Big]
\\
&-\frac {2\epsilon^2\,\varpi}{\alpha\beta^2}\,\Big[-\alpha \beta h{'}xw_{xx}+\alpha'\beta hw_x+\alpha\beta'hw_x
+\alpha \beta'hxw_{xx}-\frac 12\alpha\beta h'w_x\Big]+B_2(w)
\\
\,\equiv\,& \epsilon S_1\,+\,\epsilon S_2\,+\,\epsilon^2S_3\,+\,\epsilon^2S_4\,+\,\epsilon^2S_5
\,+\,\epsilon^2S_6\,+\,\epsilon^2S_7\,+\,\epsilon^2S_8\,+\,\epsilon^2S_9\,+\,B_2(w).
\end{aligned}
\end{align}
Let us observe that the quantities $S_1$, $S_3$, $S_5$, $S_7$ and $S_9$ are odd functions of $x$, while $S_2,~S_4$, $S_6$ and $S_8$ are even functions.

\medskip
For the first approximate solution $v_1$, the boundary errors can be formulated as follows.
For $z=0$,
\begin{equation}
\label{boundary-1-w-x-z}
\begin{split}
\epsilon &\Big [b_5\,x\,+\,\beta (k_1\,f+f')\Big]\,w_x\,-\,\epsilon\, \alpha'\,\alpha^{-1}\,w
\\
&+\epsilon^2\,\Big[b_1\,\Big(\frac x \beta+f+h\Big)^2\, \beta- k\,\Big(\frac x \beta+f+h\Big)\Big(\frac{\beta'}{\beta}\,x-\beta f'-\beta h'\Big)\Big]\,w_x
\\
&-\epsilon^2\,\frac{k\,\alpha'}{\alpha}\,\Big(\frac x \beta+f+h\Big)\,w+D_0^0(w).
\end{split}
\end{equation}
Similarly, for $z\,=\,1/\epsilon$,  we have
\begin{equation}
\label{boundary-2-w-x-z}
\begin{split}
\epsilon &\Big [b_6\,x\,+\,\beta (k_2\,f+f')\Big]\,w_x\,-\,\epsilon \,\alpha'\,\alpha^{-1}\,w
\\
&+\epsilon^2\,\Big[b_3\,\Big(\frac x \beta+f+h\Big)^2 \,\beta- k\,\Big(\frac x \beta+f+h\Big)\Big(\frac{\beta'}{\beta}\,x-\beta\,f'-\beta\,h'\Big)\Big]\,w_x
\\
&-\epsilon^2\,\frac{k\,\alpha'}{\alpha}\Big(\frac x \beta+f+h\Big)\,w+D_0^1(w).
\end{split}
\end{equation}

\subsection{The first improvement}

We now want to construct correction terms and establish a further approximation to a real solution that eliminates the terms of order $\epsilon$ in the errors.

\subsubsection{Interior correction layers}
For any given smooth function $\phi$, it can be checked that
\begin{equation}
\label{S(w+phi1)}
S(w+\phi)\,=\,S(w)+L_0(\phi)+B_3(\phi)+N_0(\phi),
\end{equation}
where
\begin{equation}
\label{L0phi}
L_0(\phi)\,=\,\beta^{-2}\phi_{zz}+\phi_{xx}-\phi+pw^{p-1}\phi,
\end{equation}
and
\begin{equation}
\label{N0phi}
N_0(\phi)\,=\,(w+\phi)^p-w^p-pw^{p-1}\phi.
\end{equation}
By (\ref{sw-gather}), we write
\begin{equation}
\label{S(w+phi)}
\begin{split}
S(w+\phi)\,=\,&\, \Big[\epsilon(S_1+S_2)+\phi_{xx}-\phi+p\,w^{p-1}\phi\Big]+\epsilon^2 S_3+\epsilon^2 S_4+\epsilon^2 S_5+\epsilon^2 S_6
\\
&+\epsilon^2 S_7+\epsilon^2 S_8+\epsilon^2 S_9+B_2(w)+\beta^{-2}\phi_{zz}+B_3(\phi)+N_0(\phi).
\end{split}
\end{equation}

\medskip
We choose $\phi=\phi_1$ in order to eliminate the terms between brackets in the above formula.
Namely, for fixed $z$,  we need a solution to
$$
-\phi_{xx}+\phi-pw^{p-1}\phi\,=\,\epsilon(S_1+S_2)\quad\mbox{in } {\mathbb R},\qquad \phi(\pm \infty)\,=\,0.
$$
It is well-known that this problem is solvable provided that
\begin{equation}
\label{solvable condition}
\int_{\Bbb R}(S_1+S_2)w_x\,{\rm d}x \,=\,0.
\end{equation}
Furthermore, the solution is unique under the constraint
\begin{equation}
\label{constraint}
\int_{\Bbb R}\phi w_x\,{\rm d}x \,=\,0.
\end{equation}
In fact,
\begin{equation}
\label{integral S1S2}
\begin{split}
\int_{\Bbb R}(S_1+S_2)w_x\,{\rm d}x &\,=\,\int_{\Bbb R}S_1w_x\,{\rm d}x
\\
%&\,=\,-\beta^{-1}\Big[k\int_{-\infty}^{\infty}w_x^2\,{\rm d}x +\,\sigma^{-1}k\,\int_{-\infty}^{\infty}xww_x\,{\rm d}x \Big]
%\\
&\,=\,-\beta^{-1}\Big[k\int_{\Bbb R}w_x^2\,{\rm d}x -\frac{1}{2}\,\,\sigma^{-1}k\,\int_{\Bbb R}w^2\,{\rm d}x \Big].
\end{split}
\end{equation}
The validity of the identity
\begin{equation}\label{constant1}
-2\int_{\Bbb R}xww_x\,{\rm d}x\,=\int_{\Bbb R}w^2\,{\rm d}x \,=\,2\sigma \int_{\Bbb R}w_x^2\,{\rm d}x,
\end{equation}
will precisely make the amount between brackets in (\ref{integral S1S2}) identically zero.
Hence,  the solution exists and has the form
 \begin{equation}
 \label{phi1}
 \phi_1\,=\,\epsilon(\phi_{11}+\phi_{12}),
 \end{equation}
where
\begin{equation}
\label{phi11phi12}
\phi_{11}(x, z)\,=\, a_{11}(\epsilon z)w_1(x),
\qquad
\phi_{12}(x, z)\,=\, \big[f(\epsilon z)+h(\epsilon z)\big]a_{12}(\epsilon z)w_2(x),
\end{equation}
with
\begin{equation}
\label{a11a12}
a_{11}(\theta)\,=\,-\frac{k(\theta)}{\beta(\theta)},
\qquad
a_{12}(\theta)\,=\,-\sigma^{-1}k(\theta)\,=\,-\frac{V_t(0, \theta)}{|\beta(\theta)|^2}.
\end{equation}
The function $w_1$ is the unique odd solution to
\begin{equation}
\label{w1}
-w_{1,xx}+w_1-pw^{p-1}w_1\,=\,w_x+\frac {1}{\sigma}xw,\qquad\int_{\Bbb R}w_1w_x\,{\rm d}x \,=\,0,
\end{equation}
and $w_2$ is the unique even function satisfying
\begin{equation}
\label{w2}
-w_{2,xx}+w_2-pw^{p-1}w_2\,=\,w.
\end{equation}
In fact, we can also write
\begin{equation}
\label{w2detail}
w_2\,=\,-\frac {1}{p-1}w-\frac  12xw_x.
\end{equation}

\medskip
Take
$$v_2(x,z)\,=\,w(x)+\phi_1(x,z)$$
as the second approximation.
Substituting $\phi=\phi_1$ into (\ref{S(w+phi)}), we get the new error
\begin{equation}
\label{error S(w+phi1)}
\begin{split}
  S(v_2)\,&=\,\epsilon^2\big(S_3+S_4+S_5+S_6+ S_7+ S_8+ S_9
\big)
\\&\qquad+B_2(w)+\beta^{-2}\phi_{1,zz}+B_3(\phi_1)+N_0(\phi_1).
\end{split}
\end{equation}
Observe that since $\phi_1$ is of size $O(\epsilon)$, all terms above carry $\epsilon^2$ in front. We compute, for instance,
\begin{equation}
\label{B3(phi1)}
  B_3(\phi_1)\,=\,-\epsilon \beta^{-1}\big[k(\phi_1)_x+\,\sigma^{-1}k\,x\phi_1\big]-\epsilon \,\sigma^{-1}k\,(f+h)\phi_1+\epsilon^3a_7({\epsilon}s, {\epsilon}z).
\end{equation}

\subsubsection{The boundary correction layers}

In the following, we are in a position to improve the approximate solution so as to remove the boundary error terms of first order  given in (\ref{boundary-1-w-x-z}) and (\ref{boundary-2-w-x-z}), i.e.,
\begin{equation*}
\epsilon \,b_5\,x\,w_x-\epsilon\,\frac{\alpha'(0)}{\alpha(0)}w
\quad\mbox{ and }\quad
\epsilon\, b_6\,x\,w_x-\epsilon\,\frac{\alpha'(1)}{\alpha(1)}w.
\end{equation*}
Note that we will deal with the boundary terms
$$
\epsilon\,\big(k_1\, \beta \,f+\beta \,f'\big)\,w_x
\quad\mbox{ and }\quad
\epsilon\,\big(k_2\, \beta \,f+\beta \,f'\big)\,w_x
$$
by the standard reduction  procedure.

\medskip
We first introduce a function $\mathfrak{b}({\tilde\theta})$ satisfying
\begin{equation}
\label{bequation}
\epsilon^2 \mathfrak{b}{''}+\lambda_0\mathfrak{b}\,=\,0\quad\mbox{in }(0, {\ell}),
\quad \mathfrak{b}{'}({\ell})\,=\,{\bf{c}_1},
\quad \mathfrak{b}{'}(0)\,=\,\bf{c}_0,
\end{equation}
where $\bf{c}_1$, $\bf{c}_0$ are constants as the following
\begin{equation}
\label{c1col}
\begin{split}
\bf{c}_1\,=\,&\int_{\Bbb R}\Big[b_6xw_x-\frac{\alpha'(1)}{\alpha(1)}w\Big]Z\,{\rm d}x,
\\
\bf{c}_0\,=\,&\int_{\Bbb R}\Big[b_5xw_x-\frac{\alpha'(0)}{\alpha(0)}w\Big]Z\,{\rm d}x.
\end{split}
\end{equation}
Note that the constant $\ell$ is given in (\ref{definenumber}), which can be also written as
\begin{equation}
\begin{split}\label{ellnumber}
{\ell}\,=\,&\int_0^1\beta(\theta)\,{\rm d}\theta,
\end{split}
\end{equation}
where $\beta$ is the function defined in (\ref{alpha-beta}).
\noindent It is easy to calculate that
\begin{equation}
\label{btheta}
 \begin{split}
   \mathfrak{b}({\tilde\theta})
\,\equiv\,&\,\epsilon A({\tilde\theta}),\,\qquad \forall\,{\tilde\theta}\in (0, {\ell})
 \end{split}
\end{equation}
where $A$ has the form
\begin{align}
A({\tilde\theta})\,=\,&\, \frac {{\bf{c}_0}\cos\big(\sqrt{\lambda_0}\,{\ell}/\epsilon\big)-\bf{c}_1}{\sqrt{\lambda_0}\sin\big(\sqrt{\lambda_0}\,{\ell}/\epsilon\big)}\cos\Bigg(\frac{\sqrt{\lambda_0}}{\epsilon}{\tilde\theta}\Bigg)
\,+\,
\frac {\bf{c}_0}{\sqrt{\lambda_0}}\sin\Bigg(\frac{\sqrt{\lambda_0}}{\epsilon}{\tilde\theta}\Bigg).
\end{align}
Hence, we choose
\begin{equation}
\label{tildephi}
\phi_{21}(x,z)\,=\, A\big(\mathfrak{a}(\epsilon z)\big)Z(x),
\end{equation}
where
\begin{equation}\label{a(theta)}
\mathfrak{a}(\theta)\,=\,\int_0^{\theta}\beta(r)\,{\rm d}r.
\end{equation}

\medskip
By Corollary 2.4 in \cite{wei-yang}, we then  find a unique solution  $\phi_*$ of the following problem:
\begin{equation*}
\begin{split}
\Delta \phi_*-\phi_*+pw^{p-1}\phi_*\,=\,0\qquad\quad&{\rm in }~~\hat{S},\qquad\quad
\\
\frac {\partial \phi_*}{\partial {\tilde z}}\,=\,b_6xw_x-\frac{\alpha'(1)}{\alpha(1)}w-{\bf{c}_1}Z\quad&{\rm on }~~\partial_1\hat{\mathcal S},
\\
\frac {\partial \phi_*}{\partial {\tilde z}}\,=\,b_5xw_x-\frac{\alpha'(0)}{\alpha(0)}w-{\bf{c}_0}Z\quad&{\rm on }~~\partial_0\hat{\mathcal S},
\end{split}
\end{equation*} where $\hat{\mathcal S}$, $\partial_0\hat{\mathcal S}$
and $\partial_1\hat{\mathcal S}$
are defined in (\ref{domainS1}).\;
Moreover, $\phi_*$ is even in $x$. \;
By the diffeomophism
$$
{\Upsilon}: [0,  {1}/{\epsilon}]\rightarrow [0, {{\ell}}/{\epsilon}]
$$
in the form
\begin{equation}
\label{mathfrak-a-function}
{\Upsilon}(z)\,=\,\epsilon^{-1}\int_0^{\epsilon z}\beta(r){\rm d}r,
\end{equation}
we define
$$
\phi_{22}(x,z)\,=\,\phi_*\big(x,{\Upsilon}(z)\big).
$$
It is easy to check that
\begin{align*}
\begin{aligned}
\phi_{22,z}(x,z)&\,=\,\beta(\epsilon z)\,\phi_{*,{\tilde z}}\big(x,{\Upsilon}(z)\big),
\\
\\
  \phi_{22,zz}(x,z)&\,=\,|\beta(\epsilon z)|^2\,\phi_{*,{\tilde z}{\tilde z}}\big(x,{\Upsilon}(z)\big)\,+\,\epsilon\beta{'}(\epsilon z)\phi_{*,{\tilde z}}\big(x,{\Upsilon}(z)\big).
\end{aligned}
\end{align*}
Hence, $\phi_{22}$ satisfies the following problem:
  \begin{equation}\label{boundary-1}
  \begin{split}
  \beta^{-2}\phi_{22,zz}+\phi_{22,xx}-\phi_{22}+pw^{p-1}\phi_{22}\,=\,\epsilon \beta^{-2}\beta{'}\phi_{*,\tilde z}\big(x,{\Upsilon}(z)\big)
  \qquad&{\rm in }~~ {\mathcal S},
  \\
  \frac {\partial \phi_{22}}{\partial z}\,=\,\beta(1)\Big[\,b_6xw_x-\frac{\alpha'(1)}{\alpha(1)}w-{\bf{c}_1} Z\,\Big]
  \qquad\qquad\quad&{\rm on }~~\partial_1{\mathcal S},\qquad
\\
\frac {\partial \phi_{22}}{\partial z}\,=\,\beta(0)\Bigg[\,b_5xw_x-\frac{\alpha'(0)}{\alpha(0)}w-{\bf{c}_0} Z\,\Bigg]
\qquad\qquad\quad&{\rm on }~~\partial_0{\mathcal S},\qquad
  \end{split}\end{equation}
where $\mathcal S$, $\partial_0{\mathcal S}$ and $\partial_1{\mathcal S}$ are defined in (\ref{domainS}).

\medskip
We finally set the boundary correction  term
\begin{align}
\begin{aligned}
\label{phi2}
 \phi_2(x,z)\,=\,&\,\epsilon\,\xi(\epsilon z)\,\big(\phi_{21}(x,z)+\phi_{22}(x,z)\big),
\end{aligned}
\end{align}
where
$$
\xi(\theta)\,\equiv\,\frac{\chi_0(\theta)}{\beta(0)}
+\frac{1-\chi_0(\theta)}{\beta(1)},
$$
and the smooth cut-off function $\chi_0$ is defined by
$$
\chi_0(\theta)\,=\,1~~{\rm if}~|\theta|<\frac 12,
\quad {\rm and}\quad
\chi_0(\theta)\,=\,0~~{\rm if}~|\theta|>\frac 34.
$$
Note that $\phi_2(x,z)$ is size of $O(\epsilon)$ under the gap condition
(\ref{gapcondition}).

\medskip
Let $v_3(x,z)\,=\,w+\phi_1+\phi_2$
 be the third approximate solution. A careful computation can indicate that
 the new boundary error takes the following form
 \begin{align}
 &\epsilon \,\beta\,\big(k_1\,f+f'\big)\,w_x
 \nonumber
 \\
  &+\epsilon^2\big(b_5x+k_1\beta f+\beta f'\big) \Big(a_{11}w_{1,x}+a_{12}(f+h)w_{2,x}+\beta^{-1}\big(A(0)Z_x+\phi_{22,x}\big)\Big)
  \nonumber
  \\
 &-\epsilon^2\,\alpha'\,\alpha^{-1}\Big\{a_{11}\,w_1+a_{12}\,(f+h)\,w_2
 +\beta^{-1}\,\big[A(0)\,Z+\phi_{22}\big]\Big\}
  \nonumber
  \\
&-\epsilon^2\,\frac {k\alpha'}{\alpha}\Big(\frac{x}{\beta}+f+h\Big)\,\Big[w+\epsilon\, a_{11}\,w_1+\epsilon\, a_{12}\,(f+h)\,w_2+\epsilon\,\beta^{-1}\,\Big( A(0)\,Z+\phi_{22}\Big)\Big]
  \nonumber
  \\
 &+\epsilon^2\,\Big[b_1\,\Big(\frac x \beta +f+h\Big)^2 \,\beta-k\,\Big(\frac x \beta +f+h\Big)\Big(\frac{\beta'}{\beta}\,x-\beta\, f'-\beta\, h'\Big)\Big]
  \nonumber
  \\
 &\quad\times \Big[w_x+\epsilon\, a_{11}\,w_{1,x}+\epsilon \,a_{12}\,(f+h)\,w_{2,x}+\epsilon\, \beta^{-1}\,\big(A(0)\,Z_x+\phi_{22,x}\big)\Big]
  \nonumber
  \\
 &-\epsilon^2\,\Big[a'_{11}\,w_1+a'_{12}\,(f+h)\,w_2+a_{12}\,(f'+h')\,w_2\Big] \label{second-approximated-boundary-new-1}
 \\
 &-\epsilon^2\,k\Big(\frac{x}{\beta}+f+h\Big)\times \Big[\epsilon\, a'_{11}\,w_1+\epsilon\, a'_{12}\,(f+h)\,w_2+\epsilon\, a_{12}\,(f'+h')\,w_2
  \nonumber
 \\
 &\qquad\qquad\qquad\qquad\qquad\qquad\qquad\qquad+\beta^{-1}\,\Big(\epsilon \, A'(0)\,\beta(0)\,Z+\phi_{22,z}\Big)\Big]
  \nonumber
  \\
 &+\epsilon^2\,b_2\,\Big(\frac{x}{\beta}+f+h\Big)^2\times\Big[\epsilon^2 \,a'_{12}\,w_1+\epsilon^2\, a'_{12}\,(f +h)\,w_2+\epsilon^2\, a_{12}\,(f'+h')\,w_2
  \nonumber
 \\
& \qquad\qquad\qquad\qquad\qquad\qquad\qquad\qquad+\epsilon \,\beta^{-1}\,\big(\epsilon\, A'(0)\,\beta(0)\,Z+\phi_{22,z}\big)\Big]
 \nonumber
\\
&+D_0^0(w+\phi_1+\phi_2)\quad {\rm on } ~\partial S_0.
 \nonumber
 \end{align}
Similar estimate holds on  $\partial S_1$.

\subsection{The second improvement}

To decompose the coupling of the parameters $f$ and $e$ on the boundary of $\mathcal S$ (in the sense of projection against $Z$ in $L^2$), by Lemma 2.2 in \cite{wei-yang}, we
introduce a new term $\phi^{*}$ (even in $x$) defined by the following problem
$$
\Delta \phi^{*}-\tilde{K}\phi^{*}+p\,w^{p-1}\phi^{*}\,=\,0\quad{\rm in}~\hat{\mathcal S},
$$
$$
\phi^{*}_{{\tilde z}}\,=\,H_1(x)\quad {\rm on }~~\partial_0\hat{\mathcal S},
\qquad
\phi^{*}_{{\tilde z}}\,=\,H_2(x)\quad {\rm on }~~\partial_1\hat{\mathcal S},
$$
where $\tilde{K}$ is a large positive constant, the functions
$H_1(x)$ and $H_2(x)$ are given by the following
\begin{align}
H_1(x)\,=\,&\,2b_1\big[f(0)+h(0)\big]x\,w_x
\,-\,k(0)\Big[\frac{\beta'(0)}{\beta(0)}\big(f(0)+h(0)\big)-\big(f'(0)+h'(0)\big)\Big]x\,w_x
\nonumber
\\
&\,+\,b_5\,x\Big[a_{12}(0)\,\big(f(0)+h(0)\big)\,w_{2,x}+\beta^{-1}(0)\big(A(0)\,Z_x+\phi_{22,x}(x,0)\big)\,\Big]
\nonumber
\\
&\,+\,\Big[k_1\,\beta(0)\,f(0)+\beta(0)\,f'(0)\Big]\,a_{11}(0)\,w_{1,x}-\,k(0)\,\big(f(0)+h(0)\big)\frac{\alpha'(0)}{\alpha(0)}\,w
\nonumber
\\
&\,-\,\frac{\alpha'(0)}{\alpha(0)}\Big[a_{12}(0)\,(f(0)+h(0))\,w_2
+\beta^{-1}(0)\Big(A(0)\,Z+\phi_{22}(x,0)\Big)\Big]
\label{H1new}
\\
&\,-\,[f'(0)+h'(0)]\,a_{12}(0)\,w_2\,-\,[f(0)+h(0)]\,a'_{12}(0)\,w_2
\nonumber
\\
&\,-\,k(0)\,\beta^{-1}(0)\,\big[f(0)+h(0)\big]\cdot\big[\epsilon \, A'(0)\,\beta(0)\,Z+\phi_{22,z}(x,0)\big],
\nonumber
\end{align}
and
\begin{align}
H_2(x)\,=\,&\, 2b_3\big[f(1)+h(1)\big]x\,w_x
\,-\,k(1)\Big[\frac{\beta'(1)}{\beta(1)}\big (f(1)+h(1)\big)-\big(f'(1)+h'(1)\big)\Big]x\,w_x
\nonumber
\\
&\,+\,b_6\,x\Big[a_{12}(1)\,\big(f(1)+h(1)\big)\,w_{2,x}+\beta^{-1}(1)\big(A({\ell})\,Z_x
+\phi_{22,x}(x,1/\epsilon)\big)\,\Big]
\nonumber
\\
&\,+\,\Big[k_2\,\beta(1)\,f(1)+\beta(1)\,f'(1)\Big]\,a_{11}(1)\,w_{1,x}
-\,k(1)\,\big(f(1)+h(1)\big)\frac{\alpha'(1)}{\alpha(1)}\,w
\nonumber
\\
&\,-\,\frac{\alpha'(1)}{\alpha(1)}\Big[a_{12}(1)\,\big(f(1)+h(1)\big)\,w_2
+\beta^{-1}(1)\Big(A({\ell})\,Z
+\phi_{22}(x,1/\epsilon)\Big)\Big]
\label{H2new}
\\
&\,-\,\big[f'(1)+h'(1)\big]\,a_{12}(1)\,w_2\,-\,\big[f(1)
+h(1)\big]\,a'_{12}(1)\,w_2\,
\nonumber
\\
&\,-\,k(1)\beta^{-1}(1)\,\big[f(1)+h(1)\big]\cdot\big[\epsilon \, A'({\ell})\,\beta(1)\,Z+\phi_{22,z}(x,1/\epsilon)\big].
\nonumber
\end{align}
We define a boundary correction term again
\begin{equation}
\label{boundary layer}
\phi_3(x, z)\,=\,\epsilon^2\xi(\epsilon z)\phi^{*}(x,{\Upsilon}(z))\,\equiv\, \epsilon^2 \hat{\phi}(x, z).
\end{equation}
Note that $\phi_3$ is an exponential decaying function which is of order $\epsilon^2$ and even in the variable $x$.

\medskip
To deal with the resonance phenomena and improve the approximation for a solution still keeping the term of $\epsilon^2$, we need to introduce a new parameter $e$ in additional to $f$ and $h$,
and define the fourth approximate solution to the problem near $\Gamma_\epsilon$ as
\begin{equation}
\label{u3define}
v_4(x, z)\,=\,w+\phi_1+\phi_2+\epsilon e(\epsilon z)Z(x)+\phi_3.
\end{equation}

\subsection{The third improvement}
We want to construct a further approximation to a solution that eliminates the term of $h$, $h'$
and the even terms in the error
$$
S(w+\phi_1+\phi_2+\epsilon eZ+\phi_3).
$$
This can be fulfilled by adding a term ${\tilde\phi}$ and then considering the following term
\begin{equation}
\label{w+phi1+phi2+eeZ+phi3+Phi}
\begin{split}
S&(w+\phi_1+\phi_2+\epsilon eZ+\phi_3+{\tilde\phi})
\\[1mm]
\,=\,&S(w)+L_0(\phi_1+\phi_2+\epsilon eZ+\phi_3+{\tilde\phi})+B_3(\phi_1)+B_3(\phi_2)
\\[1mm]
&+B_3(\epsilon e Z)+B_3(\phi_3)+B_3({\tilde\phi})+N_{0}(\phi_1+\phi_2+\epsilon eZ+\phi_3+{\tilde\phi})
\\[1mm]
\,=\,&S(w)+ \epsilon(\epsilon^2\beta^{-2}e{''}Z+\lambda_0 eZ)+L_0(\phi_1)+L_0(\phi_2)+L_0(\phi_3)+L_0({\tilde\phi})
\\[1mm]
&+B_3(\phi_1)+B_3(\phi_2)+B_3(\epsilon e Z)
+B_3(\phi_3)+B_3({\tilde\phi})+N_{0}(\phi_1+\phi_2+\epsilon eZ+\phi_3+{\tilde\phi}).
\end{split}
\end{equation}
The details will be given in the sequel.

\subsubsection{Rearrangements of the error components}
The first objective of this part is to compute the terms in (\ref{w+phi1+phi2+eeZ+phi3+Phi}).
It is easy to compute that
\begin{equation*}
  L_0(\phi_1)\,=\,\beta^{-2}\phi_{1,zz}+\phi_{1,xx}-\phi_1+pw^{p-1}\phi_1\,=\,\beta^{-2}\phi_{1,zz}-\epsilon(S_1+S_2).
\end{equation*}

\medskip
Recall the expression of $\phi_2$ and $A(\tilde{\theta})$ defined in (\ref{phi2}) and (\ref{btheta}).
By using the equation of $\phi_{22}$ in (\ref{boundary-1}) and the equation of $Z$ in (\ref{eigenvalue}),
we get
 \begin{equation}
 \label{L0phi2}
 \begin{split}
  L_0(\phi_2)&\,=\,\beta^{-2}\,\phi_{2,zz}+\phi_{2,xx}-\phi_2+p\,w^{p-1}\,\phi_2
%\\
%  \,&=\,\Bigg\{\frac{2\epsilon^{2}}{\beta^{2}}\xi'(\epsilon z)\big[\epsilon A'(\mathfrak{a}(\epsilon z))\beta %Z+\phi_{22,z}\big]+\frac{\epsilon^2\beta'}{\beta^{2}}\xi(\epsilon z)\big[\epsilon A'(\mathfrak{a}(\epsilon z)) %Z+\phi_{*,\tilde{z}}\big]\Bigg\}
%\\
%&\quad+\frac{\epsilon^3}{\beta^{2}}\xi''(\epsilon z)\Big[A(\mathfrak{a}(\epsilon z))Z+\phi_{22}\Big]
\\
&=M_{11}(x,z)+M_{12}(x,z),
  \end{split}
 \end{equation}
where
 \begin{equation}
 \label{hatL0phi2}
M_{11}(x,z)=\frac{2\epsilon^{2}}{\beta^{2}}\xi'(\epsilon z)\big[\epsilon A'(\mathfrak{a}(\epsilon z))\beta Z+\phi_{22,z}\big]+\frac{\epsilon^2\beta'}{\beta^{2}}\xi(\epsilon z)\big[\epsilon A'(\mathfrak{a}(\epsilon z)) Z+\phi_{*,\tilde{z}}\big],
 \end{equation}
 and
 \begin{equation}
 \label{tildeL0phi2}
   M_{12}(x,z)=\frac{\epsilon^3}{\beta^{2}}\xi''(\epsilon z)\big[A(\mathfrak{a}(\epsilon z))Z+\phi_{22}\big].
 \end{equation}

\medskip
According to the fact that $\phi^{*}$ satisfies
$$
\Delta \phi^{*}-\tilde{K}\phi^{*}+pw^{p-1}\phi^{*}\,=\,0\quad{\rm in}~\hat{\mathcal S},
$$
and
$$
\phi_3\,=\,\epsilon^2\,\hat{\phi}\,=\,\epsilon^2\,\xi(\epsilon z)\,\phi^*,
$$
it follows that
 \begin{equation}
 \label{L0phi3}
 \begin{split}
  L_0(\phi_3)&=\,\epsilon^2\,\big[\beta^{-2}\,\hat{\phi}_{zz}+\hat{\phi}_{xx}-\hat{\phi}+p\,w^{p-1}\,\hat{\phi}\big]
 \\[1mm]
 &=\,\epsilon^2\,(\tilde{K}-1)\hat{\phi}+\epsilon^2\,\big[\beta^{-2}\,\hat{\phi}_{zz}+\hat{\phi}_{xx}-\tilde{K}\,\hat{\phi}+p\,w^{p-1}\,\hat{\phi}\big]
% \\[1mm]
% &=\,\epsilon^2(\tilde{K}-1)\hat{\phi}+\frac{\epsilon^4}{\beta^{2}}\xi''(\epsilon  z)\phi^*(x,{\Upsilon}(z))+\frac{2\epsilon^3}{ \beta}\xi'(\epsilon z)
%\phi^*_{{\tilde z}}(x,{\Upsilon}(z))
%\\
%&\quad+\frac{\epsilon^3\beta'}{\beta^{2}}\xi(\epsilon z)\phi^*_{{\tilde z}}(x,{\Upsilon}(z))
\\[1mm]
&=M_{21}(x,z)+M_{22}(x,z),
 \end{split}
\end{equation}
where
\begin{equation}
\label{hatL0e^2hatphi}
M_{21}(x,z)\,=\,\epsilon^2\,(\tilde{K}-1)\,\hat{\phi},
\end{equation}
and
\begin{equation}
\label{tildeL0e^2hatphi}
M_{22}(x,z)=\frac{\epsilon^4}{\beta^{2}}\xi''(\epsilon z)\phi^*(x,{\Upsilon}(z))+\frac{2\epsilon^3}{ \beta}\xi'(\epsilon z)
\phi^*_{{\tilde z}}(x,{\Upsilon}(z))+\frac{\epsilon^3\beta'}{\beta^{2}}\xi(\epsilon z)\phi^*_{{\tilde z}}(x,{\Upsilon}(z)).
\end{equation}

\medskip
Recall the expression of $B_3(\phi_1)$ in (\ref{B3(phi1)}), we obtain that
\begin{equation}
\label{Bphi1}
\begin{split}
B_3(\phi_1)&=-\epsilon\, \beta^{-1}\big[k(\phi_1)_x+\,\sigma^{-1}k\,x\phi_1\big]-\epsilon \,\sigma^{-1}k\,(f+h)\,\phi_1+\epsilon^3\,a_7({\epsilon}s,{\epsilon}z )
%\\
%&=\frac{\epsilon^2k^2}{\beta\sigma}\,\Big[\frac{\sigma}{\beta}w_{1,x}+hw_{2,x}
%+\frac{1}{\beta}xw_{1}+\frac{1}{\sigma}hxw_{2}+
%hw_{1}+\frac{\beta }{\sigma}h^2w_2+\frac{2\beta }{\sigma}fhw_2\Big]
%\\
%&\qquad\quad+\frac{\epsilon^2k^2}{\beta\sigma}\,\big[fw_{2,x}
%+\sigma^{-1}fxw_{2}+fw_{1}+\sigma^{-1}\beta f^2w_2\big]+\epsilon^3a_7(\epsilon s,\epsilon z)
\\
&=M_{31}(x,z)+M_{32}(x,z),
\end{split}
\end{equation}
where
\begin{equation}
\label{M31}
M_{31}(x,z)=\frac{\epsilon^2k^2}{\beta\sigma}\,\Big[\frac{\sigma}{\beta}w_{1,x}+hw_{2,x}
+\frac{1}{\beta}xw_{1}+\frac{1}{\sigma}hxw_{2}+
hw_{1}+\frac{\beta }{\sigma}h^2w_2+\frac{2\beta }{\sigma}fhw_2\Big],
\end{equation}
and
\begin{equation}
\label{M32}
M_{32}(x,z)=\frac{\epsilon^2k^2}{\beta\sigma}\,\big[fw_{2,x}
+\sigma^{-1}fxw_{2}+fw_{1}+\sigma^{-1}\beta f^2w_2\big]+\epsilon^3a_7(\epsilon s,\epsilon z).
\end{equation}

\medskip
What's more, we can decompose $B_3(\phi_2)$ as following
\begin{equation*}
%\label{B3(phi2)}
  \begin{split}
    B_3(\phi_2) &=-\epsilon\beta^{-1}k\phi_{2,x}+2\epsilon\beta^{-2} \Big(\frac {\beta'}{\beta}x-\beta f'-\beta h'\Big)\phi_{2,xz}+\frac {2\epsilon \alpha'}{\alpha \beta^2}\phi_{2,z}-\frac {2\epsilon\,\varpi}{ \beta}\Big(\frac x \beta+f+h\Big)\, \phi_{2,xz}
    \\
      &\quad-\frac {\epsilon\,\varpi}{\beta^2}\phi_{2,z} -\epsilon \,\sigma^{-1}k\,\Big(\frac x \beta+f+h\Big)\phi_2+\epsilon^3a_8(\epsilon s,\epsilon z)
      %\\
     % &=-\frac{\epsilon^2}{\beta}\,k\,\xi(\epsilon z)\big[A(\mathfrak{a}(\epsilon z))Z_x+\phi_{22,x}\big]+\frac{2\,\epsilon^2}{\beta^{2}}\Big(\frac{\beta'}{\beta}x-\beta h'\Big)\xi(\epsilon z)\big[\epsilon A'(\mathfrak{a}(\epsilon z))\beta Z_x+\phi_{22,xz}\big]
     % \\
     % &
     % \quad+\frac {2\epsilon^2 \alpha'}{\alpha \beta^2}\xi(\epsilon z)\big[\epsilon A'(\mathfrak{a}(\epsilon z))\beta Z+\phi_{22,z}\big]
     % -\frac {2\epsilon^2\varpi}{\beta}\Big(\frac x \beta+h\Big)\,\xi(\epsilon z)\big[\epsilon A'(\mathfrak{a}(\epsilon z))\beta Z_x+\phi_{22,xz}\big]
     % \\
     % &\quad-\frac{\epsilon^2\varpi}{\beta^{2}} \,\xi(\epsilon z)\big[\epsilon A'(\mathfrak{a}(\epsilon z))\beta Z+\phi_{22,z}\big]-\epsilon^2\,\sigma^{-1}k\Big(\frac{x}{\beta}+f+h\Big)\xi(\epsilon z)\big[A(\mathfrak{a}(\epsilon z)) Z+\phi_{22}\big]
      \\
      &= M_{41}(x,z)+M_{42}(x,z),
  \end{split}
\end{equation*}
where
\begin{equation*}
%\label{M41}
  \begin{split}
     M_{41}(x,z)\,&=-\frac{\epsilon^2}{\beta}\,k\,\xi(\epsilon z)\big[A(\mathfrak{a}(\epsilon z))Z_x+\phi_{22,x}\big]+\frac{2\,\epsilon^2}{\beta^{2}}\Big(\frac{\beta'}{\beta}x-\beta h'\Big)\xi(\epsilon z)\big[\epsilon A'(\mathfrak{a}(\epsilon z))\beta Z_x+\phi_{22,xz}\big]
      \\
      &
      \quad+\frac {2\epsilon^2 \alpha'}{\alpha \beta^2}\xi(\epsilon z)\big[\epsilon A'(\mathfrak{a}(\epsilon z))\beta Z+\phi_{22,z}\big]
      -\frac {2\epsilon^2\varpi}{\beta}\Big(\frac x \beta+h\Big)\,\xi(\epsilon z)\big[\epsilon A'(\mathfrak{a}(\epsilon z))\beta Z_x+\phi_{22,xz}\big]
      \\
      &\quad-\frac{\epsilon^2\varpi}{\beta^{2}} \,\xi(\epsilon z)\big[\epsilon A'(\mathfrak{a}(\epsilon z))\beta Z+\phi_{22,z}\big]-\epsilon^2\,\sigma^{-1}k\Big(\frac{x}{\beta}+f+h\Big)\xi(\epsilon z)\big[A(\mathfrak{a}(\epsilon z)) Z+\phi_{22}\big],
   \end{split}
\end{equation*}
and
\begin{equation*}
%\label{M42}
M_{42}(x,z)=-\,\frac{2\epsilon^2}{\beta}\,\big(f'+\varpi\,f\big)\,\xi(\epsilon z)\,\big[\epsilon\,A'(\mathfrak{a}(\epsilon z))\,\beta \,Z_x+\phi_{22,xz}\big]
+\epsilon^3a_8(\epsilon s,\epsilon z).
\end{equation*}

\medskip
The computation for next term is the following:
\begin{equation}
\label{ B3(epsiloneZ)}
\begin{split}
  B_3(\epsilon eZ)&=-\epsilon^2\,\beta^{-1}\,k\,e\,Z_{x}-\epsilon^2\sigma^{-1}k\,\Big(\frac{x}{\beta}+f+h\Big)\,e\,Z+\epsilon^3a_9(\epsilon s,\epsilon z)
  %\\
 %&=-\epsilon^2\sigma^{-1}k\,(f+h)\,e\,Z+\Big\{-\epsilon^2\,\beta^{-1}\,k\,e\,Z_{x}-\epsilon^2\sigma^{-1}\beta^{-1}k\,x\,e\,Z
 % +\epsilon^3a_9(\epsilon s,\epsilon z)\Big\}
  \\[1mm]
  &= M_{51}(x,z)+M_{52}(x,z),
  \end{split}
\end{equation}
where
\begin{equation}
\label{M51}
 M_{51}(x,z)=-\epsilon^2\sigma^{-1}k\,(f+h)\,e\,Z,
\end{equation}
and
\begin{equation}
\label{M52}
M_{52}(x,z)=-\epsilon^2\,\beta^{-1}\,k\,e\,Z_{x}-\epsilon^2\sigma^{-1}\beta^{-1}k\,x\,e\,Z
  +\epsilon^3a_9(\epsilon s,\epsilon z).
\end{equation}

\medskip
The main order of the nonlinear term
\begin{equation}
\label{N0()}
\begin{split}
N_0(\phi_1+\phi_2+\epsilon eZ+\phi_3+{\tilde\phi})\,=\,&(w+\phi_1+\phi_2+\epsilon eZ+\phi_3+{\tilde\phi})^p-w^p
\\
&-pw^{p-1}(\phi_1+\phi_2+\epsilon eZ+\phi_3+{\tilde\phi}),
\end{split}
\end{equation}
is in the following
$$
\frac {p(p-1)}{2}w^{p-2}(\phi_1+\phi_2+\epsilon eZ+\phi_3+{\tilde\phi})^2,
$$
which can be decomposed as following
\begin{align*}
   &\quad\frac {p(p-1)}{2}w^{p-2}(\phi_1+\phi_2+\epsilon eZ+\phi_3+{\tilde\phi})^2
   \\
   & =\frac {p(p-1)}{2}w^{p-2}\Big[\phi_1^2+\phi_2^2+\epsilon^2e^2Z^2+\epsilon^4\hat{\phi}^2+{\tilde\phi}^2+2\,\phi_1\phi_2+2\,\epsilon \phi_1eZ+2\,\epsilon^2\phi_1\hat{\phi}
   \\
   &\qquad\qquad\qquad\qquad+2\,\phi_1{\tilde\phi}+2\,\epsilon \phi_2eZ+2\,\epsilon^2\phi_2\hat{\phi}+2\,\phi_2{\tilde\phi}+2\,\epsilon^3 eZ\hat{\phi}+2\,\epsilon eZ{\tilde\phi}+2\,\epsilon^2\hat{\phi}{\tilde\phi}\Big]
    \\
    &=\epsilon^2\frac {p(p-1)}{2}w^{p-2}\Big[a_{11}^2w_{1}^2+2a_{11}a_{12}hw_{1}w_{2}
    +a_{12}^2\big(2fh+h^2\big)w_{2}^2+\xi(\epsilon z)^2\big(A(\mathfrak{a}(\epsilon z)Z+\phi_{22})^2
    \\[1mm]
    &\qquad\qquad\qquad\qquad\quad+e^2Z^2+2 a_{11}w_1\xi(\epsilon z)\big(A(\mathfrak{a}(\epsilon z)Z+\phi_{22}\big)+2 a_{12}(f+h)\,w_2eZ
    \\[1mm]
    &\qquad\qquad\qquad\qquad\quad+2\,a_{12}(f+h)w_2\xi(\epsilon z)\big(A(\mathfrak{a}(\epsilon z)Z+\phi_{22}\big)
+2\,\xi(\epsilon z)\big(A(\mathfrak{a}(\epsilon z)Z+\phi_{22}\big)e Z\Big]
 \\[1mm]
  &\quad +\frac {p(p-1)}{2}w^{p-2}\Big[2\,\epsilon^2\,a_{11}\,a_{12}\,w_1\,w_2\,f
  +\epsilon^2\,a_{12}^2\,w_{2}^2\,f^2+2\,\epsilon^2\,a_{11}w_1 eZ+\epsilon^4\,\hat{\phi}^2+{\tilde\phi}^2+2\,\epsilon^2\,\phi_1\,\hat{\phi}
  \\[1mm]
  &\qquad\qquad\qquad\qquad\quad+2\,\phi_1\,{\tilde\phi}+2\,\epsilon^2\,\phi_2\,\hat{\phi}+2\,\phi_2\,{\tilde\phi}+2\,\epsilon^3\, e\,Z\,\hat{\phi}+2\, \epsilon \,e\,Z\,{\tilde\phi}+2\,\epsilon^2 \,\hat{\phi}\,{\tilde\phi}\Big]
  \\[1mm]
  &\equiv M_{61}(x,z)\,+\,{\tilde M}_{62}(x,z).
\end{align*}
For the convenience of notation, we also denote
\begin{align}
M_{62}(x,z)=N_0(\phi_1+\phi_2+\epsilon eZ+\phi_3+{\tilde\phi})-M_{61}(x, z).
\end{align}

\medskip
Whence, according to the above rearrangements, we rewrite the expression (\ref{w+phi1+phi2+eeZ+phi3+Phi}) in terms of
\begin{equation}
\label{S(w+phi_1+phi_2+varphi)}
\begin{split}
  &S(w+\phi_1+\phi_2+\epsilon\,e\,Z+\phi_3+{\tilde\phi})
  \\
  &=\Big[\epsilon^2S_6+\epsilon^2S_7+\epsilon^2S_8+\epsilon^2S_9+M_{11}(x,z)+M_{21}(x,z)+M_{31}(x,z)
  \\
  &\quad\qquad+M_{41}(x,z)+M_{51}(x,z)+M_{61}(x,z)+{\tilde\phi}_{xx}-{\tilde\phi}+pw^{p-1}{\tilde\phi}\Big]
  \\[1.3mm]
   &\quad+\epsilon^2S_3+\epsilon^2S_4+\epsilon^2S_5+B_2(w)+\beta^{-2}\phi_{1,zz}+\beta^{-2}{\tilde\phi}_{zz}
   +\epsilon(\epsilon^2\beta^{-2}e{''}Z+\lambda_0 eZ)
   \\[1.5mm]
   &\quad+M_{12}(x,z)+M_{22}(x,z)+M_{32}(x,z)+M_{42}(x,z)+M_{52}(x,z)
   \\
   &\quad+M_{62}(x,z)+B_{3}(\phi_3)+B_3({\tilde\phi}).
\end{split}
\end{equation}

\subsubsection{Finding new correction terms and defining the basic approximation}
We now choose ${\tilde\phi}=\phi_4$ in order to eliminate the terms between the first two brackets in (\ref{S(w+phi_1+phi_2+varphi)}). Namely, for fixed $z$, we need a solution to the problem
\begin{equation}
\label{equation of Phi}
\begin{split}
-{\tilde\phi}_{xx}+{\tilde\phi}-pw^{p-1}{\tilde\phi}=\,&\epsilon^2S_6+\epsilon^2S_7+\epsilon^2S_8+\epsilon^2S_9+M_{11}(x,z)+M_{21}(x,z)
\\
&+M_{31}(x,z)+M_{41}(x,z)+M_{51}(x,z)+M_{61}(x,z),
\quad
\forall x\in {\mathbb R}.
\end{split}
\end{equation}
It is well-known that this problem is solvable provided that
\begin{equation}
\label{solvable of varphi1}
\begin{split}
\int_{\Bbb R}\Big[\epsilon^2&S_6+\epsilon^2S_7+\epsilon^2S_8+\epsilon^2S_9+M_{11}(x,z)+M_{21}(x,z)
\\
&\quad+M_{31}(x,z)+M_{41}(x,z)+M_{51}(x,z)+M_{61}(x,z)\Big]w_x\,{\rm d}x \,=\,0.
\end{split}
\end{equation}

\medskip
The computation of (\ref{solvable of varphi1}) can be showed as follows.
Since $S_6$, $S_8$, $M_{11}(x,z)$, $M_{21}(x,z)$, $M_{51}(x,z)$ are even functions, integration against $ w_x$ therefore just vanish.
This gives that
\begin{equation*}
\begin{split}
&\int_{\Bbb R}\Big[\epsilon^2S_6+\epsilon^2S_7+\epsilon^2S_8+\epsilon^2S_9+M_{11}(x,z)+M_{21}(x,z)
\\
&\qquad\qquad+M_{31}(x,z)+M_{41}(x,z)+M_{51}(x,z)+M_{61}(x,z)\Big]w_x\,{\rm d}x \,
 \\
&=\int_{\Bbb R}\Big[\epsilon^2\,S_7\,+\,\epsilon^2\,S_9\,+\,M_{31}(x,z)\,+\,M_{41}(x,z)\,+\,M_{61}(x,z)\,\Big]\,w_x\,{\rm d}x
\\
&\equiv J_1\,+\,J_2\,+\,J_3\,+\,J_4\,+\,J_5.
\end{split}
\end{equation*}
Recalling the expression of $S_7$ in (\ref{sw-gather}), direct computation leads to
\begin{equation}
\label{J1}
\begin{split}
J_{1}\,=\,&
-\epsilon^2\beta^{-1}h{''}\int_{\Bbb R}w_x^2\,{\rm d}x
-2\epsilon^2\beta^{-2}\beta'h{'}\int_{\Bbb R}\big(w_x^2+x w_xw_{xx}\big)\,{\rm d}x
\\[1mm]
&-2\epsilon^2\alpha^{-1}\beta^{-1}\alpha' h'\int_{\Bbb R}w_x^2\,{\rm d}x-\epsilon^2 h\Big(\beta^{-1}k^2\int_{\Bbb R}w_x^2\,{\rm d}x+\beta^{-3}V_{tt}\int_{\Bbb R}xww_x\,{\rm d}x\Big)
\\[1mm]
\,=\,&-\epsilon^2\varrho_1\beta^{-1}h{''}
\,-\,
\epsilon^2\varrho_1\beta^{-1}\big(\beta^{-1}\beta'+2\alpha^{-1}\alpha'\big)h{'}
-\epsilon^2\varrho_1\beta^{-1}\Big(k^2-\sigma \beta^{-2}V_{tt}\Big)h,
\end{split}
\end{equation}
where we have used (\ref{constant1}) and the following relation
\begin{equation}
\label{constant2}
 \varrho_{1}\,\equiv\,\int_{\Bbb R}w_x^2\,{\rm d}x\,=\,-2\int_{\Bbb R}xw_xw_{xx}\,{\rm d}x.
\end{equation}
According to the definition of $S_9$, it follows that
\begin{equation}
\label{J2}
\begin{split}
J_{2}\,=\,&2\,\epsilon^2\,\beta^{-1}\,h'\,\varpi\,\int_{\Bbb R}\Big(xw_{xx}w_x+\frac 12 w_x^2\Big)\,{\rm d}x-2\,\epsilon^2\,\alpha^{-1}\,\alpha'\,\beta^{-1}\,h\,\varpi\,\int_{\Bbb R}w_x^2\,{\rm d}x
\\
&\qquad-2\,\epsilon^2\,\beta'\,\beta^{-2}\,h\,\varpi\,\int_{\Bbb R}\big(w_x^2+xw_{xx}w_x\big)\,{\rm d}x
\\
\,=\,&-\epsilon^2 \varrho_{1}\beta^{-1}\Big[\,\varpi\,\beta'\,\beta^{-1}+2\,\varpi\,\alpha^{-1}\,\alpha'\Big]\,h.
\end{split}
\end{equation}
Since $w_1$ is an odd function and $w_2$ is a even function, we obtain
\begin{equation}
\label{J3}
\begin{split}
J_3=\int_{\Bbb R}\,M_{31}(x,z)w_x\,{\rm d}x
=&\,\epsilon^2\beta^{-1}\sigma^{-1}k^2\int_{\Bbb R}\Big[hw_{2,x}+\sigma^{-1}hxw_{2}+
hw_{1}\Big]w_x\,{\rm d}x
\\
=&\,\epsilon^2\beta^{-1}\sigma^{-1}k^2\,h\int_{\Bbb R}\Big[w_{2,x}w_x+\frac {1}{\sigma}xw_2w_x+w_1w_x\Big]\,{\rm d}x.
\end{split}
\end{equation}
By the definition of $M_{41}(x,z)$, we get
\begin{equation}
\label{J4}
  \begin{split}
    J_4&=\int_{\Bbb R}\,M_{41}(x,z)w_x\,{\rm d}x
   % \\
   % &=\,-2\,\epsilon^2\,h'\,\beta^{-1}\xi(\epsilon z)\int_{\Bbb R}\Big[\epsilon\, A'\big(\mathfrak{a}(\epsilon z)\big)\,\beta (\epsilon z)\,Z_x+\phi_{22,xz}(x,z)\Big]w_x\,{\rm d}x
   % \\
  %  &\quad-2\,\epsilon^2\,h\,\beta^{-1}\,\varpi\,\xi(\epsilon z)\int_{\Bbb R}\Big[\epsilon\, A'\big(\mathfrak{a}(\epsilon z)\big)\,\beta (\epsilon z)\,Z_x+\phi_{22,xz}(x,z)\Big]w_x\,{\rm d}x
%\\
%&\quad-\epsilon^2\,\beta^{-1}\,k\,\xi(\epsilon z)\int_{\Bbb R}\Big[A\big(\mathfrak{a}(\epsilon z)\big)\,Z_x+\phi_{22,x}(x,z)\Big]w_x\,{\rm d}x
%\\
%     &\quad-\epsilon^2\,\sigma^{-1}\,\beta^{-1}\,k\,\xi(\epsilon z)\int_{\Bbb R}\Big[A\big(\mathfrak{a}(\epsilon z)\big)\,Z(x)+\phi_{22}(x,z)\Big]xw_x\,{\rm d}x
     \\
&\,=\,  -\epsilon^2\,  \varrho_{1}\,\beta^{-1}\,\alpha_1(z)\,h'
\,-\,\epsilon^2\, \varrho_{1}\,\beta^{-1}\,\alpha_2(z)\,h
\,+\,\epsilon^2\, \varrho_{1}\,\beta^{-1}\,\big[G_1(z)+G_2(z)\big],
  \end{split}
\end{equation}
 where
\begin{equation}
\label{alpha1}
\alpha_1(z)\,=\,2 \varrho_{1}^{-1}\xi(\epsilon z)\,\int_{\Bbb R}\Big[\epsilon\, A'\big(\mathfrak{a}(\epsilon z)\big)\,\beta (\epsilon z)\,Z_x+\phi_{22,xz}(x,z)\Big] w_x\,{\rm d}x,
\end{equation}
\begin{equation}
\label{alpha2}
 \alpha_2(z)=\varpi(\epsilon z)\,\alpha_1(z),
\end{equation}
\begin{equation}
\label{G1}
 G_1(z)\,= -\frac{k(\epsilon z)}{ \varrho_{1}}\xi(\epsilon z)\,\int_{\Bbb R} \Big[A\big(\mathfrak{a}(\epsilon z)\big)\,Z_x+\phi_{22,x}(x,z)\Big]\,w_x\,{\rm d}x,
\end{equation}
\begin{equation}
\label{G2}
G_2(z)\,= -\frac{k(\epsilon z)}{\sigma \varrho_{1}}\xi(\epsilon z)\,\int_{\Bbb R} \,\Big[A\big(\mathfrak{a}(\epsilon z)\big)\,Z+\,\phi_{22}(x,z)\Big]\,x\,w_x\,{\rm d}x.
\end{equation}
Since $w_1$ and $w_x$ are odd functions, while $w_2$, $\phi_{21}$ and $\phi_{22}$ are even functions, so we obtain that
\begin{equation}
\label{J5}
\begin{split}
  J_5 & =\epsilon^2\,p(p-1)\int_{\Bbb R}\Big[a_{11}a_{12}\,h\,w_1\,w_2
   +a_{11}w_1\,\xi(\epsilon z)(\phi_{21}+\phi_{22})
   \Big]\,w^{p-2}w_x\,{\rm d}x
   \\
    &=\,\epsilon^2\,a_{11}\,a_{12}\,h\int_{\Bbb R}\,p(p-1)w^{p-2}\,w_1\,w_2\,w_x\,{\rm d}x
   \\
   &\qquad+\epsilon^2\,a_{11}\,\xi(\epsilon z)\int_{\Bbb R}p(p-1)w^{p-2}\,w_1\,(\phi_{21}+\phi_{22})
   w_x\,{\rm d}x
   \\
   &\,\equiv\, \epsilon^2\beta^{-1}\sigma^{-1}k^2\,h\int_{\Bbb R}\,p(p-1)w^{p-2}w_1\,w_2\,w_x\,{\rm d}x
+\epsilon^2 \varrho_{1}\,\beta^{-1}G_3(z),  \end{split}
\end{equation}
where
\begin{equation}
\label{G3}
 G_3(z)\,=\,a_{11}\,\beta\, \varrho_{1}^{-1}\,\xi(\epsilon z)\,p(p-1)\int_{\Bbb R}w^{p-2}\,w_1\Big[A\big(\mathfrak{a}(\epsilon z)\big)Z+\,\phi_{22}(x,z)\Big]\,
   w_x\,{\rm d}x.
\end{equation}

\medskip
By  differentiating the equation (\ref{w2}) and using equation (\ref{w1}), we obtain
\begin{equation}
\label{relation}
\int_{\Bbb R}p(p-1)w^{p-2}w_x w_1w_2\,{\rm d}x\,=\,-\int_{\Bbb R}w_xw_1\,{\rm d}x+\int_{\Bbb R}\Big(w_x+\frac 1 \sigma xw\Big)w_{2,x}\,{\rm d}x.
\end{equation}
\noindent
Adding (\ref{J3}), (\ref{J5}) and using (\ref{relation}), we have
\begin{equation}
\label{J3+J5}
\begin{split}
  J_3+J_5\,&=\,\epsilon^2\beta^{-1}\sigma^{-1}k^2h\int_{\Bbb R}\Big[p(p-1)w^{p-2}w_1w_2w_x+w_{2,x}w_x+\frac 1 \sigma xw_2w_x
  \\
  &\qquad\qquad\qquad\qquad\qquad\qquad+w_1w_x\Big]\,{\rm d}x+\epsilon^2 \varrho_{1}\beta^{-1}G_3(z)
  \\
&\,=\,\epsilon^2 \beta^{-1}\sigma^{-1}k^2h\int_{\Bbb R}\Big[2w_{2,x}w_x+\sigma^{-1}x(w_2 w)_x\Big]\,{\rm d}x
+\epsilon^2 \varrho_{1}\beta^{-1}G_3(z)
\\
&\,=\,-\epsilon^2 \varrho_{1}\beta^{-1}\sigma^{-1}k^2h
\,+\,\epsilon^2 \varrho_{1}\,\beta^{-1}G_3(z),
  \end{split}
\end{equation}
where we have used (\ref{stationary}) and the following integral identities
$$
2\int_{\Bbb R}w_{2,x}w_x\,{\rm d}x\,=\,-\big(\frac {2}{p-1}+\frac 12\big)\int_{\Bbb R}w_x^2\,{\rm d}x,
\qquad
\sigma^{-1}\int_{\Bbb R}w_2w\,{\rm d}x\,=\,\big(\frac 12-\frac {2}{p-1}\big)\int_{\Bbb R}w^2_x\,{\rm d}x.
$$
Consequently, we infer that
\begin{equation*}
\begin{split}
 &\int_{\Bbb R}\Big[\epsilon^2S_6+\epsilon^2S_7+\epsilon^2S_8+\epsilon^2S_9+M_{11}(x,z)+M_{21}(x,z)
\\
&\qquad\qquad+M_{31}(x,z)+M_{41}(x,z)+M_{51}(x,z)+M_{61}(x,z)\Big]w_x\,{\rm d}x \,
 \\
 &=-\epsilon^2 \varrho_{1}\beta^{-1}\Big[h''
+\big(\hbar_1(\epsilon z)+\alpha_1(z)\big)h'+(\hbar_2(\epsilon z)+\alpha_2( z))h\Big]
\\
&\qquad\qquad+\epsilon^2 \varrho_{1}\beta^{-1}\Big[G_1(z)\,+\,G_2(z)\,+\,G_3(z)\Big],
 \end{split}
\end{equation*}
where functions $\hbar_1$ and $\hbar_2$ are given by
\begin{equation}
\label{gamma1}
\hbar_1(\theta)\,=\,\beta^{-1}\beta'+2\alpha^{-1}\alpha'\,=\,\sigma V^{-1}V_{\theta},
\end{equation}
and
\begin{align}
\begin{aligned}
\label{gamma2}
\hbar_2(\theta)\,=\,&\,-\sigma V^{-1}V_{tt}+(\sigma^{-1}+1)k^2+\varpi\big(\beta^{-1}\beta'+2\alpha^{-1}\alpha'\big)
\\
\,=\,&\,\sigma V^{-1}\,V_{\theta}\,\Theta_{tt}-\sigma V^{-1}V_{tt}+(\sigma^{-1}+1)k^2.
\end{aligned}
\end{align}
%Since
%\begin{equation}
%\label{solvable of varphi}
%\begin{split}
% &\int_{\Bbb R}\Big[\epsilon^2S_6+\epsilon^2S_7+\epsilon^2S_8+\epsilon^2S_9+M_{11}(x,z)+M_{21}(x,z)
%\\
%&\qquad\qquad+M_{31}(x,z)+M_{41}(x,z)+M_{51}(x,z)+M_{61}(x,z)\Big]w_x\,{\rm d}x=0.
%\end{split}
%\end{equation}

\medskip
The above computation give that the validity of (\ref{solvable of varphi1}) holds
if the following problem
\begin{equation}
\label{equation of h}
  \begin{split}
  h''
+\big(\hbar_1(\epsilon z)+\alpha_1(z)\big)h'+(\hbar_2(\epsilon z)+\alpha_2( z))h& =G_1(z)+G_2(z)+G_3(z),
  \\
   h'(1)+k_2h(1)=0,\qquad\qquad &h'(0)+k_1h(0)=0,
  \end{split}
\end{equation}
 has a solution. In fact, the solvability of problem (\ref{equation of h}) will be established in Lemma \ref{lemma of f}. Moreover, $h$ has the following estimate
$$
 \|h\|_{*}\leq C\e^{\frac{1}{2}}.
$$
where the norm $\|\cdot\|_*$ is given in (\ref{fnorm}).

\medskip
We can now find a function defined by ${\tilde\phi}=\phi_4(x,z)$, such that the terms between brackets in (\ref{S(w+phi_1+phi_2+varphi)}) disappear.
More precisely,  the solution has the form
$$
\phi_4(x,z)=\epsilon^2\phi_{41}({x,\epsilon}z)+\epsilon^2\phi_{42}(x,{\epsilon}z),
$$
where $\phi_{41}\big(x, \theta\big)$ and  $\phi_{42}\big(x, \theta\big)$ only depend on $ f(\theta), e(\theta) $, but their derivatives.
The function $\epsilon^2\phi_{41}(x, \theta)$ satisfies the following equation, for fixed $\theta=\e z$
\begin{equation*}
  \begin{split}
    &-\phi_{xx}+\phi-pw^{p-1}\phi
    \\
    &=\epsilon^2S_6+\epsilon^2S_8+M_{11}+M_{21}+M_{51}
    \,+\,\frac{\epsilon^2k^2}{\beta\sigma}\Big[\,\frac{\sigma}{\beta} w_{1,x}+\frac{1}{\beta}xw_1+\frac{\beta}{\sigma}h^2w_2+\frac{2\beta}{\sigma} fhw_2\,\Big]
    \\
    &\quad+\frac{2\epsilon^2}{\beta^2}\,\xi\,\Big(\frac{\beta'}{\beta}-\varpi\Big)
    \Big[\,\epsilon\, A'(\mathfrak{a})\,\beta \,x\, Z_x+x \,\phi_{22,xz}\,\Big]
    \,-\,\epsilon^2\, \sigma^{-1}\,k\,(f+h)\,\xi\,\Big[\,A(\mathfrak{a})\,Z+\phi_{22}\,\Big]
    \\
    &\quad+\frac{\epsilon^2}{\beta^2}\,\xi\,\Big(\frac{2 \alpha'}{\alpha}-\varpi\Big)
    \Big[\,\epsilon\, A'(\mathfrak{a})\,\beta\, Z+\phi_{22,z}\,\Big]
    \\
    &\quad+\epsilon^2\frac {p(p-1)}{2}\,w^{p-2}\Big[\, a_{11}^2\,w_{1}^2+a_{12}^2\,(2\,f\,h+h^2)\,w_{2}^2
    +e^2\,Z^2+\xi^2\,\big(A(\mathfrak{a})\,Z+\phi_{22}\big)^2
    \\
    &\qquad\qquad\qquad\qquad\quad\ +2 \,a_{12}\,(f+h)\,w_2\,\xi \, \big(A(\mathfrak{a})\,Z+\phi_{22}\big)
    \\
    &\qquad\qquad\qquad\qquad\quad\ +2\,a_{12}\,(f+h)\,w_2\, e\,Z+2\,\xi \big(A(\mathfrak{a})Z+\phi_{22}\big)\,e\, Z\, \Big],
  \end{split}
\end{equation*}
and $\phi_{42}(x,\theta)$ satisfies
\begin{equation*}
  \begin{split}
    &-\phi_{xx}+\phi-pw^{p-1}\phi
    \\
    &=S_7+S_9+\frac{k^2}{\beta\sigma}\Big[\,h\,w_{2,x}+\sigma^{-1}\,h\,x\,w_2
    +h\,w_1\,\Big]
    -\frac{k}{\beta}\,\xi\,\Big[A(\mathfrak{a})\,Z_x+\phi_{22,x}\Big]
    \\
    &\quad-\frac{2}{\beta}\,h'\,
    \xi\,\Big[\,\epsilon \,A'(\mathfrak{a})\,\beta \,Z_x+\phi_{22,xz}\,\Big]
    -\frac{2\varpi}{\beta}\,h\,\xi
    \Big[\,\epsilon\, A'(\mathfrak{a})\,\beta \,Z_x+\phi_{22,xz}\,\Big]
    \\
    &\quad
    -\sigma^{-1}\,\beta^{-1}\,k\,x\,\xi\,\Big[\,A(\mathfrak{a})\, Z+\phi_{22}\,\Big]
    +p(p-1)\,w^{p-2}\,\Big[\,a_{11}\,a_{12}\,h\,w_1\,w_2+a_{11}\,w_1\,\xi \, \big(A(\mathfrak{a})Z+\phi_{22}\big)\,\Big].
   \end{split}
\end{equation*}
Finally, our basic approximate solution $v$ in (\ref{vdefine1}) to the problem near the curve $\Gamma_\epsilon$ is
\begin{equation}
\label{basic approximate}
v_5(x, z)\,=\,w(x)+\phi_1(x, z)+\phi_2(x, z)+\epsilon\,e(\e z)\, Z(x)+\phi_3(x, z)+\phi_4(x, z).
\end{equation}

\subsection{The global approximate solution and the errors}

\medskip
Recall the relation (\ref{rescaling}),  the coordinates $(t, \theta)$ in (\ref{Fermicoordinates}), $(s, z)$ in  (\ref{szcoordinate}) and
$(x ,z)$ in (\ref{vdefine1})  and also the local approximate solution $v_{5}(x, z)$ in (\ref{basic approximate}), which is constructed near the curve $\Gamma_\epsilon$ in the coordinates $(x, z)$.
By the relation in (\ref{vdefine1}),
we then make an extension and simply define the approximate solution to (\ref{problemafterscaling}) in the form
\begin{equation}
\label{globalapproximation}
{\mathbf W}(\tilde y)\,=\,\eta_{3\delta}^{\epsilon}(s)\,\alpha(\epsilon z)\,
v_5\Big(\beta(\e z)\big(s-\epsilon\,f(\e z)-\epsilon\,h(\e z)\big),\, z\Big).
\end{equation}
Note that, in the coordinates $(\tilde {y}_1,\tilde {y}_2)$ introduced in (\ref{rescaling}), ${\mathbf W}$ is a function defined on $\Omega_{\epsilon}$ which is extended globally as $0$ beyond the $6\delta/\epsilon$-neighborhood of $\Gamma_\epsilon$ and also satisfies (\ref{decayw}).

\medskip
Recalling (\ref{errorrelationinterior}), (\ref{errorrelationboundary0}) and (\ref{errorrelationboundary1}),
the local forms of the error terms can now be formulated.
The interior error ${\mathcal E}$ can be arranged as follows
\begin{align}
\begin{aligned}
{\mathcal E}\,\equiv\,&S(v_5)=S(w+\phi_1+\phi_2+\epsilon eZ+\phi_3+\phi_4)
\\
\,=&\,\epsilon^2\,S_3+\epsilon^2\,S_4+\epsilon^2\,S_5+B_2(w)+\epsilon^3\,\beta^{-2}\,e{''}\,Z+\epsilon\, \lambda_0\, e\,Z
\label{new error-2}
\\
\,&+\beta^{-2}\,\phi_{1,zz}+\beta^{-2}\,\phi_{4,zz}+M_{12}(x,z)+M_{22}(x,z)+M_{32}(x,z)
\\
\,&+M_{42}(x,z)+M_{52}(x,z)+M_{62}(x,z)+B_3(\phi_3)+B_3(\phi_4),
\end{aligned}
\end{align}
where we have used (\ref{S(w+phi_1+phi_2+varphi)}) and the equation of $\phi_4$ in (\ref{equation of Phi}).
The boundary error term $g_0$ has the form
 \begin{align}
g_0(x)
\,=&\,-\epsilon\beta\,\big(k_1\,f+f'\big)\,w_x
-\epsilon\, b_5\,x\,\Big[\epsilon\,a_{11}\,w_{1,x}
+\epsilon\, e \,Z_x+\epsilon
 ^2\,\hat{\phi}_x+\phi_{4,x}\Big]
 \nonumber
 \\
  &-\epsilon\,\beta\,\big(k_1\,f+f'\big)\times\Big\{
 \epsilon \,a_{12}\,(f+h)\,w_{2,x}+\epsilon\, e \,Z_x+\epsilon
 ^2\,\hat{\phi}_x
  \nonumber
\\
&\qquad\qquad\qquad\qquad\qquad+\epsilon\,\beta^{-1}\,\big(A(0)\,Z_x+\phi_{22,x}(x,0)\big)+\phi_{4,x}\Big\}
 \nonumber
 \\
 &+\epsilon\,\alpha'\,\alpha^{-1}\Big\{\epsilon\,a_{11}\,w_1+\epsilon\,e \, Z+\epsilon^2\,\hat{\phi}+\phi_4\Big\}+\epsilon^2\,k\,\alpha'\,\alpha^{-1}\,\beta^{-1}\,x\,w
  \nonumber
 \\
&+\epsilon^2\,\frac{k\,\alpha'}{\alpha}\Big(\frac{x}{\beta}+f+h\Big)\times\Big[\epsilon\, a_{11}\,w_1+\epsilon\, a_{12}\,(f+h)\,w_2
 \nonumber
\\
& \qquad\qquad\qquad\qquad\qquad+\epsilon\,\beta^{-1}\Big( A(0)\,Z+\phi_{22}(x,0)\Big)+\epsilon\, e \, Z+\epsilon^2\,\hat{\phi}+\phi_{4}\Big]
 \nonumber
 \\
 &-\epsilon^2\Big\{b_1(\beta^{-2}x^2 +f^2+h^2+2fh)\beta-k\Big[\beta^{-2}\beta'x^2-(f+h)(f'+h')\beta\Big]\Big\}w_x
  \nonumber
 \\
 &-\epsilon^3\,\Big[b_1\,\Big(\frac x \beta +f+h\Big)^2\, \beta-k\,\Big(\frac x \beta +f+h\Big)\Big(\frac{\beta'}{\beta}\,x-\beta \,f'-\beta\, h'\Big)\Big]
 \label{g0}
 \\
 &\times \Big[a_{11}w_{1,x}+a_{12}(f+h)w_{2,x}+\epsilon \beta^{-1}\big(A(0)Z_x+\phi_{22,x}\big)+e Z_x+\epsilon\hat{\phi}_x+\epsilon^{-1}\phi_{4,x}\Big]
  \nonumber
 \\[1.2mm]
 &+\epsilon^2\,a'_{11}w_1+\epsilon^2\,e'Z
 +\epsilon^2\hat{\phi}_z+\phi_{4,z}+\epsilon^2\, k\,\beta^{-2}x\big(\epsilon A'(0)\,\beta(0)\,Z+\phi_{22,z}(x,0)\big)
  \nonumber
 \\[1.2mm]
 &+\epsilon\, k\,\Big(\frac{x}{\beta}+f+h\Big)\times \Big[\epsilon^2\, a'_{11}\,w_1+\epsilon^2\, a'_{12}\,(f+h)\,w_2
  \nonumber
 \\
 &\qquad\qquad\qquad\qquad\qquad\quad+\epsilon^2\, a_{12}\,(f'+h')\,w_2+\epsilon^2\,e'\,Z+\epsilon^2\hat{\phi}_z+\phi_{4,z}\Big]
  \nonumber
  \\
 &-\epsilon^2\,b_2\,\Big(\frac{x}{\beta}+f+h\Big)^2\times\Big[\epsilon^2\, a'_{12}\,w_1+\epsilon^2\, a'_{12}\,(f +h)\,w_2+\epsilon^2\, a_{12}\,(f'+h')\,w_2
  \nonumber
 \\
& \qquad\qquad+\epsilon\,\beta^{-1}\big(\epsilon \, A'(0)\,\beta(0)\,Z+\phi_{22,z}(x,0)\big)+ \epsilon^2 \,e'\,Z+\epsilon^2\hat{\phi}_z+\,\phi_{4,z}\Big]
\nonumber
\\
&+D_0^{0}(v_5).
 \nonumber
 \end{align}
The term $g_1$  has a similar expression.

\medskip
We decompose
\begin{equation}
\label{E1-d}
{\mathcal E}\,=\,{\mathcal E}_{11}+{\mathcal E}_{12},\quad g_0\,=\,g_{01}+g_{02},\quad g_1\,=\,g_{11}+g_{12},
\end{equation}
with
\begin{equation}
\label{E11}
\begin{split}
{\mathcal E}_{11}\,=\,\epsilon^3\,\beta^{-2}\,e{''}\,Z+\epsilon\, \lambda_0\, e\,Z,
\qquad &{\rm and}\qquad
{\mathcal E}_{12}\,=\,{\mathcal E}-{\mathcal E}_{11},
\\
g_{01}\,=\,-\epsilon\,\Big[k_1\,\beta(0)\,f+\beta(0)\,f'\Big]\,w_x+\epsilon^2\, e'\,Z,
\qquad &{\rm and}\qquad
g_{02}\,=\,g_0-g_{01},
\\
g_{11}\,=\,-\epsilon\,\Big[k_2\,\beta(1)\,f+\beta(1)\,f'\,\Big]\,w_x+\epsilon^2\, e'\,Z,
\qquad &{\rm and}\qquad g_{12}\,=\,g_1-g_{11}.
\end{split}
\end{equation}

\medskip
For further references, it is useful to estimate the $L^2(\mathcal S)$ norm of ${\mathcal E}$.  From the uniform bound of $e$ in (\ref{enorm}), it is easy to see that
\begin{equation}
\label{E11norm}
\|{\mathcal E}_{11}\|_{L^2(\mathcal S)}\,\leq\, C\epsilon^{1/2}.
\end{equation}
Since $\phi_1, \phi_2$ and $\epsilon eZ$ are of size $O(\epsilon)$,  all terms in ${\mathcal E}_{12}$ carry $\epsilon^2$ in front. We claim that
\begin{equation}
\label{E12norm}
\|{\mathcal E}_{12}\|_{L^2(\mathcal S)}\,\leq\, C\epsilon^{ 3/2}.
\end{equation}
A rather delicate term in ${\mathcal E}_{12}$ is the one carrying $f{''}$ since we only assume a uniform bound on $\|f{''}\|_{L^2(0,1)}$.  For example, we have a term $K_1=\epsilon^2f{''}$ in $S(w)$ which has bound like
$$
\|K_1\|_{L^2(\mathcal S)}\,\leq\, C\epsilon^2.
$$
Since
\begin{equation*}
 \begin{split}
 &\big|N_0(\phi_1+\phi_2+{\epsilon}e z+\phi_3+\phi_4)\big|\,
 \\
 &=\,\big|(w+\phi_1+\phi_2+{\epsilon}e z+\phi_3+\phi_4)^p-w^p-pw^{p-1}(\phi_1+\phi_2+{\epsilon}e z+\phi_3+\phi_4)\big|\,
 \\
 &=\,\big|p\big[w+t(\phi_1+\phi_2+{\epsilon}e z+\phi_3+\phi_4)\big]^{p-2}(\phi_1+\phi_2+{\epsilon}e z+\phi_3+\phi_4)^2\big|,
 \end{split}
\end{equation*}
we obtain
$$
\big\|N_0(\phi_1+\phi_2+{\epsilon}e z+\phi_3+\phi_4)\big\|_{L^2(\mathcal S)}\,\leq\, C\epsilon^{3/2}.
$$
Other terms can be estimated in the similar way.
Similarly, we have the following estimate
\begin{equation}
\label{g02norm}
\|g_{02}\|_{L^2(\Bbb R)}+\|g_{12}\|_{L^2(\Bbb R)}\,\leq\, C\epsilon^{3/2}.
\end{equation}

\medskip
Moreover, for the Lipschitz dependence of the term of error ${\mathcal E}_{12}$ on the parameters $f$ and $e$ for the norm defined in (\ref{fnorm}) and (\ref{enorm}), we have the validity of the estimate
\begin{equation}
\label{E12L}
\|{\mathcal E}_{12}(f_1,e_1)-{\mathcal E}_{12}(f_2,e_2)\|_{L^2(\mathcal S)}
\,\leq\, C\epsilon^{ 3/2}\big[\,\|f_1-f_2\|_{*}+\|e_1-e_2\|_{**}\,\big].
\end{equation}
Similarly, we obtain
\begin{equation}
\label{g02L}
\begin{split}
\|g_{02}(f_1,e_1)-g_{02}(f_2,e_2)\|_{L^2(\Bbb R)}\,+\,\|g_{12}\,&(f_1,e_1)-g_{12}(f_2,e_2)\|_{L^2(\Bbb R)}
\\
\,\leq\, &\,C\epsilon^{3/2}\big[\,\|f_1-f_2\|_{*}+\|e_1-e_2\|_{**}\,\big].
\end{split}
\end{equation}

\section{The resolution theory for the projected problem}
\label{section5}

\subsection{The invertibility of ${\mathcal L}$}\label{theinvertibility}
\setcounter{equation}{0}

Let ${\mathcal L}$, ${\mathcal D}_1$ and ${\mathcal D}_0$ be the operators defined in $H^2(\mathcal S)$ by (\ref{Lmathcal})-(\ref{D0mathcal}),  and also $ g_0$, $g_1$ be the functions in (\ref{g0}). Note that the function $\chi(\epsilon|x|)$ is even in the definition of ${\mathcal L}$.  In this section, we study the following linear problem:
for given $\mathfrak{h}\in L^2(\mathcal S),\,g_0,\,g_1\in L^2(\Bbb R)$,  finding functions $\phi\in H^2(\mathcal S), c, d\in L^2(0,1)$ and constants $l_1, l_0, m_1, m_0$ such that
\begin{equation}
\label{invert1}
{\mathcal L}(\phi)\,=\,\mathfrak{h}\,+\,c(\epsilon z)\,\chi(\epsilon|x|)\, w_x\,+\,d(\epsilon z)\,\chi(\epsilon|x|)\, Z\quad{\rm in}~\mathcal S,
\end{equation}
\begin{equation}
\label{invert2}
{\mathcal D}_1(\phi)\,=\,\eta^\epsilon_{\delta}(s)\, g_1\,+\,l_1\,\chi(\epsilon|x|)\, w_x\,+\,m_1\,\chi(\epsilon|x|)\, Z\quad{\rm on}~\partial_1 \mathcal S,
\end{equation}
\begin{equation}
\label{inbert3}
{\mathcal D}_0(\phi)\,=\,\eta^\epsilon_{\delta}(s)\,  g_0\,+\,l_0\,\chi(\epsilon|x|)\, w_x\,+\,m_0\,\chi(\epsilon|x|)\, Z\quad{\rm on}~\partial_0\mathcal S,
\end{equation}
\begin{equation}
\label{invert4}
\int_{\Bbb R}\phi(x,z)\,w_x(x)\,{\rm d}x \,=\,0,
\qquad 
\int_{\Bbb R}\phi(x,z)\,Z(x)\,{\rm d}x \,=\,0,
\qquad 
0<z<\frac 1 \epsilon.
\end{equation}

\begin{proposition}
 \label{proposition 7.1}If $\delta$ in the definition of ${\mathcal L}$ is chosen small enough and $\mathfrak{h}\in L^2(\mathcal S)$, there exists a constant $C>0$,  independent of $\epsilon$,  such that for all small $\epsilon$,  the problem (\ref{invert1})-(\ref{invert4}) has a unique solution $\phi={\mathbf T}(\mathfrak{h})$ which satisfies
\begin{equation}
\label{proposition6.1-1-innter equation}
\|\phi\|_{H^2(\mathcal S)} \,\leq\, C\big(\|\mathfrak{h}\|_{L^2(\mathcal S)}+\|g_{02}\|_{L^2(\Bbb R)}+ \|g_{12}\|_{L^2(\Bbb R)} \big),
\end{equation}
Moreover, if $h,g_0,g_1$ have compact supports contained in $|x|\leq 20\delta/\epsilon$,  then
\begin{equation}
\label{proposition6.1-2-inter equation}
|\phi(x,z)|+|\nabla \phi(x,z)|\,\leq\, \|\phi\|_{L^{\infty}}e^{-2\delta/\epsilon}\ \ \ \  \
~{\rm for}~|x|>40\delta/\epsilon.
\end{equation}
\end{proposition}

\medskip
\noindent {\it Proof.}
Note that $\eta^\epsilon_{\delta}(s)\,g_{11}$ and $\eta^\epsilon_{\delta}(s)\,g_{01}$ can be absorbed by $l_1\chi\, w_x+m_1\chi\, Z$ and $l_0\chi\, w_x+m_0\chi\, Z$, the problem (\ref{invert1})-(\ref{invert4}) can be written as
\begin{equation}
\label{invert11}
{\mathcal L}(\phi)\,=\,\mathfrak{h}\,+\,c(\epsilon z)\,\chi(\epsilon|x|) \,w_x\,+\,d(\epsilon z)\,\chi(\epsilon|x|)\, Z\quad{\rm in}~\mathcal S,
\end{equation}
\begin{equation}
\label{invert21}
{\mathcal D}_1(\phi)\,=\,\eta^\epsilon_{\delta}(s)\, g_{12}+l_1\,\chi(\epsilon|x|)\, w_x\,+\,m_1\,\chi(\epsilon|x|)\, Z\quad{\rm on}~\partial_1 \mathcal S,
\end{equation}
\begin{equation}
\label{inbert31}
{\mathcal D}_0(\phi)\,=\,\eta^\epsilon_{\delta}(s)\,  g_{02}\,+\,l_0\,\chi(\epsilon|x|)\, w_x\,+\,m_0\,\chi(\epsilon|x|)\, Z\quad{\rm on}~\partial_0\mathcal S,
\end{equation}
\begin{equation}
\label{invert41}
\int_{\Bbb R}\phi(x,z)w_x(x)\,{\rm d}x \,=\,0,
\qquad 
\int_{\Bbb R}\phi(x,z)Z(x)\,{\rm d}x \,=\,0,
\quad 0<z<\frac 1 \epsilon.
\end{equation}
We will reduce problem (\ref{invert11})-(\ref{invert41}) to a small perturbed problem in which Lemma 2.5 in \cite{wei-yang} is applicable. We will achieve this by introducing a change of variables that eliminates the weight $\beta^{-2}$ in front of $\phi_{zz}$.

\medskip
We let
$$
\phi(x,z)\,=\,\varphi(x, {\tilde z})
\quad\mbox{with}\quad 
{\tilde z}={\Upsilon}(z),
$$
where  the map
${\Upsilon}: [0,\frac 1\epsilon)\rightarrow [0,\frac {{\ell}}{\epsilon})$
is a diffeomorphism defined in (\ref{mathfrak-a-function}) with the constant ${\ell}$ given in (\ref{ellnumber}).
We then have
$$
\phi_z
\,=\,\beta(\epsilon z)\varphi_{{\tilde z}},
\qquad
\phi_{zz}
\,=\,\beta^2(\epsilon z)\varphi_{{\tilde z}{\tilde z}}\,+\,\epsilon \beta'(\epsilon z)\varphi_{{\tilde z}},
$$
while the differentiation in $x$ does not change.
Similarly, for any function $g$, we set 
$$
\tilde{g}(x,{\tilde z})\,=\,g(x,\, {\Upsilon}^{-1}({\tilde z})).
$$
The equation in terms of $\varphi$ now reads
\begin{equation}
\label{equation of varphi}
\begin{split}
\Delta \varphi-\varphi+p\,w^{p-1}\varphi
\,=\,&-p\,\big({\tilde\beta}^{-2}\chi(\epsilon|x|)
{\tilde{\mathbf W}}^{p-1}-w^{p-1}\big)\varphi
\,-\,\chi(\epsilon|x|)\, {\tilde B}_3(\varphi)
\,-\,\epsilon\, {\widetilde{\beta'}}{\tilde\beta}^{-2}\,\varphi_{{\tilde z}}
\\
&\,+\,\tilde {\mathfrak{h}}\,+\,\tilde{c}(\epsilon {\tilde z})\,\chi(\epsilon|x|)\, w_x
+\tilde {d}(\epsilon {\tilde z})\,\chi(\epsilon|x|)\, Z
\qquad
{\rm in }~\hat {\mathcal S},
\end{split}
\end{equation}
with the boundary conditions
\begin{equation}
\label{boundary-prior-1}
\begin{split}
&\varphi_{{\tilde z}}\,=\,\beta^{-1}\Big[\,
{\mathbb G}_1(\varphi)
\,-\,l_1\,\chi(\epsilon|x|)\, w_x
\,-\,m_1\,\chi(\epsilon|x|)\, Z\,\Big]~~~~{\rm on }~~\partial_1\hat{\mathcal S},
\\
&\varphi_{{\tilde z}}\,=\,\beta^{-1}\Big[\,
{\mathbb G}_0(\varphi)
\,-\,l_0\,\chi(\epsilon|x|)\, w_x
\,-\,m_0\,\chi(\epsilon|x|) \,Z\,\Big]~~~~{\rm on }~~\partial_0\hat{\mathcal S},
\end{split}
\end{equation}
where
$$
{\mathbb G}_1(\varphi)\,=\,-\widetilde{\eta^\epsilon_{\delta}} \,\tilde {g}_{12}+\chi(\epsilon|x|) \,{\tilde D}_3^1(\varphi)+\chi(\epsilon|x|) \, \tilde{D}_0^1(\varphi),
$$
$$
{\mathbb G}_0(\varphi)\,=\,-\widetilde{\eta^\epsilon_{\delta}}\,\tilde {g}_{02}+\chi(\epsilon|x|)\, {\tilde D}_3^0(\varphi)+\chi(\epsilon|x|)\, \tilde{D}_0^0(\varphi),
$$
and the orthogonality condition
\begin{equation}
\label{orthogonality-prior}
\m \varphi(x,{\tilde z})w_x(x){\rm d}x\,=\,\m \varphi(x,{\tilde z})Z(x){\rm d}x\,=\,0,~~~0<{\tilde z}<\frac {{\ell}}{\epsilon}.
\end{equation}
The operators $\tilde {B}_3$, $\tilde{D}_3^i$, $\tilde{D}_0^i$ are defined by using the above formulas to replace the $z$-derivatives by ${\tilde z}$-derivatives and the variable $z$ by ${\Upsilon}^{-1}({\tilde z})$ in the operators $B_3$,  $D_3^i$, $D_0^i$.

\medskip
The key point is the following: the operator
$$
B_4(\varphi)\,=\,p\,
\big({\tilde\beta}^{-2}\,\chi(\epsilon|x|)\,{\tilde{\mathbf W}}^{p-1}-w^{p-1}\big)\varphi
\,+\,\chi(\epsilon|x|)\, \tilde{B}_3(\varphi)
\,+\,\epsilon \widetilde{\beta'}{\tilde\beta}^{-2}\varphi_{{\tilde z}},
$$
is small in the sense that
$$
\|B_4(\varphi)\|_{L^2(\hat {\mathcal S})}\,\leq\, C \delta \|\varphi\|_{H^2(\hat {\mathcal S})}.
$$
\noindent This last estimate is a rather straightforward consequence of the fact that $| s|\leq 20 \delta \epsilon^{-1}$ wherever the operator $\tilde {B}_3$ is supported, and the other terms are even smaller when $\epsilon$ is small.  Similar results hold for  ${\mathbb G}_1(\varphi)$ and ${\mathbb G}_0(\varphi)$.
Thus by reducing $\delta$ if necessary, we apply the invertibility result of Lemma 2.5 in \cite{wei-yang}.~~The result thus follows by transforming the estimate for $\varphi$ into similar one for $\phi$ via change of variables. This concludes the proof. \qed

 \subsection{Solving the nonlinear projection problem}\label{solvingthenonlinearprojectionproblem}

 In this section, we will solve (\ref{projectedproblem1})-(\ref{projectedproblem4}) in $\mathcal S$.  A first elementary, but crucial observation is that: The term
 $$
 {\mathcal E}_{11}\,=\,\big(\epsilon^3\, \beta^{-2}\,e{''}\,+\epsilon \,\lambda_0 \,e\big) \,Z,
 $$
 in the decomposition of ${\mathcal E}$,  has precisely the form $d(\epsilon z)Z$ and can be absorbed in that term $\chi(\e|x|)\, d(\epsilon z)\, Z$.
 Then, the equivalent equation of (\ref{projectedproblem1}) is
 $$
 {\mathcal L}(\phi)\,=\,-\eta^\epsilon_{\delta}(s)\,{\mathcal E}_{12}
 \,-\,\eta^\epsilon_{\delta}(s)\,{\mathcal N}(\phi)
 \,+\,c(\epsilon z)\,\chi(\e|x|)\,w_{x}
 \,+\,d(\epsilon z)\,\chi(\e|x|)\,Z.
 $$
 Similarly, we can also absorb the terms of order $O(\epsilon)$ in $g_0$ and $g_1$.

\medskip
 Let ${\mathbf T}$ be the bounded operator defined by Proposition \ref{proposition 7.1}. Then the problem (\ref{projectedproblem1})-(\ref{projectedproblem4}) is equivalent to the following fixed point problem
 \begin{equation}
 \label{section-5-1-fixed}
 \phi\,=\,{\mathbf T}\big(-\eta^\epsilon_{\delta}(s)\,{\mathcal E}_{12}
 \,-\,
 \eta^\epsilon_{\delta}(s)\,{\mathcal N}(\phi)\big)\,\equiv\, \mathcal A(\phi).
 \end{equation}
 \noindent
 We collect some useful facts to find the domain of the operator $\mathcal A$ such that $\mathcal A$ becomes a contraction mapping.

\medskip
 The big difference between ${\mathcal E}_{11}$ and ${\mathcal E}_{12}$ is their size. From (\ref{E11norm}) and (\ref{E12norm})
 \begin{equation}
 \label{E12-1norm}
 \|{\mathcal E}_{12}\|_{L^2(\mathcal S)}\,\leq\, c_{*}\epsilon^{3/2},
 \end{equation}
 \noindent while ${\mathcal E}_{11}$ is only of size $O(\epsilon^{1/2})$.  Similarly, we have
 \begin{equation}
 \label{g0212norm}
 \|g_{02}\|_{L^2(\Bbb R)}+\|g_{12}\|_{L^2(\Bbb R)}\,\leq\, c_{*}\epsilon^{3/2}.
 \end{equation}
On the other hand, the operator ${\mathbf T}$ has a useful property: assuming $\bar{h}$ has a support contained in $|x|\leq 20\delta/\epsilon$,  then $\phi={\mathbf T}(\bar{h})$ satisfies the estimate
 \begin{equation}
 \label{T2estimate}
 |\phi(x,z)|+|\nabla \phi(x,z)|\,\leq\, \|\phi\|_{L^{\infty}}e^{-2\delta/\epsilon}
~~{\rm for}~|x|>40\delta/\epsilon.
\end{equation}
\noindent Recall that the operator $\psi(\phi)$ satisfies, as seen directly from its definition
\begin{equation}
\label{psiphidefinition}
\|\psi(\phi)\|_{L^{\infty}}\,\leq\, C\epsilon\Big[\big \| |\phi|+|\nabla \phi|\big\|_{L^{\infty}(|x|>20\delta/\epsilon)}+e^{-\delta/\epsilon}\Big],
\end{equation}\noindent
and a Lipschitz condition of the form
\begin{equation}
\label{psi-lip1}
\|\psi(\phi_1)-\psi(\phi_2)\|_{L^{\infty}}\,\leq\, C\epsilon \Big[\parallel |\phi_1-\phi_2|+|\nabla(\phi_1-\phi_2)|\parallel_{L^{\infty}(|x|>20\delta/\epsilon)}\Big].
\end{equation}
Now, the facts above will allow us to construct a region where contraction mapping principle applies and then solve the problem (\ref{projectedproblem1})-(\ref{projectedproblem4}).

\medskip
Consider the following closed, bounded subset
\begin{equation}
\label{D}
\Xi\,=\,\Big\{\phi\in H^2(\mathcal S)\,\Big|\,\,\|\phi\|_{H^2(\mathcal S)}
\,\leq\, \tau \epsilon^{3/2},
\quad\big\| |\phi|+|\nabla \phi|\big\|_{L^{\infty}(|x|>40\delta/\epsilon)}\,\leq\, \|\phi\|_{H^2(\mathcal S)}e^{-\delta/\epsilon}\Big\}.
\end{equation}
If the constant $\tau$ is sufficiently large, then the map $\mathcal A$ defined in (\ref{section-5-1-fixed}) is a contraction from $\Xi$ into itself.
This property will lead to the following proposition:
\noindent \begin{proposition}
 \label{prop}There is a number $\tau>0$ such that for all $\epsilon$ small enough satisfying gap condition\;(\ref{gapcondition})\; and all parameters $(f,e)$ in $\mathcal F$,  problem (\ref{projectedproblem1})-(\ref{projectedproblem4}) has a unique solution $\phi=\phi(f,e)$ which satisfies
$$
\|\phi\|_{H^2(\mathcal S)}\,\leq\, \tau \epsilon^{3/2},$$
$$
\big\| |\phi|+|\nabla \phi| \big\|_{L^{\infty}(|x|>40\delta/\e)}\,\leq\, \|\phi\|_{H^2(\mathcal S)}e^{-\delta/\epsilon}.
$$
 Moreover, $\phi$ depends Lipschitz-continuously on the parameters $f$ and $e$ in the sense of the estimate
 \begin{equation}
\label{characteriztion}
\|\phi(f_1,e_1)\,-\,\phi(f_2,e_2)\|_{H^2(\mathcal S)}
\,\leq\, C\epsilon^{3/2}\big [\,\|f_1-f_2\|_{*}\,+\,\|e_1-e_2\|_{**}\,\big].
\end{equation}
\end{proposition}

\medskip
\proof
The proof is similar as that for Proposition 5.1 in \cite{dPKW}.
\qed

 \section{Estimates of the projection against $w_x$ and $Z$}
 \label{section6}
\setcounter{equation}{0}

\noindent
As we have mentioned in Section \ref{section3}, in the next part of the paper, we will set up
 equations for the parameters $f$ and $e$ which are equivalent to making $c(\epsilon z)$, $d(\epsilon z)$, $l_1$, $l_0$, $m_1$ and $m_0$ are identically zero in the system (\ref{projectedproblem1})-(\ref{projectedproblem4}). To achieve this, we first multiply the equation against $w_x$ only in $x$.
 The equations
 $$
 c(\epsilon z)=0,\quad l_1=0,\quad l_0=0,
 $$
 are then equivalent to the relations
 \begin{align}
 \int_{\Bbb R}\Big[\eta^\epsilon_{\delta}(s)\, {\mathcal E}+\eta^\epsilon_{\delta}(s)\, {\mathcal N}(\phi)+\chi(\epsilon|x|) B_3(\phi)+p\big(\beta^{-2}\chi(\e |x|){\mathbf W}^{p-1}&-w^{p-1}\big)\phi \Big]\,w_x\,{\rm d}x \,=\,0, \label{c=0}
     \\
      \int_{\Bbb R}\Big[\eta^\epsilon_{\delta}(s)\, g_1+ \chi(\epsilon|x|) D_3^1(\phi)+\chi(\epsilon|x|) D^1_0(\phi)\Big]\,w_x\,{\rm d}x\,&=\,0,
  \quad z=1/\epsilon, \label{l1=0}
  \\
  \int_{\Bbb R}\Big[\eta^\epsilon_{\delta}(s)\, g_0+\chi(\epsilon|x|) D_3^0(\phi)+\chi(\epsilon|x|) D_0^1(\phi)\Big]\,w_x\,{\rm d}x\,&=\,0,
  \quad z=0. \label{l2=0}
 \end{align}
Similarly,
$$
d(\epsilon z)=0,
\quad
m_1=0,\quad
m_0=0,
$$
if and only if
\begin{align}
    \int_{\Bbb R}\Big[\eta^\epsilon_{\delta}(s)\, {\mathcal E}+\eta^\epsilon_{\delta}(s)\,{\mathcal N}(\phi)+\chi(\epsilon|x|) B_3(\phi)+p\big(\beta^{-2}\chi(\e |x|){\mathbf W}^{p-1}&-w^{p-1}\big)\phi \Big]\,Z\,{\rm d}x\,=\,0,\label{d=0}
    \\
    \int_{\Bbb R}\Big[\eta^\epsilon_{\delta}(s)\, g_1+\chi(\epsilon|x|) D_3^1(\phi)+\chi(\epsilon|x|) D^1_0(\phi)\Big]\,Z\,{\rm d}x\,&=\,0,
  \quad z=1/\epsilon,\label{m1=0}
  \\
  \int_{\Bbb R}\Big[\eta^\epsilon_{\delta}(s)\, g_0+\chi(\epsilon|x|) D_3^0(\phi)+\chi(\epsilon|x|) D_0^1(\phi)\Big]\,Z\,{\rm d}x\,&=\,0,
   \quad z=0.\label{m0=0}
\end{align}
It is therefore of crucial importance to carry out computation of the estimates of the terms
$$\int_{\Bbb R}{\mathcal E}w_x\,{\rm d}x
\quad
\mbox{and}
\quad
\int_{\Bbb R}{\mathcal E}Z\,{\rm d}x
$$
and, similarly, some other terms involving $\phi$.

\subsection{Estimates for projections of the error}
In this section, we carry out some estimates for the terms
$$\int_{\Bbb R}{\mathcal E}w_x\,{\rm d}x
\quad
\mbox{and}
\quad
\int_{\Bbb R}{\mathcal E}Z\,{\rm d}x.
$$
 For the pair $(f,e)$ satisfying (\ref{fnorm}) and (\ref{enorm}),
  we denote  ${\mathbf b}_{1\epsilon}$ and ${\mathbf b}_{2\epsilon}$,  generic, uniformly bounded continuous functions of the form
    $$
    {\mathbf b}_{l\epsilon}\,=\,{\mathbf b}_{l \epsilon}\big(z, f(\epsilon z), e(\epsilon z), f'(\epsilon z), \epsilon e'(\epsilon z)\big),~l=1,2
    $$
where ${\mathbf b}_{1\epsilon}$ is uniformly Lipschitz in its four last arguments.
\medskip

\subsubsection{}
First, we estimates for the term $\int_{\Bbb R}{\mathcal E}w_x\,{\rm d}x $, where ${\mathcal E}$, we recall, was defined in (\ref{new error-2}) and $w_x$ is an odd function. Integration against all even terms, say ${\mathcal E}_{11}$ and $S_4, M_{12}, M_{22}$  in ${\mathcal E}_{12}$, therefore just vanish. We have
\begin{equation}
\label{intE1wx}
\begin{split}
\int_{\Bbb R}\,{\mathcal E}\,w_x\,{\rm d}x \,=\,&\,\int_{\Bbb R}{\mathcal E}_{12}\,w_x\,{\rm d}x
\\
\,=\,&\,\int_{\Bbb R}\epsilon^2\,S_3\,w_x\,{\rm d}x +\int_{\Bbb R}\epsilon^2\,S_5\,w_x\,{\rm d}x+\int_{\Bbb R}B_2(w)\,w_x\,{\rm d}x+\int_{\Bbb R}\beta^{-2}\,\phi_{1,zz}\,w_x\,{\rm d}x
\\
&+\int_{\Bbb R}\beta^{-2}\,\phi_{4,zz}\,w_x\,{\rm d}x+\int_{\Bbb R}M_{32}(x,z)\,w_x\,{\rm d}x+\int_{\Bbb R}M_{42}(x,z)\,w_x\,{\rm d}x
\\
&+\int_{\Bbb R}M_{52}(x,z)\,w_x\,{\rm d}x+\int_{\Bbb R}M_{62}(x,z)\,w_x\,{\rm d}x+\int_{\Bbb R}\big[B_3(\phi_3)+B_3(\phi_4)\big]\,w_x\,{\rm d}x
\\
\,\equiv\,&\, \textrm{I}_1\,+\,\textrm{I}_2\,+\,\textrm{I}_3\,+\,\textrm{I}_4\,+\,\textrm{I}_5
\,+\,\textrm{I}_6\,+\,\textrm{I}_7\,+\,\textrm{I}_8\,+\,\textrm{I}_9\,+\,\textrm{I}_{10}.
\end{split}
\end{equation}
These terms will be estimated as follows.

\medskip
By repeating the same computation used in (\ref{J1}) and (\ref{J2}), we get
\begin{equation}
\label{I1}
\begin{split}
\textrm{I}_{1}\,=\,&\int_{\Bbb R}\epsilon^2\,S_3\,w_x\,{\rm d}x
\\
\,=\,&-\epsilon^2\, \varrho_{1}\,\beta^{-1}\,f{''}
\,-\,
\epsilon^2\, \varrho_{1}\,\beta^{-1}\,(\beta^{-1}\,\beta'+2\,\alpha^{-1}\,\alpha')\,f{'}
-\epsilon^2\, \varrho_{1}\,\beta^{-1}\Big(k^2-\sigma \,\beta^{-2}\,V_{tt}\Big)\,f,
\end{split}
\end{equation}
and also
\begin{equation}\label{I2}
 \textrm{I}_{2}\,=\,\int_{\Bbb R}\,\epsilon^2\,S_5\,w_x\,{\rm d}x\,=\,-\epsilon^2\, \varrho_{1}\,\beta^{-1}\,\big[\,\varpi\,\beta'\beta^{-1}+2\,\varpi\,\alpha^{-1}\,\alpha'\big]f.
\end{equation}
Recall the expression of $B_2(w)$ in (\ref{B2v}) and $\phi_1$, it is easy to check that
\begin{equation}\label{I3}
\textrm{I}_{3}+\textrm{I}_{4}\,=\,\int_{\Bbb R}B_2(w)\,w_x\,{\rm d}x
\,+\int_{\Bbb R}\beta^{-2}\,\phi_{1,zz}\,w_x\,{\rm d}x\,=\,\epsilon^3 \,{\mathbf b}_{1\epsilon}\,f{''}
+\epsilon^3\,{\mathbf b}_{2\epsilon}.
\end{equation}
Since
$$
\phi_4(x,z)=\epsilon^2 \phi_{41}\big(x, {\epsilon}z\big)+\epsilon^2\phi_{42}(x,{\epsilon}z),
$$
it can be derived  that
\begin{equation}
\label{I5}
  \textrm{I}_5\,=\,\int_{\Bbb R}\beta^{-2}\,\phi_{4,zz}\,w_x\,{\rm d}x\,=\,O(\epsilon^4).
\end{equation}
Since $w_x$ is odd in $x$, we need only consider the odd terms in $M_{32}(x,z)$ and get
\begin{align}
\begin{aligned}
\label{I6}
\textrm{I}_6\,=\,&\int_{\Bbb R}M_{32}(x,z)\,w_x\,{\rm d}x
\\
\,=\,&\epsilon^2\,\beta^{-1}\,\sigma^{-1}\,k^2\,f\int_{\Bbb R}\Big[w_{2,x}\,w_x+\frac {1}{\sigma}\,x\,w_x\,w_2+w_1\,w_x\Big]\,{\rm d}x+O(\epsilon^3).
\end{aligned}
\end{align}
From the definition of $M_{42}(x,z)$,  we can estimate the $I_7$ in (\ref{intE1wx}) as the following
\begin{align}
\textrm{I}_7=&\int_{\Bbb R}\,M_{42}(x,z)\,w_x\,{\rm d}x
\nonumber
\\
=&-\,\frac{2\epsilon^2}{\beta}\,f'\,\xi(\epsilon z)\int_{\Bbb R}\,\big[\epsilon\,A'(\mathfrak{a}(\epsilon z))\,\beta\, Z_x+\phi_{22,xz}\big]\,w_x\,{\rm d}x
\label{I7}
\\
&-\,\frac{2\epsilon^2}{\beta}\,\varpi\,f\,\xi(\epsilon z)\int_{\Bbb R}\,\big[\epsilon\,A'(\mathfrak{a}(\epsilon z))\,\beta\, Z_x+\phi_{22,xz}\big]\,w_x\,{\rm d}x+\epsilon^3\, {\mathbf b}_{1\epsilon}\,f{''}+\epsilon^3\,{\mathbf b}_{2\epsilon}
\nonumber
\\
\equiv&\,  \epsilon^2\,  \varrho_{1}\,\beta^{-1}\,\alpha_1(z)\,f'
\,+\,\epsilon^2\, \varrho_{1}\,\beta^{-1}\,\alpha_2(z)\,f+\epsilon^3 \,{\mathbf b}_{1\epsilon}\,f{''}+\epsilon^3\,{\mathbf b}_{2\epsilon},
\nonumber
\end{align}
where $\alpha_1(z)$ and $\alpha_2(z)$ are defined in (\ref{alpha1})-(\ref{alpha2}).

\medskip
From the definition of $M_{52}(x,z)$, we need only consider the odd terms and the higher order terms involving $e'$ and $e''$, so we get
\begin{equation}
\label{I8}
\begin{split}
 \textrm{I}_8=&-\epsilon^2\,k\,\beta^{-1}\,e\int_{\Bbb R}\big[\,Z_{x}+\sigma^{-1}\,x\,Z\big]\,w_x\,{\rm d}x+2\,\epsilon^4\, k\,\beta^{-3}\,e{''}(\epsilon z)\m x\,w_x\,Z(x)\,{\rm d}x
 \\
  &\qquad+\epsilon^3\,\Big[{\mathbf b}_{1\epsilon}^1\,e'+ {\mathbf b}_{1\epsilon}\,f{''}+{\mathbf b}_{2\epsilon}\Big]
  \\
  \,\equiv\,&\epsilon^2 \varrho_{1}\beta^{-1}\Big[\hbar_{31}({\epsilon}z)\,e+\epsilon^2\hbar_4({\epsilon}z)\,e{''}\Big]
+\epsilon^3\Big[{\mathbf b}_{1\epsilon}^1\,e'+{\mathbf b}_{1\epsilon}\,f{''}+{\mathbf b}_{2\epsilon}\Big].
  \end{split}
\end{equation}
It can easily be verified that
\begin{equation}
\label{I9}
\begin{split}
\textrm{I}_9\,=&\,\int_{\Bbb R}M_{62}(x,z)\,w_x\,{\rm d}x
\\
=&\epsilon^2\,\beta^{-1}\,\sigma^{-1}\,k^2\,f\,\int_{\Bbb R}p(p-1)\,w_x\,w^{p-2}\,w_1\,w_2\,{\rm d}x
\\
&+\e^2p(p-1)a_{11}\,e\int_{\Bbb R}w^{p-2}\,w_1\,Z\,
   w_x\,{\rm d}x+\epsilon^3{\mathbf b}_{2\epsilon}.
\\
=&\epsilon^2\,\beta^{-1}\,\sigma^{-1}\,k^2\,f\,\int_{\Bbb R}p(p-1)\,w_x\,w^{p-2}\,w_1\,w_2\,{\rm d}x+\epsilon^2 \varrho_{1}\beta^{-1}\,\hbar_{32}({\epsilon}z)\,e+\epsilon^3{\mathbf b}_{2\epsilon}
\end{split}
\end{equation}
According to the fact that the terms in $B_3(\phi_3)$ and $B_3(\phi_4)$ are of order $O(\e^3)$,
it follows that
\begin{equation}
\label{I10}
\textrm{I}_{10}=\int_{\Bbb R}\Big[B_3(\phi_3)+B_3(\phi_4)\Big]\,w_x\,{\rm d}x\,=\,\epsilon^3\, {\mathbf b}_{1\epsilon}\,f{''}+\epsilon^3\,{\mathbf b}_{2\epsilon}.
\end{equation}

\medskip
Adding (\ref{I6}) and (\ref{I9}), we get
\begin{equation*}
\begin{split}
\textrm{I}_6+ \textrm{I}_9\,&=\,\int_{\Bbb R}M_{32}(x,z)\,w_x\,{\rm d}x+\int_{\Bbb R}M_{62}(x,z)\,w_x\,{\rm d}x
\\
&=\epsilon^2\,\beta^{-1}\,\sigma^{-1}\,k^2\,f\,\int_{\Bbb R}\Big[p(p-1)w^{p-2}\,w_x\,w_1\,w_2+w_{2,x}\,w_x+\sigma^{-1}\,x\,w_x\,w_2
  \\
  &\qquad\qquad\qquad\qquad\qquad+w_1w_x\Big]\,{\rm d}x+\epsilon^2 \varrho_{1}\beta^{-1}\,\hbar_{32}({\epsilon}z)\,e+\big[\epsilon^3 {\mathbf b}_{1\epsilon}f{''}+\epsilon^3{\mathbf b}_{2\epsilon}\Big]
\\
&=\,-\epsilon^2 \varrho_{1}\beta^{-1}\sigma^{-1}k^2f\,+\epsilon^2 \varrho_{1}\beta^{-1}\,\hbar_{32}({\epsilon}z)\,e+\,\epsilon^3\big[{\mathbf b}_{1\epsilon}f{''}+{\mathbf b}_{2\epsilon}\big].
\end{split}
\end{equation*}
This leads to  the conclusion of this section: there holds
\begin{equation}
\label{wxjifen}
\begin{split}
\i {\mathcal E}w_x\,{\rm d}x\,=\,&-\epsilon^2\, \varrho_{1}\,\beta^{-1}\,\Big\{f{''}
+\big[\hbar_1(\epsilon z)+\alpha_1(z)\big]\,f'+[\hbar_2(\epsilon z)+\alpha_2( z)]\,f\Big\}
\\
&+\epsilon^2\, \varrho_{1}\,\beta^{-1}\,\big[\hbar_3(\epsilon z)\,e+\epsilon^2\,\hbar_4({\epsilon}z)\,e{''}\big]+\epsilon^3\,\Big[{\mathbf b}_{1\epsilon}^1\,e{'}+{\mathbf b}_{1\epsilon}^2\,f{''}+{\mathbf b}_{2\epsilon}\Big],
\end{split}
\end{equation}
the functions $\hbar_1$ and $\hbar_2$ are defined in (\ref{gamma1})-(\ref{gamma2}), $\hbar_{3}=\hbar_{31}+\hbar_{32}$.

\medskip
\subsubsection{}

In the following, we are in a position to provide the concise estimate for the integral $\int_{\Bbb R}{\mathcal E}\,Z\,{\rm d}x $.
Using the decomposition of ${\mathcal E}$,  we obtain
$$\int_{\Bbb R}{\mathcal E}\,Z\,{\rm d}x \,=\,\int_{\Bbb R}{\mathcal E}_{11}\,Z\,{\rm d}x +\int_{\Bbb R}{\mathcal E}_{12}\,Z\,{\rm d}x ,$$
where
\begin{equation*}
 \int_{\Bbb R}{\mathcal E}_{11}\,Z\,{\rm d}x \,=\,\epsilon\,\big[\epsilon^2\,\beta^{-2}\,e{''}+\lambda_0\,e\,\big]\int_{\Bbb R}\,Z^2\,{\rm d}x \,=\,\epsilon^3\,\beta^{-2}\,e{''}+\epsilon\lambda_0\,e,
\end{equation*}
and
\begin{equation}
\label{E12Z}
\begin{split}
\m  {\mathcal E}_{12}\,Z\,{\rm d}x
\,=\,&\int_{\Bbb R}\epsilon^2\,S_4\,Z\,{\rm d}x+\int_{\Bbb R}B_2(w)\,Z\,{\rm d}x+\int_{\Bbb R}\beta^{-2}\phi_{1,zz}\,Z\,{\rm d}x
\\
&+\int_{\Bbb R}\beta^{-2}\,\phi_{4,zz}\,Z\,{\rm d}x+\int_{\Bbb R}M_{12}(x,z)\,Z\,{\rm d}x+\int_{\Bbb R}M_{22}(x,z)\,Z\,{\rm d}x
\\
&+\int_{\Bbb R}M_{32}(x,z)\,Z\,{\rm d}x+\int_{\Bbb R}M_{42}(x,z)\,Z\,{\rm d}x+\int_{\Bbb R}M_{52}(x,z)\,Z\,{\rm d}x
\\
&+\int_{\Bbb R}M_{62}(x,z)\,Z\,{\rm d}x+\int_{\Bbb R}B_3(\phi_3)\,Z\,{\rm d}x+\int_{\Bbb R}B_3(\phi_4)\,Z\,{\rm d}x
\\
\,\equiv\,& \textrm{II}_1\,+\,\textrm{II}_2\,+\,\textrm{II}_3\,+\,\textrm{II}_4\,+\,\textrm{II}_5\,
+\,\textrm{II}_6\,+\,\textrm{II}_7\,+\,\textrm{II}_8\,+\,\textrm{II}_9
+\textrm{II}_{10}+\textrm{II}_{11}+\textrm{II}_{12}.
\end{split}
\end{equation}

\medskip
According to the expression of $S_4$ in (\ref{sw-gather}) and the constraint of $f$ in (\ref{fnorm}), it follows that
\begin{equation}
\begin{split}
\label{II1}
&\textrm{II}_1\,=\,\int_{\Bbb R}\,\epsilon^2\,S_4\,Z\,{\rm d}x
\\
&
=\epsilon^2\int_{\Bbb R}\big[ (f'^2+2f'h')w_{xx}+2\,\varpi\,(ff'+f'h)w_{xx}-\frac{1}{2\beta^{2}}\,V_{tt}\,f^2\,w\big]Z\,{\rm d}x
%\\
%&=\epsilon^2\int_{\Bbb R}(2f'h'w_{xx}+2\,\varpi\,f'h\,w_{xx})Z\,{\rm d}x+O(\epsilon^3)
\\
&=\e^2\,\varrho_{2}\,f'\,h'+\e^2\,\varrho_{2}\,\varpi\,f'h+O(\epsilon^3),
\end{split}
\end{equation}
where
$$
\varrho_{2}\,=2\,\int_{\Bbb R}\,w_{xx}\,Z\,{\rm d}x.
$$
The estimate of $\textrm{II}_7$ can be proved by the same way as employed in the above estimate.
%In the same way, due to the assumption on $f$ and the expression of $M_{32}(x,z)$ in (\ref{M32}), we can obtain that
\begin{equation}
\label{II7}
 \textrm{II}_7=\int_{\Bbb R}M_{32}(x,z)\,Z\,{\rm d}x
=\int_{\Bbb R}\Big(\epsilon^2\,k^2\,\sigma^{-2}\,f^2\,w_2+\epsilon^3\,a_7({\epsilon}s,
{\epsilon}z)\Big)\,Z\,{\rm d}x
=\,\epsilon^3\,{\mathbf b}_{1\epsilon}\,f{''}+\epsilon^3\,{\mathbf b}_{2\epsilon}.
\end{equation}

\medskip
Note that $B_2(w)=O(\epsilon^3)$, it is easy to check that
\begin{equation}
\label{II2}
\textrm{II}_{2}\,=\,\m B_2(w)Z\,{\rm d}x \,=\,\epsilon^3\,{\mathbf b}_{1\epsilon}\,f{''}+\epsilon^3\,{\mathbf b}_{2\epsilon}.
\end{equation}
Since
 $$\phi_1(x, z)=\epsilon\,a_{11}({\epsilon}z)\,w_1(x)+\epsilon\,a_{12}({\epsilon}z)\,(f({\epsilon}z)+h({\epsilon}z))\,w_2(x),$$
  so we obtain
\begin{equation}
\label{II3}
\begin{split}
\textrm{II}_{3}&\,=\,\beta^{-2}\m \phi_{1,zz}Z\,{\rm d}x
\\
&\,=\,\epsilon^3 \,\beta^{-2}\,a_{12}(\epsilon z)\,f{''}\m w_2(x)\,Z(x)\,{\rm d}x+\epsilon^3 \,{\mathbf b}_{2\epsilon}
\\
&\,=\,\epsilon^3\,{\mathbf b}_{1\epsilon}\,f''+\epsilon^3 \,{\mathbf b}_{2\epsilon}.
\end{split}
\end{equation}
Since
$$\phi_4(x,z)=\epsilon^2\, \phi_{41}(x,{\epsilon}z)+\epsilon^2\,\phi_{42}(x,{\epsilon}z),$$
 it is easy to prove that
\begin{equation}
\label{II4}
\textrm{II}_{4}\,=\,\beta^{-2}\,\m \phi_{4,zz}\,Z\,{\rm d}x\,=\,O(\epsilon^4).
\end{equation}
According to the expression of $M_{12}(x,z)$ and $M_{22}(x,z)$, we know that the terms in $M_{12}(x,z)$ and $M_{22}(x,z)$ are of order $O(\e^3)$. Hence, it is easy to obtain that
\begin{align}
  \textrm{II}_{5}&=\int_{\Bbb R}M_{12}(x,z)\,Z\,{\rm d}x=O(\epsilon^3),\label{II5}
  \\
\textrm{II}_{6}&=\int_{\Bbb R}M_{22}(x,z)\,Z\,{\rm d}x=O(\epsilon^3). \label{II6}
\end{align}

\medskip
What's more, we can compute that
\begin{equation}
\label{II8}
\begin{split}
\textrm{II}_{8}&=\int_{\Bbb R}M_{42}(x,z)\,Z(x)\,{\rm d}x
\\
&=-\,2\epsilon^2\,\beta^{-1}\big(f'+\varpi\,f\big)\,\xi(\epsilon z)\,\int_{\Bbb R}\,\Big[\epsilon\,A'\big(\mathfrak{a}(\epsilon z)\big)\,\beta\, Z_x+\phi_{22,xz}\Big]\,Z\,{\rm d}x+O(\epsilon^3)
\\
&=O(\epsilon^3).
\end{split}
\end{equation}
We need only to compute those parts in $M_{52}(x,z)$ which are even in $x$. It is easy to check that
\begin{equation}
\label{II9}
 \textrm{II}_{9}=\int_{\Bbb R}M_{52}(x,z)\,Z\,{\rm d}x=-\epsilon^2\beta^{-1}\,k\,e\int_{\Bbb R}\big(Z_{x}+\sigma^{-1}x\,Z\big)\,Z\,{\rm d}x+O(\epsilon^3)=O(\epsilon^3).
\end{equation}
Additionally, we also need to consider some higher order terms in $\textrm{II}_{9}$. The ones involving first derivative of $e$ are
\begin{equation}
\label{M121important}
\begin{split}
2\epsilon^3\,e'&\,\Big(\frac{\beta'}{\beta^{3}}-\frac{\varpi}{\beta^{2}}\,\,\Big)\m xZZ_x\,{\rm d}x+\epsilon^3\,e'\,\Big(\frac{2\alpha'}{\alpha\beta^{2}}-\frac{\varpi}{\beta^{2}}\,\Big)\m Z^2\,{\rm d}x
\\
&\,=\,\epsilon^3\,\Big(\frac {2\alpha'}{\alpha \beta^2}-\frac {\beta'}{\beta^3}\Big)e'
\,\equiv\, \epsilon^3\,\hbar_5({\epsilon}z)\,e',
\end{split}
\end{equation}
\noindent where
\begin{equation}
\label{gamma5}
\hbar_5({\epsilon}z)\,=\,\frac {2\alpha'({\epsilon}z)}{\alpha({\epsilon}z) \beta^2({\epsilon}z)}-\frac {\beta'({\epsilon}z)}{\beta^3({\epsilon}z)}.
\end{equation}
Moreover, the ones involving second derivative of $e$ in $\textrm{II}_{9}$ are
$$\epsilon^3\Big[\epsilon\, f\,\beta^{-2}\,\hbar_6(\epsilon z)+O(\epsilon^2)\Big]\,e{''}(\epsilon z)$$
with $O(\epsilon^2)$ uniform in $\epsilon$.

\medskip
In the terms of $\textrm{II}_{10}$ and $\textrm{II}_{12}$, we need only to consider those parts which are even in $x$. It is best that the even (in $x$) terms in $\textrm{II}_{10}$ and $\textrm{II}_{12}$ are of  order $o(\epsilon^3)$.
Moreover, the terms in $B_3(\phi_3)$ are of order $O(\epsilon^3)$. Consequently, we deduce that
\begin{equation}
  \textrm{II}_{10}+\textrm{II}_{11}+\textrm{II}_{12}=\int_{\Bbb R}M_{62}(x,z)\,Z\,{\rm d}x+\int_{\Bbb R}B_3(\phi_3)\,Z\,{\rm d}x+\int_{\Bbb R}B_3(\phi_4)\,Z\,{\rm d}x=O(\epsilon^3).
\end{equation}

\medskip
 In summary, we have established that, as a function of $\theta=\epsilon z$,
\begin{equation}
\begin{split}
\label{E1Z}
\m {\mathcal E}Z\,{\rm d}x\,=&\,\epsilon^3\,\Big[1+\epsilon\, f\,\hbar_6({\epsilon}z)+O(\epsilon^2)\Big]\,\beta^{-2}\,e{''}+\epsilon^3 \,\hbar_5({\epsilon}z)\,e'
\\
&\quad+\e^2\,\varrho_{2}\,f'\,h'+\e^2\,\varrho_{2}\,\varpi\,f'\,h+\epsilon \,\lambda_0 \, e+O(\epsilon^3),
\end{split}
\end{equation}
where $\hbar_6({\epsilon}z)$ is a  smooth functions of their argument.

\subsection{Projection of terms involving $\phi$}
We will estimate the terms that involve $\phi$ in (\ref{c=0})-(\ref{d=0}) integrated against the functions $w_x$ and $Z$ in the variable $x$. Concerning
$w_x$,  we denote by $\Lambda(\phi)$  the sum of these terms, which can be decomposed as $\Lambda(\phi)\,=\,\sum_{i=1}^3\Lambda_i(\phi)$.

\medskip
Let $\Lambda_1(\epsilon z)\,=\,\m \chi(\epsilon|x|) B_3(\phi)w_x\,{\rm d}x$.
We make the following observation: all terms in $B_3(\phi)$ carry $\epsilon$ and involve power
of $x$ times derivatives of 1,\;2 orders of $\phi$.  The conclusion is that since $w_x$ has exponential decay, then
$$
\int_0^1|\Lambda_1(\theta)|^2{\rm d}\theta\,\leq\, C\epsilon^3\|\phi\|^2_{H^2(\mathcal S)}.
$$

\noindent Hence,
\begin{equation}
\label{Lambda1}
\|\Lambda_1\|_{L^2(0,1)}\,\leq\, C\epsilon^3.
\end{equation}

\medskip
We shall analyze the properties of the operator $\Lambda_1$ acting on the pair $(f,e)$ in $H^2(0,1)$.  We single out two less regular terms which
are operators depending Lipschitz continuously on $(f,e)$.  The one whose coefficient depends on $f{''}$ explicitly has the form
\begin{equation*}
\begin{split}
\Lambda_{1*}\,=\,&-\,\epsilon^2\,\beta^{-1}\,f{''}\m \phi_x\, w_x\,{\rm d}x
\,-\,
2\,\epsilon^3\,k\,\beta^{-1}\,f{''}\m \Big(\frac{x}{\beta}+f+h\Big)\,\phi_x\, w_x\,{\rm d}x
\\
&
\,-\,\epsilon^4\,\beta^{-1}\,a_1(\epsilon s,\epsilon z)\,f{''}
\m \Big(\frac{x}{\beta}+f+h\Big)^2\,\phi_x\,w_x\,{\rm d}x.
\end{split}
\end{equation*}
Since $\phi$ has Lipschitz dependence on $(f,e)$ in the form (\ref{characteriztion}), from Sobolev's embedding, we derive that
$$\|\phi(f_1,e_1)-\phi(f_2,e_2)\|_{L^{\infty}(\mathcal S)}\,\leq\, C \epsilon^{3/2}\big[\,\|f_1-f_2\|_{*}+\|e_1-e_2\|_{**}\,\big].$$
Hence,
\begin{equation}
\label{Lambda1lip}
\|\Lambda_{1*}(f_1,e_1)-\Lambda_{1*}(f_2,e_2)\|_{L^2(0,1)}\,\leq\, C^{3+ 1/2}\big[\,\|f_1-f_2\|_{*}+\|e_1-e_2\|_{**}\,\big].
\end{equation}
Another one comes from second derivative of $\phi$ in $z$
$$\Lambda_{1**}\,=\,2\,\epsilon\, k\alpha\m \Big(\frac{x}{\beta}+f+h\Big)\phi_{zz}\,w_x\,{\rm d}x+\epsilon^2\, a_1(\epsilon s,\epsilon z)\alpha \m \Big(\frac{x}{\beta}+f+h\Big)^2\,\phi_{zz}\,w_x\,{\rm d}x.$$
Then
 \begin{equation}
 \label{Lamda1**}
 \|\Lambda_{1**}(f_1,e_1)-\Lambda_{1**}(f_2,e_2)\|_{L^2(0,1)}\,\leq\, C\,\epsilon^3\, \Big[\|f_1-f_2\|_{*}+\|e_1-e_2\|_{**}\Big].
 \end{equation}

\medskip
For fixed $\epsilon$,  the remainder $\Lambda_1-\Lambda_{1*}-\Lambda_{1**}$ actually defines a compact operator of the pair $(f,e)$
from $H^2(0,1)$ into $L^2(0,1)$.  This is a consequence of the fact that weak convergence in $H^2(\mathcal S)$ implies local strong convergence in $H^1(\mathcal S)$,  and
the same is the case for $H^2(0,1)$ and $C^1[0,1]$. If $f_j$ and $e_j$ are bounded sequences in $H^2(0,1)$,  then clearly the functions $\phi(f_j,e_j)$
 constitute a bounded sequence in $H^2(\mathcal S)$.  In the above remainder we can integrate by parts once in $x$
 if necessary. Averaging against $w_x$ which decays exponentially, localizes the situation and the desired result follows.

\medskip
 We also observe that
 $$
 \Lambda_2(\epsilon z)\,=\,\m \eta^\epsilon_{\delta}(s)\, {\mathcal N}(\phi)w_x\,{\rm d}x,
 $$
 can be estimated similarly. Using the definition of ${\mathcal N}(\phi)$ and
 the exponential decay of $w_x$ we get
 $$
 \|\Lambda_2\|_{L^2(0,1)}\,\leq\, C\epsilon^{ 1/2}\|\phi\|_{H^2(\mathcal S)}\,\leq\, C\epsilon^3.
 $$

 \medskip
Finally,  let  us consider
 $$
 \Lambda_3(\epsilon z)\,=\,\m p\big(\beta^{-2}\chi(\e |x|){\mathbf W}^{p-1}-w^{p-1}\big)\phi \,w_x\,{\rm d}x.
 $$
 Since
 $$
v_5\,=\,w\,+\,\phi_1\,+\,\phi_2\,+\,\epsilon e Z\,+\,\phi_3\,+\,\phi_4
 \quad
 \mbox{ and }
 \quad
 \phi_1\,+\,\phi_2\,+\,\epsilon e Z+\,\phi_3\,+\,\phi_4
 $$
 can be estimated as
 $$\,|\phi_1|\,+\,|\phi_2|\,+\epsilon|e\,Z|\,+\,|\phi_3|\,+\,|\phi_4|\,\leq\, C\epsilon (1+|x|^2)e^{-|x|},$$
 we obtain that for some $\rho>0$ the following uniform bound holds
 $$\big |\big(\beta^{-2}\chi(\e |x|){\mathbf W}^{p-1}-w^{p-1}\big)w_x\big|\,\leq\, C\,\epsilon\, e^{-\rho |x|}.$$
 From here we find that
 $$\|\Lambda_3\|_{L^2(0,1)}\,\leq\, C\,\epsilon^{3/2}\|\phi\|_{H^2(\mathcal S)}\,\leq\, C \epsilon^3.$$
 These two terms $\Lambda_2$ and $\Lambda_3$ also define
 compact operators similarly as before.

\medskip
We observe that exactly the same estimates can be carried out in the terms obtained from integration against
 $Z$.

\subsection{Projection of errors on the boundary}
In this section we compute the projection of errors on the boundary.
Without loss of generality,  only the projections of the error components on $\partial_0{\mathcal S}$ will be given.

\medskip
The main errors on the boundary integrated against $w_x$ and $Z$ in the
variable $x$ can be computed as the following:
\begin{align*}
%\label{g0boundary-new}
\m& g_0(x)w_x\,{\rm d}x\,
\\
=&\,-\epsilon \beta\big(k_1f+f'\big)\m w_x^2\,{\rm d}x
\,-\,\epsilon^2b_5\,a_{11}\m xw_{1,x}w_x\,{\rm d}x
\\
&-\epsilon^2 \beta\big(k_1f+f'\big)\m \Big \{a_{12}(f+h)w_{2,x}+eZ_x+\beta^{-1}\big[A(0)Z_x+\phi_{22}(x,0)\big]\Big\}w_x\,{\rm d}x
\\
&-\epsilon^2 \,\m \Big\{ b_1\,\Big(\beta^{-2}x^2+h^2+2\,f\,h\Big)\beta-k\,
\big[\beta^{-2}\,\beta'\,x^2-(f\,h'+h\,f'+h\,h')\beta\big]\Big\}\,w_x^2\,{\rm d}x
\\
&+\epsilon^2(\alpha'\alpha^{-1}a_{11}+a_{11}')\m w_1w_x\,{\rm d}x
\,+\,\epsilon^2k\alpha'\,\alpha^{-1}\beta^{-1}\m xww_x\,{\rm d}x
\\
&+\epsilon^2k\beta^{-2}\m\Big[\epsilon A'(0)\beta(0)Z(x)+\phi_{22,z}(x,0)\Big]xw_x\,{\rm d}x+O(\epsilon^3).
\end{align*}
%Similarly
%\begin{equation}
%\label{g1boundary-new}
%\begin{split}
%\m& g_1(x)\,w_x\,{\rm d}x\,
%\\
%=&\,-\epsilon\,\beta(1)\,\big[k_2\,f(1)+f'(1)\big]\m w_x^2\,{\rm d}x
%\,-\,\epsilon^2\,b_6\,a_{11}(1)\m x\,w_{1,x}\,w_x\,{\rm d}x
%\\
%&-\epsilon^2\,\beta(1)\,\big[k_2\,f(1)+f'(1)\big]
%\m \Big \{a_{12}\,(f+h)\,w_{2,x}+e(1)\,Z_x
%\\
%&\qquad\qquad\qquad\qquad\qquad\qquad\qquad+\beta^{-1}(1)\big[A({\ell})\,Z_x+\phi_{22}(x,1/ %\epsilon)\big]\Big\}\,w_x\,{\rm d}x
%\\
%&-\epsilon^2 \,\m \Big\{ b_3\,\beta\Big(x^2\,\beta^{-2}+h^2+2\,f\,h\Big)-k(1)\,\beta
%\big[\beta'\,\beta^{-2}\,x^2-(f\,h'+h\,f'+h\,h')\big]\Big\}\,w_x^2\,{\rm d}x
%\\
%&+\epsilon^2\,(\alpha'\,\alpha^{-1}\,a_{11}+a_{11}')\m w_1\,w_x\,{\rm d}x
%\,+\,\epsilon^2\,\beta^{-1}(1)\,k(1)\,\alpha'(1)\,\alpha(1)^{-1}\m x\,w\,w_x\,{\rm d}x
%\\
%&+\epsilon^2\,\beta^{-2}(1)\,k(1)\m\Big[\epsilon \,A'({\ell})\,\beta(1)\,Z(x)+\phi_{22,z}(x,1/ %\epsilon)\Big]\,x\,w_x\,{\rm d}x+O(\epsilon^3).
%\end{split}
%\end{equation}
Using the following formulas
\begin{equation*}
\m Z^2(x)\,{\rm d}x\,=\,-2\m x\,Z_x(x)\,Z(x)\,{\rm d}x\,=\,1,
\end{equation*}
we get the following two estimates
\begin{equation*}
%\label{g0z-new}
\begin{split}
\m g_0(x)Z(x)\,{\rm d}x\,=\,&\frac{\epsilon^2}{2}b_5\,e(0)+\epsilon^2 \big[\alpha'(0)\alpha^{-1}(0)e(0)+e'(0)\big]
+O(\epsilon^3).
\end{split}
\end{equation*}
%\begin{equation}
%\label{g1z-new}
%\begin{split}
%\m g_1(x)Z(x)\,{\rm d}x\,=\,&\frac{\epsilon^2}{2}b_6\,e(1)+\epsilon^2 %\big[\alpha'\alpha^{-1}e(1)+e'(1)\big]
%+O(\epsilon^3).
%\end{split}
%\end{equation}

\medskip
Higher order errors can be proceeded as follows:
%\begin{equation}
%\label{D31wx}
%\begin{split}
%\m &D_3^1\big(\phi(x,1/\epsilon)\big)w_x\,{\rm d}x\,
%\\
%&=\,
%\epsilon\m \Big[b_6\,x+\beta (k_2f+f')\Big]\phi_{x}(x,1/\epsilon)w_x\,{\rm d}x
%\\
%&\quad-\epsilon\alpha'\alpha^{-1}\m \phi(x,1/\epsilon)w_x\,{\rm d}x
%\,-\,\epsilon k(1)\m \Big(\frac{x}{\beta}+f+h\Big)\phi_{z}(x,1/\epsilon)w_{x}\,{\rm d}x
%\\
%&\quad+\epsilon^2\m \Big[b_3\Big(\frac{x}{\beta}+f+h\Big)^2\beta-k(1)\Big(\frac{x}{\beta}+f+h\Big) \Big(\frac{\beta'}{\beta}x-\beta f-\beta h)\Big)\Big]\phi_{x}(x,1/\epsilon)w_{x}\,{\rm d}x
%\\
%&\quad-\epsilon^2k(1)\frac{\alpha'}{\alpha}\m\Big(\frac{x}{\beta}+f+h\Big)\phi(x,1/\epsilon)w_x\,{\rm %d}x
%\,+\,\epsilon^2 b_4\m \Big(\frac{x}{\beta}+f+h\Big)^2\phi_z(x,1/\epsilon)w_{x}\,{\rm d}x
%\\
%&=\, O(\epsilon^{\frac 52}).
%\end{split}
%\end{equation}
\begin{align*}
\m& D_3^0(\phi(x,0))w_x\,{\rm d}x\,
\nonumber
\\
=&\,
\epsilon\m \big[b_5\,x+\beta (k_1f+f')\big]\phi_{x}(x,0)w_x\,{\rm d}x
\nonumber
\\
&-\epsilon\alpha'\alpha^{-1}\m \phi(x,0)w_x\,{\rm d}x
\,-\,\epsilon k\m\Big(\frac{x}{\beta}+f+h\Big)\phi_{z}(x,0)w_{x}\,{\rm d}x
%\label{D30wx}
\\
&+\epsilon^2\m \Big[b_1\Big(\frac{x}{\beta}+f+h\Big)^2\beta-k\Big(\frac{x}{\beta}+f+h\Big) \Big(\frac{\beta'}{\beta}x-\beta f'- \beta h'\Big)\Big]\phi_{x}(x,0)w_{x}\,{\rm d}x
\nonumber
\\
&-\epsilon^2k\frac{\alpha'}{\alpha}\m\Big(\frac{x}{\beta}+f+h\Big)\phi(x,0)w_x\,{\rm d}x
\,+\,\epsilon^2 b_2\m  \Big(\frac{x}{\beta}+f+h\Big)^2\phi_z(x,0)w_{x}\,{\rm d}x
\nonumber
\\
=&\, O(\epsilon^{\frac 52}),
\nonumber
\end{align*}
%\begin{equation}
%\label{D31Z}
%\begin{split}
%\m& D_3^1(\phi(x,1/\epsilon))Z\,{\rm d}x\,
%\\
%&=\,
%\epsilon\m \Big[b_6\,x+\beta (k_2f+f')\Big]\phi_{x}(x,1/\epsilon)Z\,{\rm d}x
%\\
%&-\epsilon\alpha'\alpha^{-1}\m \phi(x,1/\epsilon)Z\,{\rm d}x
%\,-\,\epsilon k(1)\m \Big(\frac{x}{\beta}+f+h\Big)\phi_{z}(x,1/\epsilon)Z\,{\rm d}x
%\\
%&+\epsilon^2\m \Big[b_3\Big(\frac{x}{\beta}+f+h\Big)^2\beta-k(1)\Big(\frac{x}{\beta}+f+h\Big) %\Big(\frac{\beta'}{\beta}x-\beta f-\beta h\Big)\Big]\phi_{x}(x,1/\epsilon)Z\,{\rm d}x
%\\
%&-\epsilon^2k(1)\frac{\alpha'}{\alpha}\m\Big(\frac{x}{\beta}+f+h\Big)\phi(x,1/\epsilon)Z\,{\rm d}x
%\,+\,\epsilon^2 b_4\m \Big(\frac{x}{\beta}+f+h\Big)^2\phi_z(x,1/\epsilon)Z\,{\rm d}x
%\\
%&=\, O(\epsilon^{\frac 52}).
%\end{split}
%\end{equation}
and also
\begin{equation*}
%\label{D30Z}
\begin{split}
\m & D_3^0(\phi(x,0))Z\,{\rm d}x\,
\\
=&\,
\epsilon\m \big[b_5\,x+\beta (k_1f+f')\big]\phi_{x}(x,0)Z\,{\rm d}x
\\
&-\epsilon\alpha'\alpha^{-1}\m \phi(x,0)Z\,{\rm d}x
\,-\,\epsilon k\m \Big(\frac{x}{\beta}+f+h\Big)\phi_{z}(x,0)Z\,{\rm d}x
\\
&+\epsilon^2\m \Big[b_1\Big(\frac{x}{\beta}+f+h\Big)^2\beta
-k\Big(\frac{x}{\beta}+f+h\Big)\Big(\frac{\beta'}{\beta}x-\beta f'-\beta h'\Big)\Big]\phi_{x}(x,0)Z\,{\rm d}x
\\
&-\epsilon^2k\frac{\alpha'}{\alpha}\m\Big(\frac{x}{\beta}+f+h\Big)\phi(x,0)Z\,{\rm d}x
\,+\,\epsilon^2 b_2\m \Big(\frac{x}{\beta}+f+h\Big)^2\phi_z(x,0)Z\,{\rm d}x
\\
 =&\,O(\epsilon^{\frac 52}).
\end{split}
\end{equation*}
The term $D_0^0(\phi)$~on the boundary integrated against $w_x$ and $Z$ in the variable $x$ are of size of order $O(\epsilon^3)$.

\section{The system for $(f,e):$ proof of the theorem }
\label{section7}
\setcounter{equation}{0}

Using the estimates in previous section and $\theta=\epsilon z$,  and then defining the operators
\begin{align}
{\mathbb L}_1(f)&\,\equiv\, f{''}
\,+\,\big[\hbar_1(\theta)\,+\,\alpha_1(\theta/\epsilon)\big]f'\,+\,\big[\hbar_2(\theta)\,+\,\alpha_2( \theta/\epsilon)\big]f,
 \\
{\mathbb L}_2(e)&\,\equiv\,\epsilon^2\,(\beta^{-2}\,e{''}\,+\,\hbar_5(\theta)\,e')\,+\, \lambda_0 e,
\end{align}
where $\hbar_1(\theta)$, $\hbar_2(\theta)$, $\alpha_1(z)$ and $\alpha_2(z)$ are smooth functions defined in (\ref{gamma1}), (\ref{gamma2}), (\ref{alpha1}) and (\ref{alpha2}),
we find the following nonlinear system of differential equations for the parameters $(f,e)$
\begin{align}
{\mathbb L}_1(f)
  &\,=\,\hbar_3(\theta)\,e\,+\,\epsilon^2\,\hbar_4(\theta)\,e{''}\,+\,\epsilon \,{\mathcal M}_{1,\epsilon}, \label{f}
  \\
  {\mathbb L}_2(e)
  &\,=\,\epsilon^3\,f\,\hbar_6(\theta)\,e{''} \,+{\epsilon}\,\varrho_{2}\,f'\,h'+{\epsilon}\,\varrho_{2}\,\varpi \,f'\,h+\,\epsilon^2 {\mathcal M}_{2,\epsilon}, \label{e}
\end{align}
with the boundary conditions
\begin{equation}
\label{boundary condition 1}
  f'(1)\,+\,k_2\,f(1)\,+\,{\mathcal M}_1^1(f,e)\,=\,0,
\end{equation}
\begin{equation}
\label{boundary condition 2}
  f'(0)\,+\,k_1\,f(0)\,+\,{\mathcal M}_1^2(f,e)\,=\,0,
\end{equation}
\begin{equation}
\label{boundary condition 3}
  e'(1)\,+\,\tilde{b}_6\,e(1)\,+\,{\mathcal M}_2^1(f,e)\,=\,0,
\end{equation}
\begin{equation}
\label{boundary condition 4}
  e'(0)\,+\,\tilde{b}_5\,e(0)\,+\,{\mathcal M}_2^2(f,e)\,=\,0.
\end{equation}
The constants ${\tilde b}_6$ and ${\tilde b}_5$ are given by
$$\tilde{b}_6=\frac{b_6}{2}+\frac{\alpha'(1)}{\alpha(1)},
 \qquad
\tilde{b}_5=\frac{b_5}{2}+\frac{\alpha'(0)}{\alpha(0)},
 $$
where $b_5$ and $b_6$ are constants defined in (\ref{b5}) and (\ref{b6}).
${\mathcal M}_{j}^{i}$'s are some terms of order $O(\epsilon^{1/2})$.
The operators ${\mathcal M}_{1,\epsilon}$ and ${\mathcal M}_{2,\epsilon}$ can be decomposed in the following forms
\begin{equation*}
  {\mathcal M}_{l,\epsilon}(f,e)\,=\,\mathbf{A}_{l,\epsilon}(f,e)\,+\,\mathbf{K}_{l,\epsilon}(f,e), \ \ \  \ l=1,2
\end{equation*}
where $\mathbf{K}_{l,\epsilon}$'s are uniformly bounded in $L^{2}(0,1)$ for $(f,e)$ in $\mathcal{F}$ and are also compact.
The operators $\mathbf{A}_{l,\epsilon}$, $i=1, 2$, are Lipschitz in this region, see (\ref{Lambda1lip})-(\ref{Lamda1**})
\begin{equation}
\label{Lipschitz of A}
\|\mathbf{A}_{l,\epsilon}(f_1,e_1)-\mathbf{A}_{l,\epsilon}(f_2,e_2)\|_{L^{2}(0,1)}\,\leq\, C\,\big[\,\|f_1-f_2\|_{*}\,+\,\|e_1-e_2\|_{**}\,\big].
\end{equation}

\medskip
Before solving (\ref{f})-(\ref{boundary condition 4}), some basic facts about the invertibility of corresponding linear operators will be derived. Firstly, we consider the following problem
\begin{equation}
\label{equation of f}
  \begin{split}
  f{''}(\theta)\,+\,\big(\hbar_1(\theta)\,+\,\alpha_1&(\theta/\epsilon)\big)f'(\theta)\,+\,\big( \hbar_2(\theta)\,+\,\alpha_2(\theta/\epsilon)\big)f(\theta)\,=\,\mathbf{g}(\theta),
     \\
    f'(0)+k_1\,f(0)&=0,\qquad\qquad\quad f'(1)+k_2\,f(1)=0.
  \end{split}
\end{equation}
 \begin{lemma}\label{lemma of f}
{\rm\textbf{(1).}}
Under the non-degenerate condition (\ref{nondegeneracy}), if ${\mathbf g}\in L^{2}(0,1)$, then there is a constant $\epsilon_0$ for each $0<\epsilon<\epsilon_{0}$ satisfying (\ref{gapcondition}), such that problem (\ref{equation of f}) has a unique solution $f\in H^2(0,1)$, which satisfies
 $$
 \|f\|_{*}\leq C\|{\mathbf g}\|_{L^2(0,1)}.
 $$
\\
{\rm\textbf{(2).}}
Moreover, if
$$
{\mathbf g}(\theta)=G_1(\theta/\epsilon)+G_2(\theta/\epsilon)+G_3(\theta/\epsilon),
$$
where $G_1$, $G_2$ and $G_3$ are given in (\ref{G1}), (\ref{G2}) and (\ref{G3}), then
$$
 \|f\|_{*}\leq C\e^{\frac{1}{2}}.
$$
 \end{lemma}
 \begin{proof}
(1). Under the non-degeneracy condition (\ref{nondegeneracy}), the existence part comes from the a priori estimate and the continuity method. Hence, we focus on the derivations of the estimates.

\medskip
First of all, by the relation
$$
\vartheta\,=\,\frac{ \mathfrak{a}(\theta)}{\ell},
$$
where $\ell$ and $a(\theta)$ are given in (\ref{ellnumber}) and (\ref{a(theta)}), we make the following transformation:
\begin{equation*}
 \eta(\vartheta)\,=\,f(\theta),\qquad \hat{\beta}(\vartheta)\,=\,\beta(\theta),\qquad \hat{{\mathbf g}}(\vartheta)={\mathbf g}(\theta),
\end{equation*}
\begin{equation*}
 \hat{\hbar}_1(\vartheta)\,=\,\hbar_1(\theta),
 \qquad
 \hat{\hbar}_2(\vartheta)\,=\,\hbar_2(\theta),
 \qquad
\hat{ \varpi}(\vartheta)\,=\, \varpi(\theta).
\end{equation*}
 Then (\ref{equation of f}) gets transformed into
 \begin{equation}
 \label{equation of f1}
  \begin{split}
  \eta_{\vartheta\vartheta}
  \,+\,\frac{\hat{\beta}'(\vartheta)}{\hat{\beta}(\vartheta)}\,\eta_\vartheta
  \,+\,\frac{\hat{\hbar}_1(\vartheta)\ell}{\hat{\beta}(\vartheta)}\,\eta_\vartheta
  &\,+\,\frac{\hat{\hbar}_2(\vartheta)\,\ell^2}{\hat{\beta}^2(\vartheta)}\,\eta
  \,+\,\frac{\hat{\alpha}_1(\vartheta)\ell}{\hat{\beta}(\vartheta)}\,\eta_\vartheta
  \,+\,\frac{\hat{\alpha}_2(\vartheta)\ell^2}{\hat{\beta}^2(\vartheta)}\,\eta\,=\,\frac{\ell^2}{\hat{\beta}^2(\vartheta)}\,\hat{{\mathbf g}}(\vartheta),
  \\
   \eta'(1)\,+\,k_2\,\eta(1)&\,=\,0, \qquad \eta'(0)\,+\,k_1\,\eta(0)\,=\,0.
  \end{split}
 \end{equation}
Here the terms $\hat{\alpha}_1$ and $\hat{\alpha}_2$ are define as
 \begin{equation}
 \label{equation a1}
    \begin{split}
    \hat{\alpha}_1(\vartheta)&\,=\,\hat{\bar{\alpha}}_1(\vartheta)+\hat{\tilde{\alpha}}_1(\vartheta),
 \qquad
     \hat{\alpha}_2(\vartheta)\,=\,\hat{\bar{\alpha}}_2(\vartheta)+\hat{\tilde{\alpha}}_2(\vartheta),
\end{split}
\end{equation}
where we have denoted
$$
\hat{\bar{\alpha}}_2(\vartheta)=\hat{\varpi}(\vartheta)\,\hat{\bar{\alpha}}_1(\vartheta),
\qquad
\hat{\tilde{\alpha}}_2(\vartheta)=\hat{\varpi}(\vartheta)\,\hat{\tilde{\alpha}}_1(\vartheta),
$$
\begin{equation*}
    \hat{\bar{\alpha}}_1(\vartheta)\,=2 \varrho_{1}^{-1} \,\epsilon\, \hat{\chi}(\vartheta)\,
    A'({\ell}\vartheta)\,\hat{\beta} (\vartheta)\,\int_{\Bbb R}\,Z_x\,w_x\,{\rm d}x,
\end{equation*}
\begin{equation*}
    \hat{\tilde{\alpha}}_1(\vartheta)\,=2 \varrho_{1}^{-1}\,
    \hat{\chi}(\vartheta)\,
    \displaystyle{\int_{\Bbb R}}\phi_{22,xz}\Bigg(x,\frac{ \mathfrak{a}^{-1}(\ell \vartheta)}{\e}\Bigg)\,w_x\,{\rm d}x,
\end{equation*}
and
\begin{equation*}
   \hat{\chi}(\vartheta)\,\equiv\, \frac{\,\hat{\chi}_0(\vartheta)\,}{ \beta(0)}\,
   +\,\frac{\,1-\hat{\chi}_0(\vartheta)\,}{\,\beta(1)}.
\end{equation*}
In the above, $\mathfrak{a}^{-1}$ is the inverse of the map $\mathfrak{a}$ given in (\ref{a(theta)}).

\medskip
Define the operator ${\mathbb L}_{0}$  in the form
\begin{equation*}
  {\mathbb L}_{0}(\eta):\,=\,\eta''(\vartheta)\,+\,\mathfrak{q}_1(\vartheta)\,\eta'(\vartheta)\,
  +\,\mathfrak{q}_2(\vartheta)\,\eta,
\end{equation*}
where
\begin{equation*}
 \mathfrak{q}_1(\vartheta)\,=\,\frac{\hat{\beta}'(\vartheta)}{\hat{\beta}(\vartheta)}\,
 +\,\frac{\hat{\hbar}_1(\vartheta)\,{\ell}}{\hat{\beta}(\vartheta)},\qquad
 \mathfrak{q}_2(\vartheta)\,=\,\frac{\hat{\hbar}_2(\vartheta)\,{\ell}^2}{\hat{\beta}^2(\vartheta)}.
\end{equation*}
There exists an orthonormal basis of $L^2(0,1)$ constituted by eigenfunctions $\{y_j\}$, associated to the eigenvalues $\{\lambda_j\}$, of the following eigenvalue problem
\begin{equation}
\label{expression of eigenvalue}
   \begin{split}
   -\,{\mathbb L}_{0}\big(y(\vartheta)\big)&\,=\,\lambda \,y(\vartheta),\ \ \ \ \  0<\vartheta<1,
   \\
   y'(1)\,+\,k_2\,y(1)&\,=\,0, \ \ \ \ \ \ \ y'(0)\,+\,k_1\,y(0)\,=\,0.
    \end{split}
\end{equation}
 The result in \cite{LS} shows that, as $j\rightarrow \infty$,
 \begin{equation*}
\sqrt{\lambda_j}=j\pi+\frac{k_2-k_1}{j\pi}+O\Big(\frac{1}{j^3}\Big).
 \end{equation*}
 It is easy to see that there exists a positive constant $C$ such that $|y'_j(\vartheta)|\leq C j$ for all $j\in \mathbb{N}$.

 \medskip
 We then expand
 \begin{align*}
 \frac{{\ell}^2}{\hat{\beta}^2(\vartheta)}\,\hat{{\mathbf g}}(\vartheta)&\,=\,\sum_{j=0}^{\infty}\,{\mathbf g}_j\,y_j(\vartheta),
 \qquad\qquad\eta(\vartheta)\,=\,\sum_{j=0}^{\infty}\,a_j\,y_j(\vartheta),
 \\
  \frac{\hat{\bar{\alpha}}_1{\ell}}{\hat{\beta}(\vartheta)}\,\eta'(\vartheta)&\,=\,\sum_{j=0}^{\infty}\,\bar{d}_j\,y_j(\vartheta), \qquad \frac{\hat{\bar{\alpha}}_2{\ell}^2}{\hat{\beta}^2(\vartheta)}\,\eta(\vartheta)\,=\,\sum_{j=0}^{\infty}\,\bar{c}_j\,y_j(\vartheta),
  \\
  \frac{\hat{\tilde{\alpha}}_1{\ell}}{\hat{\beta}(\vartheta)}\,\eta'(\vartheta)&\,=\,\sum_{j=0}^{\infty}\tilde{d}_j\,y_j(\vartheta),\ \qquad \frac{\hat{\tilde{\alpha}}_2{\ell}^2}{\hat{\beta}^2(\vartheta)}\,\eta(\vartheta)\,=\,\sum_{j=0}^{\infty}\tilde{c}_j\,y_j(\vartheta),
  \\
   \frac{\hat{\alpha}_1{\ell}}{\hat{\beta}(\vartheta)}\,\eta'(\vartheta)&\,=\,\sum_{j=0}^{\infty}\,d_j\,y_j(\vartheta),\qquad \frac{\hat{\alpha}_2{\ell}^2}{\hat{\beta}^2(\vartheta)}\,\eta(\vartheta)\,=\,\sum_{j=0}^{\infty}\,c_j\,y_j(\vartheta).
 \end{align*}
Now, we will estimate the Fourier coefficients above by using the asymptotic behaviors of the following two terms
\begin{equation*}
  \Phi_{j,1}(\vartheta)=\int_{0}^{\vartheta}\sin\Big(\frac{\sqrt\lambda_{0}}{\epsilon}{\ell}s\Big)\,y_j(s)\,{\mathrm d}s,
  \qquad
   \Phi_{j,2}(\vartheta)=\int_{0}^{\vartheta}\cos\Big(\frac{\sqrt\lambda_{0}}{\epsilon}{\ell}s\Big)\,y_j(s)\,{\mathrm d}s.
\end{equation*}
In the following formula, by using the equation for $y_j$ and integrating by parts two times
\begin{align*}
  \Phi_{j,1}(\vartheta)&=-\frac{\epsilon^2}{\lambda_0{\ell}^2}\int_{0}^{\vartheta}
  \Big[\sin\Big(\frac{\sqrt\lambda_{0}}{\epsilon}\,{\ell}\,s\Big)\Big]''y_j(s)\,{\rm d}s
  \\
  &=-\frac{\epsilon}{\sqrt{\lambda_0}\,{\ell}}\Big[\cos\Big(\frac{\sqrt{\lambda_0}}{\epsilon}\,{\ell}\,\vartheta\Big)\,y_j(\vartheta)\,-\,y_j(0)\Big]\,+\,
  \frac{\epsilon^2}{\lambda_0{\ell}^2}\,\sin\Big(\frac{\sqrt{\lambda_0}}{\epsilon}\,{\ell}\,\vartheta\Big)\,y'_j(\vartheta)
  \\
  &\qquad-\frac{\epsilon^2}{\lambda_0{\ell}^2}\int_{0}^{\vartheta}\sin\Big(\frac{\sqrt\lambda_{0}}{\epsilon}{\ell}s\Big)\,y''_{j}(s)\,{\rm d}s
  \\
  &=-\frac{\epsilon}{\sqrt{\lambda_0}{\ell}}\Big[\cos\Big(\frac{\sqrt{\lambda_0}}{\epsilon}\,{\ell}\,\vartheta\Big)\,y_j(\vartheta)\,-\,y_j(0)\Big]\,+\,
  \frac{\epsilon^2}{\lambda_0{\ell}^2}\,\sin\Big(\frac{\sqrt{\lambda_0}}{\epsilon}\,{\ell}\,\vartheta\Big)\,y'_j(\vartheta)
  \\
 &  \qquad-\frac{\epsilon^2}{\lambda_0{\ell}^2}\int_{0}^{\vartheta}\sin\Big(\frac{\sqrt\lambda_{0}}{\epsilon}\,{\ell}\,s\Big)
 \big[y''_{j}(s)\,+\,\mathfrak{q}_1(s)\,y'_{j}(s)+\mathfrak{q}_2(s)\,y_{j}(s)\big]\,{\rm d}s
 \\
  &  \qquad\quad+\frac{\epsilon^2}{\lambda_0{\ell}^2}\int_{0}^{\vartheta}\sin\Big(\frac{\sqrt\lambda_{0}}{\epsilon}\,{\ell}\,s\Big)
  \big[\mathfrak{q}_1(s)\,y'_{j}(s)\,+\,\mathfrak{q}_2(s)\,y_{j}(s)\big]\,{\rm d}s
 \\
 &=\frac{\epsilon^2}{\lambda_0{\ell}^2}\int_{0}^{\vartheta}
 \sin\Big(\frac{\sqrt\lambda_{0}}{\epsilon}{\ell}s\Big)\lambda_j y_j(s)\,{\mathrm d}s
 \,-\,\frac{\epsilon}{\sqrt{\lambda_0}{\ell}}
 \Big[\cos\Big(\frac{\sqrt{\lambda_0}}{\epsilon}\,{\ell}\,\vartheta\Big)\,y_j(\vartheta)-y_j(0)\Big]
  \\
&\quad\,+\,\frac{\epsilon^2}{\lambda_0{\ell}^2}\,\sin\Big(\frac{\sqrt{\lambda_0}}{\epsilon}\,{\ell}\,\vartheta\Big)\,y'_j(\vartheta)\,+\,\frac{\epsilon^2}{\lambda_0{\ell}^2}
 \int_{0}^{\vartheta}\sin\Big(\frac{\sqrt\lambda_{0}}{\epsilon}{\ell}s\Big)
 \big[\mathfrak{q}_1(s)\,y'_{j}(s)+\mathfrak{q}_2(s)\,y_{j}(s)\big]\,{\rm d}s,
\end{align*}
we can get
 \begin{equation*}
  \begin{split}
   \Big(1-\frac{\epsilon^2\lambda_j}{\lambda_0{\ell}^2}\Big)\,\Phi_{j,1}(\vartheta)
   =-&\frac{\epsilon}{\sqrt{\lambda_0}{\ell}}\,\Big[\cos\Big(\frac{\sqrt{\lambda_0}}{\epsilon}\,{\ell}\,\vartheta\Big)\,y_j(\vartheta)\,-\,y_j(0)\Big]
   \,+\,
  \frac{\epsilon^2}{\lambda_0{\ell}^2}\,\sin\Big(\frac{\sqrt{\lambda_0}}{\epsilon}\,{\ell}\,\vartheta\Big)\,y'_j(\vartheta)
  \\
  &\,+\,\frac{\epsilon^2}{\lambda_0{\ell}^2}\int_{0}^{\vartheta}\sin\Big(\frac{\sqrt\lambda_{0}}{\epsilon}\,{\ell}\,s\Big)
  \big[\mathfrak{q}_1(s)\,y'_{j}(s)+\mathfrak{q}_2(s)\,y_{j}(s)\big]\,{\rm d}s.
    \end{split}
 \end{equation*}
This implies that
\begin{equation}\label{9.5}
  |\Phi_{j,1}(\vartheta)|\leq \frac{C\epsilon(\epsilon j+1)}{|\lambda_0{\ell}^2-\epsilon^2\lambda_j|}.
\end{equation}
Similarly, it can be derived
\begin{equation}\label{9.6}
  |\Phi_{j,2}(\vartheta)|\leq \frac{C\epsilon(\epsilon j+1)}{|\lambda_0{\ell}^2-\epsilon^2\lambda_j|}.
\end{equation}

\medskip
By using the gap condition (\ref{gapcondition}), we can write $\hat{\bar{\alpha}}_2$ in the form
\begin{equation*}
  \hat{\bar{\alpha}}_2=\bar{K}_1\,\hat{\varpi}(\vartheta) \hat{\chi}(\vartheta)\,\hat{\beta}(\vartheta)\,
   \sin\Big(\frac{\sqrt{\lambda_0}}{\epsilon}\,{\ell}\,\vartheta\Big)\,+\,\bar{K}_2\, \hat{\varpi}(\vartheta) \hat{\chi}(\vartheta)\,\hat{\beta}(\vartheta)
   \cos\Big(\frac{\sqrt{\lambda_0}}{\epsilon}\,{\ell}\,\vartheta\Big),
\end{equation*}
for parameters $\bar{K}_1 ,\bar{K}_2 $ depending on $\epsilon$, which are also bounded by a universal constant independent of $\epsilon$.
Hence,
\begin{equation*}
\begin{split}
  \bar{c}_j&=\int_{0}^{1}\Big[\bar{K}_1
   \sin\Big(\frac{\sqrt{\lambda_0}}{\epsilon}\,{\ell}\,\vartheta\Big)\,+\,\bar{K}_2\,
   \cos\Big(\frac{\sqrt{\lambda_0}}{\epsilon}\,{\ell}\,\vartheta\Big)\Big]
   \,\hat{\varpi}(\vartheta) \,\hat{\chi}(\vartheta)\,\frac{{\ell}^2}{\hat{\beta}(\vartheta)}\,\eta(\vartheta)\,y_j(\vartheta)\,{\rm d}\vartheta
   \\
   &=\bar{K}_1\,{\ell}^2\,\int_{0}^{1}\,\hat{\varpi}(\vartheta) \hat{\chi}(\vartheta)\,\hat{\beta}^{-1}(\vartheta)\,\Phi'_{j,1}(\vartheta)\,\eta(\vartheta)\,{\rm d}\vartheta
   \,+\,
   \bar{K}_2\,{\ell}^2\,\int_{0}^{1}\,\hat{\varpi}(\vartheta) \hat{\chi}(\vartheta)\,\hat{\beta}^{-1}(\vartheta)\,\Phi'_{j,2}(\vartheta)\,\eta(\vartheta)\,{\rm d}\vartheta.
   \end{split}
\end{equation*}
Integrating once by parts and using (\ref{9.5})-(\ref{9.6}), we obtain
\begin{equation}
\label{cn}
 |\bar{c}_j|\leq \frac{C\epsilon(\epsilon j +1)}{|\lambda_0{\ell}^2-\epsilon^2\lambda_j|}\Big\{\|\eta'\|_{L^2}+\|\eta\|_{L^2}\Big\}.
\end{equation}
A similar approach will imply
\begin{equation*}
 |\bar{d}_j|\leq \frac{C\epsilon(\epsilon j +1)}{|\lambda_0{\ell}^2-\epsilon^2\lambda_j|}\Big\{\|\eta'\|_{L^2}+\|\eta\|_{L^2}\Big\}.
\end{equation*}
The estimate of $\tilde{d}_j$ can be showed as
\begin{align}
  |\tilde{d}_j|&=\Bigg|\int_{0}^{1}\frac{\hat{\tilde{\alpha}}_1{\ell}}{\hat{\beta}(\vartheta)}\,\eta'(\vartheta)\,y_j(\vartheta)\,{\rm d}\vartheta\Bigg|
  \nonumber
  \\
  &=2 \varrho_{1}^{-1}{\ell}\,\Bigg|\displaystyle{\int_{0}^{1}\m}\hat{\chi}(\vartheta)\,\hat{\beta}^{-1}(\vartheta)\,
  \eta'(\vartheta)\,\phi_{22,xz}\Bigg(x,\frac{ \mathfrak{a}^{-1}(\ell \vartheta)}{\e}\Bigg)\,w_x\,y_j(\vartheta)\,{\rm d}x\,{\rm d}\vartheta\Bigg|
  \label{dn}
  \\
  &\leq C\,\Bigg[\int_{0}^{1}\m\big|\hat{\chi}(\vartheta)\,\hat{\beta}^{-1}(\vartheta)\,\eta'(\vartheta)\,y_j(\vartheta)\,w_x\big|^2\,{\rm d}x\,{\rm d}\vartheta\Bigg]^{\frac{1}{2}}
  \Bigg[\int_{0}^{1}\m\Big|\phi_{22,xz}\Bigg(x,\frac{ \mathfrak{a}^{-1}(\ell \vartheta)}{\e}\Bigg)\Big|^2\,{\rm d}x\,{\rm d}\vartheta\Bigg]^{\frac{1}{2}}
    \nonumber
  \\
  &\leq C \,\e^{\frac{1}{2}}\,\|\eta'\|_{L^2},  \nonumber
\end{align}
where $\mathfrak{a}^{-1}$ is the inverse of the map $\mathfrak{a}$ given in (\ref{a(theta)}).
Similar estimates hold for $\tilde{c}_j$
\begin{equation*}
 |\tilde{c}_j|\leq C \,\e^{\frac{1}{2}}\,\|\eta\|_{L^2}.
\end{equation*}

\medskip
Whence,
from the equations
\begin{equation*}
  \begin{split}
  -\lambda_j a_j\,+\,c_j\,+\,d_j\,=\,{\mathbf g}_j,
  \\
  \bar{c}_j\,+\,\tilde{c}_j\,=\,c_j,
  \quad
   \bar{d}_j\,+\,\tilde{d}_j\,=\,d_j,
  \end{split}
\end{equation*}
and estimates of the Fourier coefficients $\bar{c}_j$, $\bar{d}_j$, $\tilde{c}_j$ and $\tilde{d}_j$, we get
\begin{equation}
\label{an}
  |a_j|\,\leq\,\Big|\frac{{\mathbf g}_j}{\lambda_j}\Big|+ \frac{C\epsilon(\epsilon j+1 )}{|\lambda_j(\lambda_0{\ell}^2-\epsilon^2\lambda_j)|}\times\{\|\eta'\|_{L^2}+\|\eta\|_{L^2}\}
  \,+\,
  \frac{C \,\e^{\frac{1}{2}}}{\lambda_j}\times\{\|\eta'\|_{L^2}+\|\eta\|_{L^2}\}.
\end{equation}
Moreover, from the asymptotic expression of $\lambda_j$ in (\ref{expression of eigenvalue}), it can be shown
\begin{equation*}
 \|\eta'\|_{L^2}^{2}+\|\eta\|_{L^2}^{2}\,\leq\, C\, \|{\mathbf g}\|_{L^2}^{2}+C_*\,\sum_{j}\frac{j^2\epsilon^2(\epsilon j+1)^2}{|\lambda_j|^2(\lambda_0{\ell}^2-\epsilon^2\lambda_j)^2}
 \times\{ \|\eta'\|_{L^2}^{2}+\|\eta\|_{L^2}^{2}\},
\end{equation*}
where the positive constant $C_*$ does not depend on $\epsilon$.
%
%\medskip
From the asymptotic expression of $\lambda_j$,
there exists a positive constant $\epsilon_0$ such that, for all positive $\epsilon<\epsilon_0$  satisfying (\ref{gapcondition}), elementary analysis will imply the following estimates
\begin{equation*}
  \begin{split}
  \sum\limits_{2\epsilon^2\lambda_j\geq3\lambda_0{\ell}^2}\frac{j^2\epsilon^2(\epsilon j+1)^2}{|\lambda_j|^2(\lambda_0{\ell}^2-\epsilon^2\lambda_j)^2}&\,\leq\,C\,\epsilon^2,
  \\
  \sum\limits_{\lambda_0{\ell}^2<2\epsilon^2\lambda_j<3\lambda_0{\ell}^2}\frac{j^2\epsilon^2(\epsilon j+1)^2}{|\lambda_j|^2(\lambda_0{\ell}^2-\epsilon^2\lambda_j)^2}&\,\leq\,C\,\epsilon,
  \\
  \sum\limits_{2\epsilon^2\lambda_j\leq\lambda_0{\ell}^2}\frac{j^2\epsilon^2(\epsilon j+1)^2}{|\lambda_j|^2(\lambda_0{\ell}^2-\epsilon^2\lambda_j)^2}&\,\leq \,C\,\epsilon^2.
\end{split}
\end{equation*}
Hence, we can prove that
\begin{equation*}
 \|\eta'\|_{L^2}^{2}+\|\eta\|_{L^2}^{2}\,\leq\, C \,\|{\mathbf g}\|_{L^2}^{2}.
\end{equation*}
Since
\begin{equation*}
  \|\eta''\|_{L^2}\,\leq\, C\,\big( \|{\mathbf g}\|_{L^2}+ \|\eta'\|_{L^2}+\|\eta\|_{L^2}^{2}\big),
\end{equation*}
the final result then follows easily.

\medskip
(2). If ${\mathbf g}(\theta)=G_1(\theta/\epsilon)+G_2(\theta/\epsilon)+G_3(\theta/\epsilon)$, using the same argument as the proof of formulas (\ref{cn}) and (\ref{dn}), we can easily carry out the estimate of ${\mathbf g}_j$ in the form
$$
|{\mathbf g}_j|\leq \frac{C\epsilon(\epsilon j+1)}{|\lambda_0{\ell}^2-\epsilon^2\lambda_j|}+C \e^{\frac{1}{2}}.
$$
For all positive $\epsilon<\epsilon_0$ satisfying (\ref{gapcondition}), a similar analysis will imply the following estimates
\begin{equation*}
  \begin{split}
  \sum\limits_{2\epsilon^2\lambda_j\geq3\lambda_0{\ell}^2}\frac{\epsilon(\epsilon j+1)}{|\lambda_j(\lambda_0{\ell}^2-\epsilon^2\lambda_j)|}&\,\leq\,C \e^{\frac{1}{2}},
  \\
  \sum\limits_{\lambda_0{\ell}^2<2\epsilon^2\lambda_j<3\lambda_0{\ell}^2}\frac{\epsilon(\epsilon j+1)}{|\lambda_j(\lambda_0{\ell}^2-\epsilon^2\lambda_j)|}&\,\leq\,C \e^{\frac{1}{2}},
  \\
  \sum\limits_{2\epsilon^2\lambda_j\leq\lambda_0{\ell}^2}\frac{\epsilon(\epsilon j+1)}{|\lambda_j(\lambda_0{\ell}^2-\epsilon^2\lambda_j)|}&\,\leq \,C\,\epsilon^{\frac{1}{2}}.
\end{split}
\end{equation*}
Therefore, from (\ref{an}) and the estimate of ${\mathbf g}_j$, we get
\begin{equation*}
 \|\eta'\|_{L^2}^{2}+\|\eta\|_{L^2}^{2}
 \,\leq\,
 C\, {\epsilon}+C_*\,\sum_{j}\frac{j^2\epsilon^2(\epsilon j+1)^2}{|\lambda_j|^2(\lambda_0{\ell}^2-\epsilon^2\lambda_j)^2}
 \times\{ \|\eta'\|_{L^2}^{2}+\|\eta\|_{L^2}^{2}\},
\end{equation*}
So, we can prove that
\begin{equation*}
 \|\eta'\|_{L^2}+\|\eta\|_{L^2}\,\leq\, C \,\e^{\frac{1}{2}},
\end{equation*}
the final result then follows easily.
   \end{proof}

\medskip
Secondly, we consider the following problem
\begin{equation}
\label{equation of e}
  \begin{split}
  \,\epsilon^2\,\big[\beta^{-2}\,e{''} \,+\,\hbar_5(\theta)\,e'\big]\,+\,\lambda_0 \,e \,=\,\tilde{{\mathbf g}},
  \qquad\forall\, 0<\theta<1,
  \\
    e'(0)\,+\,\tilde{b}_5\,e(0)\,=\,0,
    \qquad\quad
   e'(1)\,+\,\tilde{b}_6\,e(1)\,=\,0.
  \end{split}
\end{equation}
\begin{lemma}\label{lemma of e}
 If $\tilde{{\mathbf g}}\in L^{2}(0,1)$, then for all small $\epsilon$ satisfying (\ref{gapcondition}) there is a unique solution $e\in H^2(0,1)$ to problem (\ref{equation of e}), which satisfies
 $$\|e\|_b\leq C\,\epsilon^{-1}\,\|\tilde{{\mathbf g}}\|_{L^2(0,1)}.
 $$
 Moreover, if $\tilde{{\mathbf g}}\in H^{2}(0,1)$, then
 \begin{equation}
\epsilon^2\,\|e''\|_{L^2(0,1)}\,+\,\epsilon\,\|e'\|_{L^2(0,1)}\,+\,\|e\|_{L^\infty(0,1)}\,\leq\, C\, \|\tilde{{\mathbf g}}\|_{H^2(0,1)}.
 \end{equation}
 \end{lemma}

\begin{proof}
See Lemma 8.1 in \cite{dPKW}.
\end{proof}

\medskip
Thirdly, we consider the following system
\begin{equation}
\label{L(f,e)}
\begin{split}
 {\mathbb L}(f,e)\,\equiv\,({\mathbb L}_1(f),\, {\mathbb L}_2(e))\,=\,({\mathbf g}(\theta),\tilde{{\mathbf g}}(\theta)),\qquad 0<\theta<1,
   \\
    f'(1)\,+\,k_2\,f(1)\,=\,\Gamma_{1}^{1},
    \qquad
      f'(0)\,+\,k_1\,f(0)\,=\,\Gamma_{0}^{1},
      \\
      e'(1)\,+\,\tilde{b}_6\,e(1)\,=\,\Gamma_{1}^{0},
      \qquad
       e'(0)\,+\,\tilde{b}_5\,e(0)\,=\,\Gamma_{0}^{0},
\end{split}
\end{equation}
where $\Gamma_{j}^{i},i,\,j=0,\,1$ are some constants.
\begin{lemma}
\label{lemma}
Under the non-degenerate condition (\ref{non-degeneracy}), if ${\mathbf g},\, \tilde{{\mathbf g}}\in L^2(0,1)$, then there exists $\epsilon_0$ such that for all $0<\epsilon<\epsilon_0$ satisfying (\ref{gapcondition}) there is a unique solution $(f,e)$ in $H^2(0,1)$ to problem (\ref{L(f,e)}) which satisfies
\begin{equation*}
  \|f\|_{*}\,+\,\|e\|_{**}\leq \Big[\|{\mathbf g}\|_{L^2(0,1)}\,+\,\epsilon^{-1}\|\tilde{{\mathbf g}}\|_{L^2(0,1)}\,+\,\sum_{i,j=0}^{1}|\Gamma_{j}^{i}|\Big].
  \end{equation*}
\end{lemma}
\begin{proof}
See Lemma 7.3 in \cite{wei-yang}.
\end{proof}

\medskip
{\textbf {Proof of Theorem \ref{theorem 1.1}.}} Let us observe now that  the linear operator
\begin{equation*}
  {\mathbb L}^*(f,e)=\big({\mathbb L}_1(f)-\hbar_3(\epsilon z)e-\epsilon^2\hbar_4(\epsilon z)e{''}, \, {\mathbb L}_2(e)\big),
\end{equation*}
is invertible with bounds for ${\mathbb L}^*(f,e)=({\mathbf g},\tilde{{\mathbf g}})$ given by
\begin{equation*}
   \|f\|_{*}\,+\,\|e\|_{**}\leq C\|{\mathbf g}\|_{L^2(0,1)}\,+\,\epsilon^{-1}\|\tilde{{\mathbf g}}\|_{L^2(0,1)}.
\end{equation*}
From (\ref{Lipschitz of A}), $\epsilon\,\mathbf{A}_{1,\epsilon}$ and $\epsilon^2\,\mathbf{A}_{2,\epsilon}$ are contraction mappings of their arguments in $\mathcal{F}$. By Banach Contraction Mapping Theorem and Lemma \ref{lemma}, we can solve the nonlinear problem
\begin{equation*}
\big[{\mathbb L}^*-(\epsilon\,\mathbf{A}_{1,\epsilon},\epsilon^2\,\mathbf{A}_{2,\epsilon} )\big](f,e)=({\mathbf g},\tilde{{\mathbf g}})
\end{equation*}
with the boundary conditions defined in (\ref{boundary condition 1})-(\ref{boundary condition 4}) is uniquely solvable for $(f,e)\in \mathcal{F}$ provided that
$$
\|{\mathbf g}\|_{L^{2}(0,1)}\,<\,\epsilon^{1/2+\rho}
\quad
\mbox{and}
\quad
\|\tilde{{\mathbf g}}\|_{L^{2}(0,1)}\,<\,\epsilon^{3/2+\rho},
 $$
with some small positive constants $\rho$.
 The desired result for full problem (\ref{f})-(\ref{boundary condition 4}) then follows directly from Schauder's fixed-point Theorem.
\qed

\bigskip
\bigskip
%\begin{acknowledgements}
{\bf Acknowledgements: }
J. Yang is supported by  NSFC(No.11371254).
Part of this work was done when the authors visited Chern Institute of Mathematics, Nankai University in summer of 2014:
we are very grateful to the institution for the kind hospitality.
\qed
%\end{acknowledgements}

\medskip
\begin{appendices}

\section{Proofs of  Lemmas \ref{lemma2.1}  and \ref{derivativeofF} }\label{appendixA}
\setcounter{equation}{0}

{\textbf{Proof of Lemma \ref{lemma2.1}: }}
In fact, the curves can be expressed in the following forms
\begin{align}\label{curves}
{\mathcal C}_1:&~{\mathbb H}\big({\tilde t},{\tilde \varphi}_1({\tilde t})\big)
\,=\,\g\big({\tilde \varphi}_1({\tilde t})\big)+{\tilde t}\,n\big({\tilde \varphi}_1({\tilde t})\big),
\\
{\mathcal C}_2:&~{\mathbb H}\big({\tilde t},{\tilde \varphi}_2({\tilde t})\big)
\,=\,\g\big({\tilde \varphi}_2({\tilde t})\big)+{\tilde t}\,n\big({\tilde \varphi}_2({\tilde t})\big),
\\
\G:&~{\mathbb H}(0,\t{\theta})\,=\,\g(\t{\theta}).
\end{align}
It follows that the tangent vectors of ${\mathcal C}_1$ at $P_1$ can be written as
$$
\frac {{\mathrm d} {\mathcal C}_1}{{\mathrm d}{\tilde t}}\Big|_{{\tilde t}=0}
\,=\,\frac {\partial \g}{\partial \t{\theta}}\Big|_{{\tilde\theta}\,=\,{{\tilde\varphi}_1(0)}}
\cdot \frac {{\mathrm d} {{\tilde\varphi}_1}}{{\mathrm d}  {\tilde t}}\Big|_{{\tilde t}=0}\,+\,n\big({\tilde \varphi}_1(0)\big),
$$
and the tangent vector of $\Gamma$ at $P_1$ is
$$
\frac {\partial \g}{\partial \t{\theta}}\Big|_{\t{\theta}\,=\,{\tilde \varphi}_1(0)}.
$$
According to the condition:~$\G\bot \partial \Omega$,
we have that
$$
\lf< \frac {{\mathrm d}{\mathcal C}_1}{{\mathrm d}{\tilde t}}\Big|_{{\tilde t}=0},\,\, \frac {\partial \g}{\partial \t{\theta}}\Big|_{\t{\theta}={\tilde \varphi}_1(0)}\ri>=0.
$$
Combining with
 $$\lf<n({\tilde \varphi}_1(0)),\,\,\frac {\partial \g}{\partial \t{\theta}}\Big|_{\t{\theta}={\tilde \varphi}_1(0)}\ri>=0,$$
we have
$$ \t{\varphi}_1'(0)=0.$$
Similarly, we can show ${\tilde \varphi}_2'(0)=0$.

\medskip
In the coordinate system $(y_1,y_2)$,
$\gamma(\tilde{\theta})$ and $n(\tilde{\theta})$ can be expressed as follows
\begin{equation*}
  \gamma(\tilde{\theta})=\big(\gamma_1(\tilde{\theta}), \gamma_2(\tilde{\theta})\big),
\qquad
  n(\tilde{\theta})=\big(n_1(\tilde{\theta}), n_2(\tilde{\theta})\big).
\end{equation*}
The relations
\begin{align*}
|\gamma_1'(\tilde{\theta})|^2\,+\,|\gamma_2'(\tilde{\theta})|^2=1,
\qquad
|n_1(\tilde{\theta})|^2+|n_2(\tilde{\theta})|^2=1,
\qquad
\gamma_1'(\tilde{\theta})n_1(\tilde{\theta})+ \gamma_2'(\tilde{\theta})n_2(\tilde{\theta})=0,
\end{align*}
will give that
\begin{align*}
\gamma_1'(\tilde{\theta})n_2(\tilde{\theta})-\gamma_2'(\tilde{\theta})n_1(\tilde{\theta})\,=\,1,
\end{align*}
provided that the sign is taken as '+' by suitable choice of the natural parameter of $\Gamma$.
Then, the curve ${\mathcal C}_1$ can be expressed in the following form
\begin{equation*}
\begin{split}
  {\mathcal C}_1:~{\mathbb H}\big({\tilde t},{\tilde \varphi}_1({\tilde t})\big)
\,
%=\,\g\big({\tilde \varphi}_1({\tilde t})\big)+{\tilde t}\,n\big({\tilde \varphi}_1({\tilde t})\big)
&=\Big(\gamma_1\big({\tilde \varphi}_1({\tilde t})\big)+\tilde{t}n_1\big({\tilde \varphi}_1({\tilde t})\big),\,\, \gamma_2\big({\tilde \varphi}_1({\tilde t})\big)+\tilde{t}n_2\big({\tilde \varphi}_1({\tilde t})\big)\Big)
\\
&\equiv\,\big( y_1(\tilde{t})\,,\,y_2(\tilde{t})\big).
\end{split}
\end{equation*}
The calculations
\begin{align*}
  y_1'({\tilde t})=&\,\gamma_1'({\tilde\varphi}_1)\cdot \frac {{\mathrm d}{{\tilde\varphi}_1}}{{\mathrm d} {\tilde t}}+n_1({\tilde\varphi}_1)+\tilde{t}\cdot n_1'({\tilde\varphi}_1)\cdot \frac {{\mathrm d} {{\tilde\varphi}_1}}{{\mathrm d} {\tilde t}},
\end{align*}
\begin{align*}
y_2'({\tilde t})=&\,\gamma_2'({\tilde\varphi}_1)\cdot \frac {{\mathrm d}{{\tilde\varphi}_1}}{{\mathrm d} {\tilde t}}+n_2({\tilde\varphi}_1)+\tilde{t}\cdot n_2'({\tilde\varphi}_1)\cdot \frac {{\mathrm d} {{\tilde\varphi}_1}}{{\mathrm d} {\tilde t}},
\end{align*}
\begin{align*}
y_1''({\tilde t})=&\,\gamma_1''({\tilde\varphi}_1)\cdot \Big(\frac {{\mathrm d} {{\tilde\varphi}_1}}{{\mathrm d} {\tilde t}}\Big)^2+\gamma_1'({\tilde\varphi}_1)\cdot \frac {{\mathrm d}^2 {{\tilde\varphi}_1}}{{\mathrm d} {\tilde t}^2}+2\,n_1'({\tilde\varphi}_1)\cdot\frac {{\mathrm d} {{\tilde\varphi}_1}}{{\mathrm d} {\tilde t}}
 \\
 &+\tilde{t}\,n_1''({\tilde\varphi}_1)\cdot\Big(\frac {{\mathrm d} {{\tilde\varphi}_1}}{{\mathrm d} {\tilde t}}\Big)^2+\tilde{t}\,n_1'({\tilde\varphi}_1)\cdot\frac {{\mathrm d}^2 {{\tilde\varphi}_1}}{{\mathrm d} {\tilde t}^2},
 \end{align*}
 \begin{align*}
 y_2''({\tilde t})=&\,\gamma_2''({\tilde\varphi}_1)\cdot \Big(\frac {{\mathrm d} {{\tilde\varphi}_1}}{{\mathrm d} {\tilde t}}\Big)^2+\gamma_2'({\tilde\varphi}_1)\cdot \frac {{\mathrm d}^2 {{\tilde\varphi}_1}}{{\mathrm d} {\tilde t}^2}+2\,n_2'({\tilde\varphi}_1)\cdot\frac {{\mathrm d} {{\tilde\varphi}_1}}{{\mathrm d} {\tilde t}}
 \\
 &+\tilde{t}\,n_2''({\tilde\varphi}_1)\cdot\Big(\frac {{\mathrm d} {{\tilde\varphi}_1}}{{\mathrm d} {\tilde t}}\Big)^2+\tilde{t}\,n_2'({\tilde\varphi}_1)\cdot\frac {{\mathrm d}^2 {{\tilde\varphi}_1}}{{\mathrm d} {\tilde t}^2},
\end{align*}
imply that
 \begin{align*}
  (y_1'(t))^2+(y_2'(t))^2\Big|_{\tilde{t}=0}\,=\,|n(0)|^2\,=\,1,
 \end{align*}
and
\begin{align*}
 y_1'(\tilde{t})\,y_2''(\tilde{t})-y_1''(\tilde{t})\,y_2'(\tilde{t})\Big|_{\tilde{t}=0}\,
 =\,\big(n_1(0)\,\gamma_2'(0)\,-\,n_2(0)\,\gamma_1'(0)\big)\,{\tilde\varphi}_1''(0)\,=\,{\tilde\varphi}_1''(0).
\end{align*}
Therefore, the signed curvature of the curve $\mathcal{C}_1$ at the point $P_1$ is
\begin{equation*}
  k_1=\frac{\, y_1'(0)y_2''(0)-y_1''(0)y_2'(0)\,}{\big((y_1'(0))^2+(y_2'(0))^2\big)^{\frac{3}{2}}}
  =\t{\varphi}_1''(0).
\end{equation*}
Similarly, we can show $ k_2={\tilde \varphi}_2''(0)$.
\qed
\\

\medskip
\noindent
{\textbf{Proof of Lemma \ref{derivativeofF}:} } Indeed,  first comes the derivative of first order
\begin{align*}
\begin{aligned}
\frac{{\partial} F}{{\partial} t}(0,\theta)&\,=\,\Big[\gamma{\,'}(\Theta)\cdot\Theta_t
\,+\,n(\Theta)\,+\,tn'(\Theta)\cdot {\Theta}_t\Big]\Big|_{(0,\theta)}
\\
&\,=\,\gamma{\,'}(\Theta)\cdot 0\,+\,n(\Theta(0,\theta))
\\
&\,=\,n(\theta).
\end{aligned}
\end{align*}

So as $t$ is small enough, $\frac{{\partial} F}{{\partial} t}\neq 0$.
Consider the derivative of second order
\begin{equation*}
\begin{split}
q_1(\theta)&\,\equiv\, \frac{{\partial}^2F}{{\partial} t^2}(0,\theta)
\\
&\,=\,\lf[\gamma''(\Theta)\cdot (\Theta_t)^2\,+\,\gamma'(\Theta)\cdot \Theta_{tt}\,+\,2n'(\Theta)\cdot \Theta_{t}\,+\,tn''(\Theta)\cdot(\Theta_t)^2
\,+\,tn'(\Theta)\cdot\Theta_{tt}\ri]\Big|_{(0,\theta)}
\\
&\,=\,\gamma'(\theta)\cdot \Theta_{tt}(0,\theta)\;\;{ \bot n(\gamma(\theta))}.
\end{split}
\end{equation*}
Moreover, there hold
\begin{align*}
\begin{aligned}
q_1'(\theta)\,=\,&\Big[\gamma''(\Theta)\cdot\Theta_{\theta}\cdot\Theta_{tt}\,+\,\gamma'(\Theta)
\cdot\Theta_{tt\theta}|_{(0,\theta)}\Big]|_{(0,\theta)}
\\
\,=\,&\gamma''(\Theta)\cdot\Theta_{tt}|_{(0,\theta)}\,+\,\gamma'(\Theta)\Theta_{tt\theta}|_{(0,\theta)},
\end{aligned}
\end{align*}
and specially,
\begin{align*}
q_1'(0)\,=\,\gamma{''}(0){k}_1\,+\,\gamma{\,'}(0)(k_2-k_1),
\qquad
q_1'(1)\,=\,\gamma{''}(1){k}_2\,+\,\gamma{\,'}(1)(k_2-k_1).
\end{align*}

\medskip
The derivative of third order is
\begin{equation*}
\begin{split}
q_2(\theta)&\,\equiv\, \frac{{\partial}^3F}{{\partial} t^3}(0,\theta)
\\
&\,=\,\Big[
\gamma{'''}(\Theta)\cdot (\Theta_t)^3
+3\gamma{''}(\Theta)\cdot \Theta_t\cdot \Theta_{tt}
+\gamma{\,'}(\Theta)\cdot \Theta_{ttt}
+3n{''}(\Theta)\cdot (\Theta_t)^2
\\
&\qquad
+3n{'}(\Theta)\cdot \Theta_{tt}
\Big]\Big|_{(0,\theta)}
\\
&\,=\,\gamma{\,'}(\theta)\cdot \Theta_{ttt}(0,\theta)+3n{'}(\theta)\cdot \Theta_{tt}(0,\theta)
\qquad
 \bot n(\theta).
 \end{split}
\end{equation*}
This finishes the proof of the lemma.
\qed

\medskip
As a conclusion, as $t$ is small enough, there hold the asymptotic behavior
\begin{align}\label{expansionofFermi}
F(t,\theta)\,=\,\gamma(\theta)
+tn(\theta)+\frac {t^2}{2}q_1(\theta)+\frac {t^3}{6}q_2(\theta)+O(t^4),~~\forall\, \theta\in [0,1], ~t\in (-\delta_0,\delta_0),
\end{align}
where $\delta_0>0$ is a small constant.  This gives us that
$$\frac{{\partial} F}{{\partial} t}(t,\theta)\,=\,n(\theta)+t q_1(\theta)+\frac{t^2}{2}q_2(\theta)+O(t^3),$$
$$\frac{{\partial} F}{{\partial} \theta}(t,\theta)\,=\,\gamma'(\theta)-k(\theta)t\gamma{\,'}(\theta)+\frac{t^2}{2}q_1'(\theta)+\frac{t^3}{6}q_2'(\theta)+O(t^4).$$
Specially, there hold
$$\frac{{\partial} F}{{\partial} t}(t,0)\,=\,n(0)+t q_1(0)+\frac{t^2}{2}q_2(0)+O(t^3),$$
$$\frac{{\partial} F}{{\partial} t}(t,1)\,=\,n(1)+t q_1(1)+\frac{t^2}{2}q_2(1)+O(t^3),$$
$$\frac{{\partial} F}{{\partial} \theta}(t,0)\,=\,\gamma'(0)-k(0)t\gamma{\,'}(0)+\frac{t^2}{2}q_1'(0)+\frac{t^3}{6}q_2'(0)+O(t^4),$$
$$\frac{{\partial} F}{{\partial} \theta}(t,1)\,=\,\gamma'(1)-k(1)t\gamma{\,'}(1)+\frac{t^2}{2}q_1'(1)+\frac{t^3}{6}q_2'(1)+O(t^4).$$
These formulas will play an important role in the derivation of the local form of (\ref{original equation-01}), which will be given in Appendix \ref{appendixB}.

\medskip
\section{Local forms of the differential operators in (\ref{original equation-01})}\label{appendixB}
\setcounter{equation}{0}

 In this section, we are devoted to presenting the expressions of the differential operators  $\Delta$ and $\partial/\partial \nu$
in (\ref{original equation-01}).
The formulas will be provided in the local forms in the modified Fermi coordinates given in Section \ref{section2}.
%Their rescaling forms in (\ref{problemafterscaling}) will be also provided.

%\subsection{Differential operators expressed in modified Fermi coordinates}
\medskip
We first derive the metric matrix.
Note that
$$\gamma'(\theta)\perp n(\theta),
\quad
n'(\theta)\bot n(\theta),
\quad
q_1(\theta)\bot n(\theta),
\quad
q_2(\theta)\bot n(\theta).$$
As the first step, here are the computation of the metric matrix:
\begin{equation*}
\begin{split}
g_{11}&\,=\,\lf<\frac{{\partial} F}{{\partial} t},\, \frac{{\partial} F}{{\partial} t}\ri>
\\
&\,=\,\lf<n+tq_1+\frac{t^2}{2}q_2+O(t^3),\,\, n+tq_1+\frac{t^2}{2}q_2+O(t^3)\ri>
\\
&\,=\,1+t^2|q_1|^2+O(t^3),
\end{split}
\end{equation*}
and
\begin{equation*}
\begin{split}
g_{12}&\,=\,\lf<\frac{{\partial} F}{{\partial} t},\,\,\frac{{\partial} F}{{\partial} \theta}\ri>
\\
&\,=\,\lf<n+tq_1+\frac{t^2}{2}q_2+O(t^3),\,\,\gamma'-kt\gamma{'}+\frac{t^2}{2}q_1'+O(t^3)\ri>
\\
&\,=\,\frac{t^2}{2}<n,q_1'>+t<q_1, \gamma'>-t^2k<q_1, \gamma'>+\frac{t^2}{2}<q_2,\gamma'>+O(t^3)
\\
&\,=\,t<q_1, \gamma'>+\frac{t^2}{2}\big(<q_2,\gamma'>-k<q_1, \gamma'>\big)+O(t^3),
\end{split}
\end{equation*}
where we have used the fact
\begin{align*}
<q_1', n>-k<q_1, \gamma'>
\,=\,<q_1', n>+<q_1, n'>
\,=\,\frac{\partial}{\partial\theta}<q_1, n>
\,=\,0.
\end{align*}
The last element is
\begin{equation*}
\begin{split}
g_{22}&\,=\,\lf<\frac{{\partial} F}{{\partial} \theta},\,\,\frac{{\partial} F}{{\partial} \theta}\ri>
\\
&\,=\,\lf<\gamma'-kt\gamma{'}+\frac{t^2}{2}q_1'+O(t^3),\,\,\gamma'-kt\gamma{'}+\frac{t^2}{2}q_1'+O(t^3)\ri>
\\
&\,=\,1-2kt+t^2\lf(<q_1', \gamma'>+k^2\ri)+O(t^3).
\end{split}
\end{equation*}
So the  determinant  of the metric matrix is
\begin{equation*}
\begin{split}
g&\,=\,\mbox{det}(g_{ij})
\\
&\,=\,1-2kt+t^2\lf(<q_1', \gamma'>+k^2\ri)+t^2|q_1|^2-t^2<q_1, \gamma'>^2+O(t^3)
\\
&\,=\,1-2kt+t^2\lf(<q_1', \gamma'>+k^2\ri)+O(t^3),
\end{split}
\end{equation*}
where we have used $|q_1|^2-<q_1, \gamma'>^2=0$ due to the expression of $q_1$ in Lemma \ref{derivativeofF}.

\medskip
We turn to the computation of the inverse of the metric matrix.
By
$$
\big[1+at+bt^2+O(t^3)\big]^{-1}\,=\,1-at+(a^2-b)t^2+O(t^3),
$$
we have
\begin{align*}
g^{-1}\,\,=\,&\,1+2kt+t^2\lf(3k^2-<q_1', \gamma'>\ri)+O(t^3).
\end{align*}
Whence
\begin{equation*}
\begin{split}
g^{11}&\,=\,g_{22}\cdot g^{-1}
\\
&\,=\,1-2kt+t^2\lf(<q_1', \gamma'>+k^2\ri)+2kt-4k^2t^2
+t^2\lf(3k^2-<q_1', \gamma'>\ri)+O(t^3)
\\
&\,=\,1+O(t^3),
\\
\\
\end{split}
\end{equation*}
\begin{equation*}
\begin{split}
-g^{12}&\,=\,g_{12}\cdot g^{-1}
\\
&\,=\,t<q_1, \gamma'>+2kt^2<q_1, \gamma'>+\frac{t^2}{2}\lf(<q_2',\gamma{'}>-k<q_1, \gamma'>\ri)+O(t^3)\qquad\qquad\qquad
\\
&\,=\,t<q_1, \gamma'>+\frac{t^2}{2}\big(3k<q_1, \gamma'>+<q_2,\gamma'>\big)+O(t^3),
\\
\\
\end{split}
\end{equation*}

\noindent
\begin{equation*}
\begin{split}
g^{22}&\,=\,g_{11}\cdot g^{-1}
\,=\,1+2kt+t^2\lf(3k^2-<q_1', \gamma'>+|q_1|^2\ri)+O(t^3).\qquad\qquad\qquad\qquad
\end{split}
\end{equation*}

\medskip
%Note that if the function has the following asymptotic expansion
%$$f(s)\,=\,1+as+bs^2+O(s^3),\quad f(0)\,=\,1,
%$$
%then for $s$ close to zero
%$$
%\sqrt{f(s)}\,=\,1+\frac{a}{2}s+\frac12\lf(b-\frac14a^2\ri)s^2+O(s^3),
%$$
%$$
%\log f\,=\,as+(b-\frac12a^2)s^2+O(s^3).
%$$
We can get the following formulas
\begin{equation*}
\begin{split}
\sqrt{g}&\,=\,1-kt+\frac12t^2<q_1', \gamma'>+O(t^3),\qquad\qquad\qquad\qquad\qquad
\end{split}
\end{equation*}

\begin{equation*}
\begin{split}
(\sqrt{g})^{ -1}&\,=\,1+kt+\frac12t^2\lf(2k^2-<q_1', \gamma'>\ri)+O(t^3),\qquad\qquad\qquad\qquad
\end{split}
\end{equation*}

\begin{equation*}
\begin{split}
\sqrt{g^{22}}&\,=\,1+kt+\frac12t^2\Big(3k^2-<q_1', \gamma'>+|q_1|^2-\frac14\cdot 4k^2\Big)+O(t^3)
\\
&\,=\,1+kt+\frac12\big(2k^2-<q_1', \gamma'>+|q_1|^2\big)t^2+O(t^3),
\end{split}
\end{equation*}

\begin{equation*}
\begin{split}
\frac1{\sqrt{g^{22}}}&\,=\,1-kt+\Big(k^2-k^2+\frac12<q_1', \gamma'>-\frac12 |q_1|^2\Big)t^2+O(t^3)\qquad
\\
&\,=\,1-kt+\frac12\big(<q_1', \gamma'>-|q_1|^2\big)t^2+O(t^3).
\end{split}
\end{equation*}

\medskip
Now, recall the definition of Laplace-Beltrami operator in the form
\begin{equation}
\begin{split}\label{laplace0}
\triangle_{t,\theta}u\,=\,&\,g^{11}\frac{{\partial}^2 u}{{\partial} t^2}+2g^{12}\frac{{\partial}^2 u}{{\partial} t{\partial} \theta}+g^{22}\frac{{\partial}^2 u}{{\partial} \theta^2}
+(\sqrt{g})^{ -1}\left[\frac{{\partial}}{{\partial} t}(\sqrt{g}g^{11})+\frac{{\partial}}{{\partial} \theta}(\sqrt{g}g^{21})\right]\frac{{\partial} u}{{\partial} t}
\\
&+(\sqrt{g})^{ -1}\left[\frac{{\partial}}{{\partial} t}(\sqrt{g}g^{12})+\frac{{\partial}}{{\partial} \theta}(\sqrt{g}g^{22})\right]\frac{{\partial} u}{{\partial} \theta}.
\end{split}
\end{equation}
The last two coefficients in (\ref{laplace0}) can be computed as follows.
Since
\begin{equation*}
\begin{split}
\sqrt{g}\cdot g^{11}
&\,=\,1-kt+\frac{t^2}{2}<q_1', \gamma'>+O(t^3),
\end{split}
\end{equation*}we can get
\begin{equation*}
\begin{split}
\frac{{\partial}}{{\partial} t}(\sqrt{g}g^{11})&\,=\,-k+t<q_1', \gamma'>+O(t^2).
\end{split}
\end{equation*}\noindent
Since
\begin{equation*}
\begin{split}
\sqrt{g}\cdot g^{12}&\,=\,-t<q_1, \gamma'>-\frac{t^2}2(3k<q_1, \gamma'>+<\gamma', q_2>)+kt^2<q_1, \gamma'>+O(t^3)
\\
&\,=\,-t<q_1, \gamma'>-\frac{t^2}2\lf(k<q_1, \gamma'>+<\gamma', q_2 >\ri)+O(t^3),
\end{split}
\end{equation*}
we can get
\begin{equation*}
\begin{split}
\frac{{\partial}}{{\partial} \theta}(\sqrt{g}\cdot g^{12})&\,=\,-t<q_1, \gamma'>'-\frac{t^2}2\lf(k<q_1, \gamma'>+<\gamma', q_2>\ri)'+O(t^3),
\\
\frac{{\partial}}{{\partial} t}(\sqrt{g}\cdot g^{12})&\,=\,-<q_1, \gamma'>-t\lf(k<q_1, \gamma'>+<\gamma',q_2>\ri)+O(t^2).
\end{split}
\end{equation*}
So
\begin{equation*}
\begin{split}
\frac{{\partial}}{{\partial} t}(\sqrt{g}\cdot g^{11})+\frac{{\partial}}{{\partial} \theta}(\sqrt{g}\cdot g^{12})
&\,=\,-k+t\lf(<q_1', \gamma'>-<q_1, \gamma'>'\ri)+O(t^2)
\\
&\,=\,-k+O(t^2),
\end{split}
\end{equation*}
where we have used the relations
$$
<q_1', \gamma'>-<q_1, \gamma'>'\,=\,-<q_1, \gamma''>\,=\,0.
$$
We obtain the coefficient of the first bracket in \eqref{laplace0} is
\begin{equation*}
\begin{split}
&(\sqrt{g})^{ -1}\lf(\frac{{\partial}}{{\partial} t}(\sqrt{g}\cdot g^{11})+\frac{{\partial}}{{\partial} \theta}(\sqrt{g}\cdot g^{12})\ri)
\,=\,-k-k^2t+O(t^2).
\end{split}
\end{equation*}

\medskip
On the other hand, by
\begin{equation*}
\begin{split}
\sqrt{g}\cdot g^{22}&\,=\,1-kt+\frac{t^2}{2}<q_1, \gamma'>+2kt-2k^2t^2
+t^2\big(3k^2-<q_1, \gamma'>+|q_1|^2\big)+O(t^3)
\\
&\,=\,1+kt+\frac{t^2}2\big(2k^2-<q_1, \gamma'>+2|q_1|^2\big)+O(t^3),
\end{split}
\end{equation*}we can get
\begin{equation*}
\begin{split}
\frac{{\partial}}{{\partial} \theta}(\sqrt{g}\cdot g^{22})&\,=\,k't+\frac{t^2}{2}\big(2k^2-<q_1, \gamma'>+2|q_1|^2\big)'+O(t^3).
\end{split}
\end{equation*}
This implies that
\begin{equation*}
\begin{split}
&\frac{{\partial}}{{\partial} t}(\sqrt{g}\cdot g^{12})+\frac{{\partial}}{{\partial} \theta}(\sqrt{g}\cdot g^{22})
\,=\,-<q_1, \gamma'>+t(k'-k<q_1, \gamma'>-<q_2,\gamma{\,'}>)+O(t^2),
\end{split}
\end{equation*}
which gives the coefficient of the second bracket in \eqref{laplace0} is
\begin{equation*}
\begin{split}
&(\sqrt{g})^{ -1}\lf(\frac{{\partial}}{{\partial} t}(\sqrt{g}\cdot g^{12})+\frac{{\partial}}{{\partial} \theta}(\sqrt{g}\cdot g^{22})\ri)
%\\
%&\,=\,-<q_1, \gamma'>+t\big(k{\,'}-k<q_1, \gamma'>-<q_2,\gamma{\,'}>\big)-kt<q_1, \gamma'>+O(t^2)
\\
&\,=\,-<q_1, \gamma'>+t\big(k'-2k<q_1, \gamma'>-<q_2,\gamma'>\big)+O(t^2).
\end{split}
\end{equation*}

\medskip
Denote
\begin{align}\label{a0}
 \varpi(\theta)\,\equiv\,<q_1(\theta), \gamma'(\theta)>\,=\,\Theta_{tt}(0, \theta),
\end{align}
where we have used the expression of $q_1$ in Lemma \ref{derivativeofF}.
Hence, the term $\triangle_y u$ in (\ref{original equation-01}) has the following form in the modified Fermi coordinate system
\begin{equation}\label{a1a5}
\begin{split}
\triangle_y u&\,=\,\lf[1+O(t^3)\ri]u_{tt} -2\lf[\,\varpi\,t+\frac{t^2}2(3k\varpi+<q_2,\gamma'>)+O(t^3)\ri]u_{t\theta}
\\
&\quad +\lf[1+2kt+t^2\big(3k^2-<q_1',\gamma'>+|q_1|^2\big)+O(t^3)\ri]u_{\theta\theta}
\\
&\quad +\lf[-k-k^2t+O(t^2)\ri]u_t +\lf[-\varpi+tk'-t\big(2ka_0+<q_2,\gamma'>\big)+O(t^2)\ri]u_\theta
\\
&\,\equiv\,\lf(1+a_2t^3\ri)u_{tt}
\,+\,\big(-2\,\varpi\,t+a_3t^2\big)u_{t\theta}
\,+\,\big(1+2kt+a_1t^2\big)u_{\theta\theta}
\\
&\quad
\,+\,\big(-k-k^2t+a_4t^2\big)u_t
\,+\,\big(-\varpi+a_5t\big)u_\theta.
\end{split}
\end{equation}
The terms in $\triangle_y u$ will be rearranged in the form
\begin{eqnarray}\label{laplaceorigin}
\triangle_y u\,\,=\,\,u_{tt}+u_{\theta\theta}+\bar{B}_1(u)+\bar{B}_0(u),
\end{eqnarray}
where
\begin{equation}\label{B1bar}
\begin{split}
\bar{B}_1(u)&\,=\,-(k+k^2t)u_{t} -2\,\varpi\,tu_{t\theta} -\varpi u_\theta,
\end{split}
\end{equation}
and
\begin{equation}
\begin{split}\label{B0bar}
\bar{B}_0(u)
&\,=\,\,2ktu_{\theta\theta}+a_1t^2u_{\theta\theta}+a_2t^3u_{tt}
+a_3t^2u_{t\theta}+a_4t^2u_t
+a_5tu_\theta.
\end{split}
\end{equation}

\medskip
We finally show the expression of $\nu$, the outward unit normal of $\partial\Omega$ near $P_1, P_2$, i.e., when $\theta=0,1$
in the modified Fermi coordinates. This will provide the local expression of $\partial u/\partial\nu$ in (\ref{original equation-01}).
Suppose
$$
\nu\,=\,\sigma_1\frac{{\partial} F}{{\partial} t}+\sigma_2\frac{{\partial} F}{{\partial} \theta}.
$$
Since ${{\partial} F}/{{\partial} t}\in T({\partial} \Omega)$, we have $<{{\partial} F}/{{\partial} t},\nu>=0$.
Hence
$$
\sigma_2\neq 0,\quad \frac{{\partial} F}{{\partial} \theta}\neq 0,
$$
and
$$\sigma_1g_{11}+\sigma_2g_{12}\,=\,0.$$
On the other hand, $<\nu,\nu>=1$, that is
$$
\lf<\sigma_1\frac{{\partial} F}{{\partial} t}+\sigma_2\frac{{\partial} F}{{\partial} \theta},\, \sigma_1\frac{{\partial} F}{{\partial} t}+\sigma_2\frac{{\partial} F}{{\partial} \theta}\ri>\,=\,1,
$$
which implies that
$$
\sigma_1^2g_{11}+\sigma_2^2g_{22}+2\sigma_1\sigma_2g_{12}\,=\,1.
$$
Combining above two equations, one can get
$$\sigma_1\,=\,\pm\frac{g^{12}}{\sqrt{g^{22}}},
\qquad
\sigma_2\,=\,\pm\sqrt{g^{22}}.$$

\medskip
Recall that
$$
q_1(\theta)\,=\,\gamma{\,'}(\theta)\big[(k_2-k_1)\theta+k_1\big];
$$
$$
q_1'(\theta)\,=\,\gamma{''}(\theta)\big[(k_2-k_1)\theta+k_1\big]\,+\,\gamma{\,'}(\theta)(k_2-k_1);
$$
$$
q_2(\theta)\,=\,\gamma{\,'}(\theta)\Big[\big(\tilde{\varphi}_2'''(0)-\tilde{\varphi}_1'''(0)\big)\theta +\tilde{\varphi}_1'''(0)\Big]
\,-\,3k\gamma{\,'}(\theta)\big[(k_2-k_1)\theta+k_1\big].
$$
By choosing the $+$, it is easy to check that
\begin{equation*}
\begin{split}
\sigma_2  \,=\,&1+kt+\frac12\Big(2k^2-<q_1',\gamma{\,'}>+|q_1|^2\Big)t^2+O(t^3),
\\
\,=\,&\,1+kt+k^2t^2-\frac 12(k_2-k_1)t^2
+\frac 12\big[(k_2-k_1)\theta+k_1\big]^2t^2+O(t^3).
\end{split}
\end{equation*}
Similarly, we get
\begin{equation*}
\begin{split}
\sigma_1\,=\,&\,-\sigma_2\cdot g_{12}\cdot (g_{11})^{-1}
\\
\,=\,&-\Big \{1+kt+k^2t^2-\frac 12(k_2-k_1)t^2
+\frac 12\big[(k_2-k_1)\theta+k_1\big]^2t^2\Big \}
\\
&\quad \times \Big [\,\varpi\,t+\frac 12 (<q_2,\gamma'>-ka_0)\Big]+O(t^3)
\\
\,=\,&\,-\,\big[(k_2-k_1)\theta+k_1\big]t
\,+\,k\big[(k_2-k_1)\theta+k_1\big]t^2
-\frac {1}{2}\lf[(\tilde{\varphi}_2'''(0)-\tilde{\varphi}_1'''(0))\theta+\tilde{\varphi}_1'''(0)\ri]t^2
\,+\,O(t^3).
\end{split}
\end{equation*}

\medskip
In the modified Fermi coordinates $(t, \theta)$ in (\ref{Fermicoordinates}), the normal derivative
$\partial u/\partial\nu$ has a local form as follows
$$
\sigma_1\frac{\partial u}{\partial t} + \sigma_2\frac{\partial u}{\partial\theta}.
$$
More precisely, for $\theta\,=\,0$, it is
\begin{align}
\begin{aligned}\label{boundaryoriginal0}
&k_1tu_t
\,+\,
b_1 t^2u_t
\,-\,u_{\theta}
\,-\,k(0)tu_{\theta}
\,+\,b_2t^2u_{\theta}
\,+\,{\bar D}_0^0(u),
\end{aligned}
\end{align}
where
\begin{align}
{\bar D}_0^0(u)\,=\,\sigma_3(t)\,u_t\,+\,\sigma_4(t)u_\theta,
\end{align} 
and the constants $b_1$ and $b_2$ are given by
\begin{align}\label{b1b2}
b_1\,=\,\frac{1}{2}\tilde{\varphi}_1'''(0)-k(0)k_1,
\qquad
b_2\,=\,\frac{1}{2}(k_2-k_1)-k^2(0)-\frac {1}{2}k_1^2.
\end{align}
On the other hand, for $\theta=1$, it has the form
\begin{align}
\begin{aligned}\label{boundaryoriginal1}
&k_2tu_t
\,+\,
b_3t^2u_t
\,-\,u_{\theta}
\,-\,k(1)tu_{\theta}
\,+\,b_4 t^2u_{\theta}
\,+\,{\bar D}_0^1(u),
\end{aligned}
\end{align}
with the notation
\begin{align}
{\bar D}_0^1(u)\,=\,\sigma_5(t)\,u_t\,+\,\sigma_6(t)u_\theta,
\end{align}
\begin{align}\label{b3b4}
b_3\,=\,\frac{1}{2}\tilde{\varphi}_2'''(0)-k(1)k_2,
\qquad
b_4\,=\,\frac{1}{2}(k_2-k_1)-k^2(1)-\frac {1}{2}k_2^2.
\end{align}
In the above, the functions $\sigma_3, \cdots, \sigma_6$ are smooth functions of $t$ with the properties
$$
|\sigma_i(t)|\leq C|t|^3,\quad i=3, 4, 5, 6.
$$

\end{appendices}

\end{document}